\documentclass{birkmult}
\usepackage{amsthm,amsmath,amsfonts,latexsym,amssymb,mathrsfs}

\renewcommand{\thefootnote}{\fnsymbol{footnote}}
\newtheorem{theorem}{Theorem}[section]
\newtheorem{lemma}[theorem]{Lemma}
\newtheorem{proposition}[theorem]{Proposition}
\newtheorem{cor}[theorem]{Corollary}

\theoremstyle{definition}
\newtheorem{definition}[theorem]{Definition}
\newtheorem{example}[theorem]{Example}

\theoremstyle{remark}
\newtheorem{remark}[theorem]{Remark}

\numberwithin{subsection}{section} \numberwithin{equation}{section}

\begin{document}

\title[Block Toeplitz Operators]
  {\LARGE Hyponormality and Subnormality of Block Toeplitz Operators}

\author{Ra{\'u}l\ E.\ Curto}
\address{Department of Mathematics, University of Iowa, Iowa City, IA 52242, U.S.A.}
\email{raul-curto@uiowa.edu}

\author{In Sung Hwang}
\address{Department of Mathematics, Sungkyunkwan University, Suwon 440-746, Korea}
\email{ihwang@skku.edu}

\author{Woo Young Lee}
\address{Department of Mathematics, Seoul National University, Seoul 151-742, Korea}
\email{wylee@snu.ac.kr}

\renewcommand{\thefootnote}{}

\thanks{The work of the first named author was partially supported
by NSF Grant DMS-0801168. \ The work of the second author was
supported by NRF grant funded by the Korea
government(MEST)(2011-0003273). \ The work of the third author was
supported by Basic Science Research Program through the NRF grant
funded by the Korea government(MEST)(2011-0001250)¡±.}


\begin{abstract}
In this paper we are concerned with hyponormality and subnormality
of block Toeplitz operators acting on the vector-valued Hardy space
$H^2_{\mathbb{C}^n}$ of the unit circle.

Firstly, we establish a tractable and explicit criterion on the
hyponormality of block Toeplitz operators having bounded type
symbols via the triangularization theorem for compressions of the
shift operator.

Secondly, we consider the gap between hyponormality and subnormality
for block Toeplitz operators. \ This is closely related to Halmos's
Problem 5: Is every subnormal Toeplitz operator either normal or
analytic\,? \ We show that if $\Phi$ is a matrix-valued rational
function whose co-analytic part has a coprime factorization then
every hyponormal Toeplitz operator $T_{\Phi}$ whose square is also
hyponormal must be either normal or analytic.

Thirdly, using the subnormal theory of block Toeplitz operators, we
give an answer to the following ``Toeplitz completion" problem: Find
the unspecified Toeplitz entries of the partial block Toeplitz
matrix
$$
A:=\left[\begin{matrix} U^*& ?\\
?&U^* \end{matrix}\right]
$$
so that $A$ becomes subnormal, where $U$ is the unilateral shift on
$H^2$.
\end{abstract}

\maketitle



%
%
%
%


\section{Introduction}

Toeplitz operators,
block Toeplitz operators and (block) Toeplitz determinants (i.e.,
determinants of sections of (block) Toeplitz operators)
arise naturally in several fields of mathematics
and in a variety of problems in physics, especially, in quantum mechanics. \
For example, the spectral theory of Toeplitz operators plays an important role
in the study of solvable models in quantum mechanics (\cite{Pr})
and in the study the one-dimensional Heisenberg Hamiltonian of
ferromagnetism (\cite{DMA}); \
the theory of block Toeplitz determinants is used in the study
of the classical dimer model (\cite{BE}) and in the study
of the vicious walker model (\cite{HI}); \
the theory of block Toeplitz operators is also used in the study
of Gelfand-Dickey Hierarchies (cf. \cite{Ca}). \
On the other hand, the theory of hyponormal
and subnormal operators is an extensive and highly developed area,
which has made important contributions to a number of problems in
functional analysis, operator theory, and mathematical physics
(see, for example, \cite{If}, \cite{HS}, and \cite{Sz} for
applications to related mathematical physics problems). \
Thus, it becomes of central significance to describe in detail
hypormality and subnormality for Toeplitz operators. \  This paper
focuses primarily on hyponormality and subnormality of
block Toeplitz operators with rational symbols.
\ For the general theory of subnormal and hyponormal
operators, we refer to \cite{Con} and \cite{MP}.\

To describe our results, we first need to review a few essential
facts about (block) Toeplitz operators, and for that we will use
\cite{BS}, \cite{Do1}, \cite{Do2}, \cite{GGK}, \cite{MAR},
\cite{Ni}, and \cite{Pe}. \  Let $\mathcal{H}$ and $\mathcal{K}$ be
complex Hilbert spaces, let $\mathcal{B(H,K)}\label{B(H,K)}$ be the
set of bounded linear operators from $\mathcal{H}$ to $\mathcal{K}$,
and write $\mathcal{B(H)}\label{B(H)}:=\mathcal{B(H,H)}$. \  For
$A,B\in\mathcal{B(H)}$, we let $[A,B]\label{[A,B]}:=AB-BA$. \  An
operator $T\in\mathcal{B(H)}$ is said to be normal if
$[T^*,T]\label{[T^*,T]}=0$, hyponormal if $[T^*,T]\ge 0$, and
subnormal if $T$ has a normal extension, i.e.,
$T=N\vert_{\mathcal{H}}$, where $N$ is a normal operator on some
Hilbert space $\mathcal{K}\supseteq \mathcal{H}$ such that $\mathcal
H$ is invariant for $N$. \  For an operator $T\in \mathcal{B(H)}$,
we write $\hbox{ker}\,T\label{ker}$ and $\hbox{ran}\, T\label{ran}$
for the kernel and the range of $T$, respectively. \  For a set
$\mathcal M$, $\hbox{cl}\, \mathcal M\label{clM}$ and $\mathcal
M^\perp\label{Mperp}$ denote the closure and the orthogonal
complement of $\mathcal M$, respectively. \  Also, let
$\mathbb{T}\label{T}=\mathbb{R}/2\pi\mathbb{Z}$ be the unit circle.
\ Recall that the Hilbert space $L^2\equiv L^2({\mathbb
T})\label{L^2T}$ has a canonical orthonormal basis given by the
trigonometric functions $e_n(z)=z^n$, for all $n\in {\mathbb Z}$,
and that the Hardy space $H^2\equiv H^2({\mathbb T})\label{H^2T}$ is
the closed linear span of $\{e_n: n=0,1,\hdots \}$. \  An element
$f\in L^2$ is said to be analytic if $f\in H^2$. \ Let
$H^\infty\label{H^infty}\equiv H^\infty({\mathbb T}):=L^\infty \cap
H^2$, i.e., $H^\infty$ is the set of bounded analytic functions on
the open unit disk $\mathbb{D}$. \

Given a bounded measurable
function $\varphi\in L^\infty\label{L^infty}$, the {\it Toeplitz operator}
$T_\varphi\label{Tvarphi}$ and
 the {\it Hankel operator} $H_\varphi\label{Hvarphi}$ with {\it symbol} $\varphi$
on $H^2$ are defined by
\begin{equation}\label{1.1}
T_\varphi g:=P(\varphi g) \quad\hbox{and}\quad H_\varphi
g:=JP^\perp(\varphi g) \qquad (g\in H^2),
\end{equation}
where $P\label{P}$ and $P^\perp\label{Pperp}$ denote the orthogonal projections that map
from $L^2$ onto $H^2$ and $(H^2)^\perp$, respectively, and $J$
denotes the unitary operator from $L^2$ onto $L^2$ defined by
$J\label{J}(f)(z)=\overline z f(\overline z)$ for $f\in L^2$. \

To study hyponormality (resp. normality and subnormality) of the
Toeplitz operator $T_\varphi$ with symbol $\varphi$ we may, without
loss of generality, assume that $\varphi(0)=0$; this is because hyponormality
(resp. normality and subnormality) is invariant under translations by scalars. \
Normal Toeplitz
operators were characterized by a property of their symbols in the
early 1960's by A. Brown and P.R. Halmos \cite{BH} and the exact
nature of the relationship between the symbol $\varphi\in L^\infty$
and the hyponormality of $T_{\varphi}$ was understood via Cowen's
Theorem \cite{Co4} in 1988. \

\medskip

\noindent {\bf Cowen's Theorem.} (\cite{Co4},
\cite{NT})\label{cowen} {\it For each $\varphi\in L^\infty$, let
$$
\mathcal{E}(\varphi)\label{Ephi}\equiv \{k\in H^\infty:\ ||k||_\infty\le 1\
\hbox{and}\ \varphi-k\overline\varphi\in H^\infty\}. \
$$
Then a Toeplitz operator $T_\varphi$ is hyponormal if and only if
$\mathcal{E}(\varphi)$ is nonempty. \  }

\medskip

This elegant and useful theorem has been used in the
works \cite{CuL1}, \cite{CuL2}, \cite{FL}, \cite{Gu1}, \cite{Gu2},
\cite{GS}, \cite{HKL1}, \cite{HKL2}, \cite{HL1}, \cite{HL2},
\cite{HL3}, \cite{Le}, \cite{NT} and \cite{Zhu}, which have been
devoted to the study of hyponormality for Toeplitz operators on
$H^2$. \  Particular attention has been paid to Toeplitz operators
with polynomial symbols or rational symbols \cite{HL2}, \cite{HL3}.
\ However, the case of arbitrary symbol $\varphi$, though solved in
principle by Cowen's theorem, is in practice very complicated. \
Indeed, it may not even be possible to find tractable necessary and
sufficient condition for the hyponormality of $T_\varphi$ in terms
of the Fourier coefficients of the symbol $\varphi$ unless certain
assumptions are made about $\varphi$. \  To date, tractable criteria
for the cases of trigonometric polynomial symbols and rational
symbols were derived from a Carath\' eodory-Schur interpolation
problem (\cite{Zhu}) and a tangential Hermite-Fej\' er interpolation
problem (\cite{Gu1}) or the classical Hermite-Fej\' er interpolation
problem (\cite{HL3}), respectively. \

Recall that a function
$\varphi\in L^\infty$ is said to be of {\it bounded type} (or in the
Nevanlinna class) if there are analytic functions $\psi_1,\psi_2\in
H^\infty (\mathbb D)\label{H^inftyD}$ such that $\varphi=\psi_1/\psi_2$ almost
everywhere on $\mathbb T$. \  To date, no tractable criterion to
determine the hyponormality of $T_{\varphi}$ when the symbol $\varphi$ is of bounded type has been
found. \

We now introduce the notion of block Toeplitz operators. \  Let
$M_{n\times r}\label{M_nr}$ denote the set of all $n\times r$ complex matrices
and write $M_n\label{M_n}:=M_{n\times n}$. \  For $\mathcal X$ a Hilbert space,
let $L^2_{\mathcal X}\label{L2X}\equiv L^2_{\mathcal X}(\mathbb T)$ be the
Hilbert space of $\mathcal X$-valued norm square-integrable
measurable functions on $\mathbb{T}$ and let $H^2_{\mathcal X}\label{H2X}\equiv
H^2_{\mathcal X}(\mathbb T)$ be the corresponding Hardy space. \  We
observe that $L^2_{\mathbb{C}^n}= L^2\otimes \mathbb{C}^n$ and
$H^2_{\mathbb{C}^n}= H^2\otimes \mathbb{C}^n$. \  If $\Phi$ is a
matrix-valued function in $L^\infty_{M_n}\equiv
L^\infty_{M_n}(\mathbb T)$ ($=L^\infty\otimes M_n$) then $T_\Phi:
H^2_{\mathbb{C}^n}\to H^2_{\mathbb{C}^n}$ denotes the {\it block Toeplitz
operator} with {\it symbol} $\Phi$ defined by
$$
T_\Phi\label{T_Phi} f:=P_n\label{P_n}(\Phi f)\quad \hbox{for}\ f\in H^2_{\mathbb{C}^n},
$$
where $P_n$ is the orthogonal projection of $L^2_{\mathbb{C}^n}$
onto $H^2_{\mathbb{C}^n}$. \  A {\it block Hankel operator} with {\it symbol}
$\Phi\in L^\infty_{M_n}$ is the operator $H_\Phi:
H^2_{\mathbb{C}^n}\to H^2_{\mathbb{C}^n}$ defined by
$$
H_\Phi\label{H_Phi} f := J_n P_n^\perp (\Phi f)\quad \hbox{for}\ f\in
H^2_{\mathbb{C}^n},
$$
where $J_n$ denotes the unitary operator from $L^2_{\mathbb{C}^n}$
onto $L^2_{\mathbb{C}^n}$ given by $J_n(f)(z):=\overline{z} I_n
f(\overline{z})$ for $f \in L^2_{\mathbb{C}^n}$, with $I_n$ the
$n\times n$ identity matrix. \  If we set
$H^2_{\mathbb{C}^n}=H^2\oplus\cdots\oplus H^2$ then we see that
$$
T_\Phi=\begin{bmatrix} T_{\varphi_{11}}&\hdots&T_{\varphi_{1n}}\\
&\vdots\\
T_{\varphi_{n1}}&\hdots&T_{\varphi_{nn}}\end{bmatrix} \quad
\hbox{and}\quad
H_\Phi=\begin{bmatrix} H_{\varphi_{11}}&\hdots&H_{\varphi_{1n}}\\
&\vdots\\
H_{\varphi_{n1}}&\hdots&H_{\varphi_{nn}}\end{bmatrix},
$$
where
$$
\Phi=\begin{bmatrix} \varphi_{11}&\hdots&\varphi_{1n}\\
&\vdots\\
\varphi_{n1}&\hdots&\varphi_{nn}\end{bmatrix}\in L^\infty_{M_n}. \
$$
For $\Phi\in L^\infty_{M_{n\times m}}$, write
$$
\widetilde\Phi\label{widetildePhi} (z):=\Phi^*(\overline z). \
$$
A matrix-valued function $\Theta\label{Theta}\in H^\infty_{M_{n\times m}}$
($=H^\infty\otimes M_{n\times m}$) is called {\it inner} if
$\Theta^*\Theta=I_m$ almost everywhere on $\mathbb{T}$. \  The
following basic relations can be easily derived:
\begin{align}
&T_\Phi^*=T_{\Phi^*},\ \  H_\Phi^*= H_{\widetilde \Phi} \quad (\Phi\in
L^\infty_{M_n});\label{1.2}\\
&T_{\Phi\Psi}-T_\Phi T_\Psi = H_{\Phi^*}^*H_\Psi \quad
(\Phi,\Psi\in L^\infty_{M_n});\label{1.3}\\
&H_\Phi T_\Psi = H_{\Phi\Psi},\ \
H_{\Psi\Phi}=T_{\widetilde{\Psi}}^*H_\Phi \quad (\Phi\in
L^\infty_{M_n}, \Psi\in H^\infty_{M_n});\label{1.4}\\
&H_\Phi^* H_\Phi - H_{\Theta \Phi}^* H_{\Theta\Phi} =H_\Phi^*
H_{\Theta^*}H_{\Theta^*}^*H_\Phi \quad (\Theta\in H^\infty_{M_n}
 \ \hbox{inner,}  \ \Phi\in L^\infty_{M_n}).\label{1.5}
\end{align}

\medskip

For a matrix-valued function $\Phi \equiv [\varphi_{ij}]\in
L^\infty_{M_n}$, we say that $\Phi$ is of {\it bounded type} if each
entry $\varphi_{ij}$ is of bounded type, and we say that $\Phi$ is {\it
rational} if each entry $\varphi_{ij}$ is a rational function. \ A matrix-valued trigonometric
polynomial $\Phi \in L^{\infty}_{M_n}$ is of the form

\begin{equation*}
\Phi(z)=\sum_{j=-m}^{N} A_j z^j \; (A_j \in M_n),
\label{outer}
\end{equation*}
where $A_N$ and $A_{-m}$ are called the \textit{outer} coefficients of $\Phi$.

We recall that for matrix-valued functions
$A:=\sum_{j=-\infty}^\infty A_j z^j\in L^2_{M_n}$ and
$B:=\sum_{j=-\infty}^\infty B_j z^j\in L^2_{M_n}$, we define the inner
product of $A$ and $B$ by
$$
(A,B)\label{AB}:=\int_{\mathbb T} \hbox{tr}\,(B^*A)\,
d\mu=\sum_{j=-\infty}^\infty \hbox{tr}\,(B_j^*A_j)\,,
$$
where $\hbox{tr}\,(\cdot)\label{tr}$ denotes the trace of a matrix and define
$||A||_2\label{A_2}:=(A,A)^{\frac{1}{2}}$. \  We also define, for $A\in
L^\infty_{M_n}$,
$$
||A||_\infty\label{A_infty}:=\hbox{ess sup}_{t\in\mathbb T} ||A(t)||\quad \hbox{($||\cdot||$
denotes the spectral norm of a matrix)}. \
$$
Finally, the {\it shift} operator $S$ on $H^2_{\mathbb C^n}$ is defined by
$$
S\label{S}:=T_{zI_n}. \
$$

The following fundamental result will be useful in the sequel. \
\medskip

\noindent {\bf The Beurling-Lax-Halmos Theorem.}\label{beur} {\it A
nonzero subspace $M$ of $H^2_{\mathbb C^n}$ is invariant for the
shift operator $S$ on $H^2_{\mathbb C^n}$ if and only if $M=\Theta
H^2_{\mathbb C^m}$, where $\Theta$ is an inner matrix function in
$H^{\infty}_{M_{n\times m}}\label{H^inftyMnm}$ \hbox{\rm ($m\le n$)}. \
Furthermore,
$\Theta$ is unique up to a unitary constant right factor; that is,
if $M=\Delta H^2_{\mathbb{C}^r}$ \hbox{\rm (}for $\Delta$ an inner
function in $H^{\infty}_{M_{n\times r}}$\hbox{\rm )}, then $m=r$ and $\Theta=\Delta W$,
where $W$ is a \hbox{\rm (}constant in z\hbox{\rm )} unitary matrix mapping $\mathbb C^m$
onto $\mathbb C^m$. \  }

\bigskip

As customarily done, we say that two matrix-valued functions $A$ and $B$
are {\it equal} if they are equal up to a unitary constant right
factor. \  Observe by (\ref{1.4}) that for $\Phi\in L^\infty_{M_n}$,
$H_\Phi S=H_\Phi T_{zI_n}=H_{\Phi\cdot zI_n}
=H_{zI_n\cdot\Phi}=T^*_{zI_n}H_\Phi$, which implies that the kernel
of a block Hankel operator $H_\Phi$ is an invariant subspace of the
shift operator on $H^2_{\mathbb C^n}$. \  Thus, if
$\hbox{ker}\,{H_{\Phi}} \ne \{0\}$, then by the Beurling-Lax-Halmos
Theorem,
$$
\hbox{ker}\, H_\Phi=\Theta H^2_{\mathbb{C}^m}
$$
for some inner matrix function $\Theta$. \  We note that $\Theta$
need not be a square matrix. \

On the other hand, recently C. Gu,
J. Hendricks and D. Rutherford \cite{GHR} considered the hyponormality of
block Toeplitz operators and characterized it in terms of their symbols. \  In particular
they showed that if $T_\Phi$ is a hyponormal block Toeplitz operator
on $H^2_{\mathbb{C}^n}$, then its symbol $\Phi$ is normal, i.e.,
$\Phi^*\Phi=\Phi\Phi^*$. \  Their characterization for hyponormality
of block Toeplitz operators resembles Cowen's Theorem except for
an additional condition -- the normality condition of the symbol. \

\bigskip

\noindent {\bf Hyponormality of Block Toeplitz Operators}
(\cite{GHR})\label{gu} {\it For each
$\Phi\in L^\infty_{M_n}$, let
$$
\mathcal{E}(\Phi)\label{Ebigphi}:=\Bigl\{K\in H^\infty_{M_n}:\ ||K||_\infty \le 1\
\ \hbox{and}\ \ \Phi-K \Phi^*\in H^\infty_{M_n}\Bigr\}. \
$$
Then $T_\Phi$ is hyponormal if and only if $\Phi$ is normal and
$\mathcal{E}(\Phi)$ is nonempty. \  }

\bigskip

The hyponormality of the Toeplitz operator $T_\Phi$ with arbitrary
matrix-valued symbol $\Phi$, though solved in principle by Cowen's
Theorem \cite{Co4} and the criterion due to Gu, Hendricks and
Rutherford \cite{GHR}, is in practice very complicated. \  Until
now, explicit criteria for the hyponormality of block Toeplitz
operators $T_\Phi$ with matrix-valued trigonometric polynomials or
rational functions $\Phi$ were established via interpolation
problems (\cite{GHR}, \cite{HL4}, \cite{HL5}). \

In Section 3, we obtain a tractable criterion for the hyponormality
of block Toeplitz operators with bounded type symbols. \ To do this
we employ a continuous analogue of the elementary theorem of Schur
on triangularization of finite matrices: If $T$ is a finite matrix
then it can be represented as $T=D+N$, where $D$ is a diagonal
matrix and $N$ is a nilpotent matrix. \  The continuous analogue is
the so-called triangularization theorem for compressions of the
shift operator: in this case, $D$ and $N$ are replaced by a certain
(normal) multiplication operator and a Volterra operator of
Hilbert-Schmidt class, respectively. \

Section 4 deals with the gap between hyponormality and subnormality
of block Toeplitz operators. \  The Bram-Halmos criterion for
subnormality (\cite{Br}, \cite{Con}) states that an operator $T\in
\mathcal{B(H)}$ is subnormal if and only if $\sum_{i,j}(T^ix_j, T^j
x_i)\ge 0$ for all finite collections
$x_0,x_1,\cdots,x_k\in\mathcal{H}$. \  It is easy to see that this
is equivalent to the following positivity test:
\begin{equation}\label{1.6}
\begin{bmatrix}
\hbox{$[T^*,T]$}& \hbox{$[T^{*2},T]$}& \hdots & \hbox{$[T^{*k},T]$}\\
\hbox{$[T^{*}, T^2]$}& \hbox{$[T^{*2},T^2]$} & \hdots & \hbox{$[T^{*k},T^2]$}\\
\vdots & \vdots & \ddots & \vdots\\
\hbox{$[T^*, T^k]$} & \hbox{$[T^{*2}, T^k]$} & \hdots &
\hbox{$[T^{*k},T^k]$}
\end{bmatrix}\ge 0
\qquad\text{(all $k\ge 1$)}. \
\end{equation}
The positivity condition (\ref{1.6}) for $k=1$ is equivalent to the
hyponormality of $T$, while subnormality requires the validity of
(\ref{1.6}) for all $k\in \mathbb{Z}_+$. \  The Bram-Halmos
criterion indicates that hyponormality is generally
far from subnormality. \  But there are special classes of operators
for which the positivity of (\ref{1.6}) for some $k$ and subnormaity
are equivalent. \  For example, it was shown in (\cite{CuL1}) that
if $W_{\sqrt{x},(\sqrt{a},\sqrt{b},\sqrt{c})^{\wedge}}$ is the
weighted shift whose weight sequence consists of the initial weight
$x$ followed by the weight sequence of the recursively generated
subnormal weighted shift $W_{(\sqrt{a},\sqrt{b},\sqrt{c})^{\wedge}}$
with an initial segment of positive weights
$\sqrt{a},\sqrt{b},\sqrt{c}$ (cf. \cite{CuF1}, \cite{CuF2},
\cite{CuF3}), then $W_\alpha$ is subnormal if and only if the
positivity condition (\ref{1.6}) is satisfied with $k=2$. \  On the
other hand, in \cite [Problem 209]{Hal3}, it was shown that there
exists a hyponormal operator whose square is not hyponormal, e.g.,
$U^*+2U$ ($U$ is the unilateral shift on $\ell^2$), which is a {\it
trigonometric} Toeplitz operator, i.e., $U^*+2U \equiv T_{\bar z+2z}$. \
This example addresses the gap
between hyponormality and subnormality for Toeplitz operators. \
This matter is closely related to Halmos's Problem 5 \cite{Hal1},
\cite{Hal2}: Is every subnormal Toeplitz operator either normal or
analytic\,?

In \cite{CuL1}, as a partial answer, it was shown that every
hyponormal Toeplitz operator $T_\varphi$ with trigonometric
polynomial symbol $\varphi$ whose square is hyponormal must be
either normal or analytic. \  In \cite{Gu3}, C. Gu showed that this
result still holds for Toeplitz operators $T_\varphi$ with rational
symbol $\varphi$ (more generally, in the cases where $\varphi$
is of bounded type). \  In Section 4 we prove the
following theorem: If $\Phi$ is a matrix-valued rational function
whose co-analytic part has a coprime factorization then every
hyponormal Toeplitz operator $T_{\Phi}$ whose square is hyponormal
must be either normal or analytic. \  This result generalizes the
results in \cite{CuL1} and \cite{Gu3}.

In Section 5, we consider a completion problem involving Toeplitz
operators. \  Given a partially specified operator matrix, the
problem of finding suitable operators to complete the given partial
operator matrix so that the resulting matrix satisfies certain given
properties is called a {\it completion problem}. \ The dilation
problem is a special case of the completion problem: in other words,
a dilation of $T$ is a completion of the partial operator matrix
$\left[\begin{smallmatrix} T&?\\?&?\end{smallmatrix}\right]$. \  In
recent years, operator theorists have been interested in the
subnormal completion problem for
$$
\begin{bmatrix} U^*&?\\?&U^*\end{bmatrix},
$$
where $U$ is the unilateral shift on $H^2$. \  If the unspecified
entries ? are Toeplitz operators this is called the {\it Toeplitz}
subnormal completion problem. \  Thus this problem is related to the
subnormality of block Toeplitz operators. \  In Section 5, we solve
this Toeplitz subnormal completion problem. \

Finally, in Section 6 we list some open problems. \

\medskip \textit{Acknowledgment}. \ The authors are indebted to the referee for suggestions that improved the presentation. 
\

\vskip 1cm

%
%
%
%

\section{Basic Theory and Preliminaries}

\bigskip

%
%




We first recall \cite[Lemma 3]{Ab} that if $\varphi\in L^\infty$
then
\begin{equation}\label{2.1}
\hbox{$\varphi$ is of bounded type}\ \Longleftrightarrow\
\hbox{ker}\, H_\varphi\ne \{0\}\,.
\end{equation}
If $\varphi\in L^\infty$, we write
$$
\varphi_+\equiv P \varphi\in H^2\quad\text{and}\quad \varphi_-\equiv
\overline{P^\perp \varphi}\in zH^2.
$$
Assume now that both $\varphi$ and $\overline\varphi$ are of bounded
type. \  Then from Beurling's Theorem, $\text{ker}\,
H_{\overline{\varphi_-}}=\theta_0 H^2$ and $\text{ker}\,
H_{\overline{\varphi_+}}=\theta_+ H^2$ for some inner functions
$\theta_0, \theta_+$. \ We thus have
$b:={\overline{\varphi_-}}\theta_0 \in H^2$, and hence we can write
\begin{equation}\label{2.2}
\varphi_-=\theta_0\overline{b} \text{~and similarly~} \varphi_+
=\theta_+\overline{a} \text{~for some~} a \in H^2.
\end{equation}
In particular, if $T_\varphi$ is hyponormal then since
\begin{equation}\label{2.3}
[T_\varphi^*, T_\varphi]=H_{\overline\varphi}^*
H_{\overline\varphi}-H_\varphi^* H_\varphi=
H_{\overline{\varphi_+}}^*
H_{\overline{\varphi_+}}-H_{\overline{\varphi_-}}^*
H_{\overline{\varphi_-}},
\end{equation}
it follows that $||H_{\overline{\varphi_+}} f||\ge
||H_{\overline{\varphi_-}} f||$ for all $f\in H^2$, and hence
$$
\theta_+ H^2= \text{ker}\, H_{\overline{\varphi_+}}\subseteq
\text{ker}\, H_{\overline{\varphi_-}}=\theta_0 H^2,
$$
which implies that $\theta_0$ divides $\theta_+$, i.e.,
$\theta_+=\theta_0\theta_1$ for some inner function $\theta_1$. \
We write, for an inner function $\theta$,
$$
\mathcal H(\theta):=H^2\ominus \theta\,H^2.
$$
Note that if $f=\theta \overline a \in L^2$, then $f\in H^2$ if and
only if $a\in \mathcal H(z\theta)$; in particular, if $f(0)=0$ then
$a\in \mathcal H(\theta)$. \  Thus, if
$\varphi=\overline{\varphi_-}+\varphi_+\in L^\infty$ is such that
$\varphi$ and $\overline\varphi$ are of bounded type such that
$\varphi_+(0)=0$ and $T_\varphi$ is hyponormal, then we can write
$$
\varphi_+=\theta_0\theta_1\bar a\quad\text{and}\quad
\varphi_-=\theta_0 \bar b, \qquad\text{where $a\in
\mathcal{H}(\theta_0\theta_1)$ and $b\in\mathcal{H}(\theta_0)$.}
$$
By Kronecker's Lemma \cite[p. 183]{Ni}, if $f\in H^\infty$ then
$\overline f$ is a rational function if and only if $\hbox{rank}\,
H_{\overline f}<\infty$, which implies that
\begin{equation}\label{2.4}
\hbox{$\overline{f}$ is rational} \ \Longleftrightarrow\
f=\theta\overline b\ \ \hbox{with a finite Blaschke product
$\theta$}.
\end{equation}
Also, from the scalar-valued case of (\ref{1.4}), we can see that if
$k\in \mathcal{E}(\varphi)$ then
\begin{equation}\label{2.5}
[T_\varphi^*, T_\varphi]=H_{\overline\varphi}^*
H_{\overline\varphi}-H_\varphi^* H_\varphi= H_{\overline\varphi}^*
H_{\overline\varphi}-H_{k\,\overline\varphi}^*
H_{k\,\overline\varphi}= H_{\overline\varphi}^*(1-T_{\widetilde{k}}
T_{\widetilde{k}}^*) H_{\overline\varphi}.
\end{equation}

\medskip

On the other hand, M. Abrahamse \cite[Lemma 6]{Ab} showed that if
$T_\varphi$ is hyponormal, if $\varphi\notin H^\infty$, and if
$\varphi$ or $\overline{\varphi}$ is of bounded type then both
$\varphi$ and $\overline{\varphi}$ are of bounded type. \  However,
by contrast to the scalar case, $\Phi^*$ may not be of bounded type
even though $T_\Phi$ is hyponormal, $\Phi\notin H^\infty_{M_n}$ and
$\Phi$ is of bounded type. \  But we have a one-way implication: if
$T_\Phi$ is hyponormal and $\Phi^*$ is of bounded type then $\Phi$
is also of bounded type (see \cite [Corollary 3.5 and Remark
3.6]{GHR}). \  Thus whenever we deal with hyponormal Toeplitz
operators $T_\Phi$ with symbols $\Phi$ satisfying that both $\Phi$
and $\Phi^*$ are of bounded type (e.g.,
$\Phi$ is a matrix-valued rational function),
it suffices to assume that only
$\Phi^*$ is of bounded type. \  In spite of this, for convenience,
we will assume that $\Phi$ and $\Phi^*$ are of bounded type whenever
we deal with bounded type symbols. \

\smallskip

For a matrix-valued function $\Phi\in H^2_{M_{n\times r}}$, we say
that $\Delta\in H^2_{M_{n\times m}}$ is a {\it left inner divisor}
of $\Phi$ if $\Delta$ is an inner matrix function such that
$\Phi=\Delta A$ for some $A \in H^{2}_{M_{m\times r}}$ ($m\le n$).
We also say that two matrix functions $\Phi\in H^2_{M_{n\times r}}$
and $\Psi\in H^2_{M_{n\times m}}$ are {\it left coprime} if the only
common left inner divisor of both $\Phi$ and $\Psi$ is a unitary
constant and that $\Phi\in H^2_{M_{n\times r}}$ and $\Psi\in
H^2_{M_{m\times r}}$ are {\it right coprime} if $\widetilde\Phi$ and
 $\widetilde\Psi$ are left coprime.
Two matrix functions $\Phi$ and $\Psi$ in $H^2_{M_n}$ are said to be
{\it coprime} if they are both left and right coprime. We note
that if $\Phi\in H^2_{M_n}$ is such that $\hbox{det}\,\Phi$ is not
identically zero then any left inner divisor $\Delta$ of $\Phi$ is
square, i.e., $\Delta\in H^2_{M_n}$: indeed, if $\Phi=\Delta A$ with $\Delta\in
H^2_{M_{n\times r}}$ ($r<n$) then for almost all $z\in \mathbb T$,
$\hbox{rank}\,\Phi(z)\le \hbox{rank}\,\Delta(z)\le r <n$, so that
$\hbox{det}\,\Phi(z)=0$ for almost all $z\in\mathbb T$.
If $\Phi\in H^2_{M_n}$ is such
that $\hbox{det}\,\Phi$ is not identically zero then we say that
$\Delta\in H^2_{M_{n}}$ is a {\it right inner divisor} of $\Phi$ if
$\widetilde{\Delta}$ is a left inner divisor of $\widetilde{\Phi}$.

\smallskip

On the
other hand, we have (in the Introduction) remarked that $\Theta$
need not be square in the equality $\hbox{ker}\, H_\Phi=\Theta
H^2_{\mathbb{C}^n}$, which comes from the Beurling-Lax-Halmos
Theorem. \  But it was known \cite [Theorem 2.2] {GHR}  that for
$\Phi\in L^\infty_{M_n}$, $\Phi$ is of bounded type if and only if
$\hbox{ker}\, H_\Phi=\Theta H^2_{\mathbb{C}^n}$ for some square
inner matrix function $\Theta$. \

Let $\{\Theta_i\in H^\infty_{M_n}: i\in J\}$ be a family of inner
matrix functions. \  Then the greatest common left inner divisor
$\Theta_d\label{thetad}$ and the least common left inner multiple
$\Theta_m\label{thetam}$ of
the family $\{\Theta_i\in H^\infty_{M_n}: i\in J\}$ are the inner
functions defined by
$$
\Theta_d H^2_{\mathbb C^p}=\bigvee_{i \in J}\Theta_{i}H^2_{\mathbb
C^n} \quad \hbox{and} \quad \Theta_m H^2_{\mathbb C^q}=\bigcap_{i\in
J}\Theta_{i}H^2_{\mathbb C^n}.
$$
The greatest common right inner divisor $\Theta_d^{\prime}$ and the
least common right inner multiple $\Theta_m^{\prime}$ of the family
$\{\Theta_i\in H^\infty_{M_n}: i\in J\}$ are the inner functions
defined by
$$
\widetilde{\Theta}_d^{\prime} H^2_{\mathbb C^r}=\bigvee_{i \in
J}\widetilde{\Theta}_{i} H^2_{\mathbb C^n} \quad \hbox{and} \quad
\widetilde{\Theta}_m^{\prime} H^2_{\mathbb C^s}=\bigcap_{i \in J}
\widetilde{\Theta}_{i}H^2_{\mathbb C^n}.
$$
The Beurling-Lax-Halmos Theorem guarantees that $\Theta_d$ and
$\Theta_m$ are unique up to a unitary constant right factor, and
$\Theta_d^{\prime}$ and $\Theta_m^{\prime}$ are unique up to a
unitary constant left factor. \  We write
$$
\begin{aligned}
&\Theta_d =\hbox{GCD}_{\ell}\,\{\Theta_i: i\in J\},\quad
\Theta_m=\hbox{LCM}_{\ell}\,\{\Theta_i: i\in J\},\\
&\Theta_d^\prime=\hbox{GCD}_r\,\{\Theta_i: i\in J\},\quad
\Theta_m^\prime=\hbox{LCM}_r\,\{\Theta_i: i\in J\}.
\end{aligned}
$$
If $n=1$, then $\hbox{GCD}_{\ell}\,\{\cdot\}=\hbox{GCD}_r\,\{\cdot\}$
(simply denoted $\hbox{GCD}\label{GCD}\,\{\cdot\}$) and
$\hbox{LCM}_{\ell}\,\{\cdot\}=\hbox{LCM}_r\,\{\cdot\}$ (simply denoted
$\hbox{LCM}\label{LCM}\,\{\cdot\}$). \  In general, it is not true
that $\hbox{GCD}_{\ell}\,\{\cdot\}=\hbox{GCD}_r\,\{\cdot\}$ and
$\hbox{LCM}_{\ell}\,\{\cdot\}=\hbox{LCM}_r\,\{\cdot\}$.

\medskip

However, we have:

\begin{lemma}\label{lem2.1}
Let $\Theta_i:={\theta_i}I_n$ for an inner function $\theta_i \ (i
\in J)$. \
\begin{enumerate}
\item[(a)] $\hbox{\rm GCD}_{\ell}\,\{\Theta_i: i\in J\}=\hbox{\rm GCD}_r\,\{\Theta_i: i\in J\}
={\theta_d}I_n$, where $\theta_d=\text{\rm GCD}\,\{\theta_i : i \in
J \}$.
\item[(b)] $\hbox{\rm LCM}_{\ell}\,\{\Theta_i: i\in J\}=\hbox{\rm LCM}_r\,\{\Theta_i: i\in J\}
={\theta_m}I_n$, where $\theta_m=\text{\rm LCM}\,\{\theta_i : i \in
J \}$.
\end{enumerate}
\end{lemma}

\begin{proof}
(a) If $\Theta_d=\hbox{\rm GCD}_{\ell}\{\Theta_i: i \in J\}$, then
$$
\Theta_d H^2_{\mathbb C^n}=\bigvee_{i \in J} \Theta_i H^2_{\mathbb
C^n} =\bigoplus_{j=1}^n \bigvee_{i \in J} \theta_i H^2
=\bigoplus_{j=1}^n \theta_d H^2,
$$
which implies that $\Theta_d= {\theta_d} I_n$ with
$\theta_d=\text{\rm GCD}\,\,\{\theta_i : i \in J \}$. \  If instead
$\Theta_d=\hbox{\rm GCD}_{r}\{\Theta_i: i \in J\}$ then
$\widetilde{\Theta}_d=\hbox{\rm GCD}_{\ell}\{\widetilde{\Theta}_i: i
\in J\}$. \  Thus we have
$\widetilde{\Theta}_d={\widetilde{\theta}_d} I_n$ and hence,
$\Theta_d={\theta_d} I_n$. \

(b) If $\Theta_m=\hbox{\rm LCM}_{\ell}\{\Theta_i: i \in J\}$, then
$$
\Theta_m H^2_{\mathbb C^n}=\bigcap_{i \in J} \Theta_i H^2_{\mathbb
C^n} =\bigoplus_{j=1}^n \bigcap_{i \in J} \theta_i
H^2=\bigoplus_{j=1}^n \theta_m H^2,
$$
which implies that $\Theta_m={\theta_m} I_n$ with
$\theta_m=\text{\rm LCM}\,\,\{\theta_i : i \in J \}$. \  If instead
$\Theta_m=\hbox{\rm LCM}_{r}\{\Theta_i: i \in J\}$, then the same
argument as in (a) gives the result. \
\end{proof}

\medskip

In view of Lemma \ref{lem2.1}, if $\Theta_i=\theta_i I_n$ for an
inner function $\theta_i$ ($i\in J)$, we can define the greatest
common inner divisor $\Theta_d$ and the least common inner multiple
$\Theta_m$ of the $\Theta_i$ by
$$
\begin{aligned}
&\Theta_d\equiv\hbox{GCD}\,\{\Theta_i:i\in
J\}:=\hbox{GCD}_{\ell}\,\{\Theta_i:i\in J\}
=\hbox{GCD}_r\,\{\Theta_i:i\in J\};\\
&\Theta_m\equiv\hbox{LCM}\,\{\Theta_i:i\in
J\}:=\hbox{LCM}_{\ell}\,\{\Theta_i:i\in J\}
=\hbox{LCM}_r\,\{\Theta_i:i\in J\}\,:
\end{aligned}
$$
they are both diagonal matrices. \

\medskip
For $\Phi\in L^\infty_{M_n}$ we write
$$
\Phi_+\label{Phi_+}:=P_n(\Phi)\in H^2_{M_n} \quad\hbox{and}\quad
\Phi_-\label{Phi_-}:=\bigl[P_n^\perp (\Phi)\bigr]^* \in H^2_{M_n}.
$$
Thus we can write $\Phi=\Phi_-^*+\Phi_+\,$. \  Suppose
$\Phi=[\varphi_{ij}] \in L^\infty_{M_n}$ is such that $\Phi^*$ is of
bounded type. \  Then we may write
$\varphi_{ij}=\theta_{ij}\overline{b}_{ij}$, where $\theta_{ij}$ is
an inner function and $\theta_{ij}$ and $b_{ij}$ are coprime. \ Thus
if $\theta$ is the least common inner multiple of $\theta_{ij}$'s
then we can write
\begin{equation}\label{2.6}
\Phi=[\varphi_{ij}]=[\theta_{ij}\overline{b}_{ij}]=[\theta
\overline{a}_{ij}]= \Theta A^* \quad (\Theta\equiv\theta I_n,\
A\equiv[a_{ij}] \in H^{2}_{M_n}).
\end{equation}
We note that in the factorization (\ref{2.6}), $A(\alpha)$ is
nonzero whenever $\theta(\alpha)=0$. \  Let $\Phi=\Phi_-^*+\Phi_+\in
L^\infty_{M_n}$ be such that $\Phi$ and $\Phi^*$ are of bounded
type. \  Then in view of (\ref{2.6}) we can write
$$
\Phi_+= \Theta_1 A^* \quad\hbox{and}\quad \Phi_-= \Theta_2 B^*,
$$
where $\Theta_i =\theta_i I_n$ with an inner function $\theta_i$ for
$i=1,2$ and $A,B\in H^{2}_{M_n}$. \  In particular, if $\Phi\in
L^\infty_{M_n}$ is rational then the $\theta_i$ can be chosen as
finite Blaschke products, as we observed in (\ref{2.4}). \

By contrast with scalar-valued functions, in (\ref{2.6}) $\Theta$
and $A$ need not be (right) coprime: for instance, if
$\Phi:=\left[\begin{smallmatrix} z&z\\ z&z\end{smallmatrix}\right]$
then we can write
$$
\Phi=\Theta A^* = \begin{bmatrix} z&0\\ 0&z\end{bmatrix}
\begin{bmatrix} 1&1\\ 1&1\end{bmatrix},
$$
but $\Theta:=\left[\begin{smallmatrix} z&0\\
0&z\end{smallmatrix}\right]$ and $A:=\left[\begin{smallmatrix} 1&1\\
1&1\end{smallmatrix}\right]$ are not right coprime because
$\frac{1}{\sqrt{2}} \left[\begin{smallmatrix} z&-z\\
1&1\end{smallmatrix}\right]$ is a common right inner divisor, i.e.,
\begin{equation}\label{2.7}
\Theta=\frac{1}{\sqrt{2}} \begin{bmatrix} 1&z\\
-1&z\end{bmatrix}
\,\cdot\, \frac{1}{\sqrt{2}} \begin{bmatrix} z&-z\\
1&1\end{bmatrix}\quad\hbox{and}\quad A=\sqrt{2} \begin{bmatrix}
0&1\\ 0&1\end{bmatrix} \,\cdot\, \frac{1}{\sqrt{2}} \begin{bmatrix}
z&-z\\ 1&1\end{bmatrix}.
\end{equation}

If $\Omega=\hbox{GCD}_{\ell}\,\{A,\Theta\}$ in the representation
(\ref{2.6}):
$$
\Phi=\Theta A^*=A^*\Theta \quad\hbox{($\Theta\equiv \theta I_n$ for
an inner function $\theta$)},
$$
then $\Theta=\Omega \Omega_{\ell}$ and $A=\Omega A_{\ell}$ for some inner
matrix $\Omega_{\ell}$ (where $\Omega_{\ell}\in H^2_{M_n}$ because
$\hbox{det}\,\Theta$ is not identically zero) and some $A_l \in
H^{2}_{M_n}$. \  Therefore if $\Phi^*\in L^\infty_{M_n}$ is of
bounded type then we can write
\begin{equation}\label{2.8}
\Phi={A_{\ell}}^*\Omega_{\ell},\quad\hbox{where $A_{\ell}$ and $\Omega_{\ell}$ are
left coprime.}
\end{equation}
$A_{\ell}^*\Omega_{\ell}$ is called the {\it left coprime factorization} of
$\Phi$; similarly, we can write
\begin{equation}\label{2.9}
\Phi=\Omega_r A_r^*, \quad\hbox{where $A_r$ and $\Omega_r$ are right
coprime.}
\end{equation}
In this case, $\Omega_r A_r^*$ is called the {\it right coprime
factorization} of $\Phi$. \

\begin{remark} \ (\cite[Corollary 2.5]{GHR}) \
As a consequence of the Beurling-Lax-Halmos Theorem, we can see that
\begin{equation}\label{RCD}
\Phi=\Omega_r A_r^*\ \hbox{(right coprime factorization)}\
\Longleftrightarrow \hbox{ker}\,H_{\Phi^*}=\Omega_r H_{\mathbb
C^n}^2.
\end{equation}
In fact, if $\Phi=\Omega_r A_r^*$ (right coprime factorization)
then it is evident that
\begin{equation*}
\hbox{ker}\, H_{\Phi^*} \supseteq \Omega_r H_{\mathbb C^n}^2.
\end{equation*}
From the Beurling-Lax-Halmos Theorem,
\begin{equation*}
\hbox{ker}\, H_{\Phi^*}=\Theta H_{\mathbb C^n}^2,
\end{equation*}
for some inner function $\Theta$, and hence $(I-P)( \Phi^* \Theta )=0$,
i.e., $\Phi^* = D \Theta^*$, for some $D \in H_{\mathbb C^n}^2$. \
We want to show that $\Omega_r=\Theta$ up to a unitary constant right factor. \
Since $\Theta H_{\mathbb C^n}^2 \supseteq \Omega_r H_{\mathbb C^n}^2$,
we have (cf. \cite[p.240]{FF})
that $\Omega_r = \Theta \Delta$ for some square inner function $\Delta$. \
Thus,
\begin{equation*}
D \Theta^*=\Phi^*=A_r \Omega_r^*=A_r \Delta^* \Theta^*,
\end{equation*}
which implies $A_r=D \Delta$, so that $\Delta$ is a common right
inner divisor of both $A_r$ and $\Omega_r$. \ But since $A_r$ and
$\Omega_r$ are right coprime, $\Delta$ must be a unitary constant. \
The proof of the converse implication is entirely similar. \qed
\end{remark}
\medskip

From now on, for notational convenience we write
$$
I_\omega\label{Iomega}:= \omega\, I_n\ \ \hbox{($\omega\in H^2$)}
\quad\hbox{and}\quad H_0^2\label{H_0^2} := I_z\, H^2_{M_n}.
$$
It is not easy to check the condition ``$B$ and $\Theta$ are
coprime" in the decomposition $F=B^*\Theta$ ($\Theta\equiv I_\theta$
is inner and $B\in H^2_{M_n}$). \  But if $F$ is rational (and hence
$\Theta$ is given in a form $\Theta\equiv I_\theta$ with a finite
Blaschke product $\theta$) then we can obtain a more tractable
criterion. \  To see this, we need to recall the notion of finite
Blaschke-Potapov product. \

Let $\lambda\in\mathbb D$ and write
$$
b_\lambda(z)\label{blambdaz}:=\frac{z-\lambda}{1-\overline \lambda
z},
$$
which is called a {\it Blaschke factor}. \  If $M$ is a closed
subspace of $\mathbb C^n$ then the matrix function of the form
$$
b_\lambda P_M+(I-P_M)\quad \hbox{($P_M:=$the orthogonal projection
of $\mathbb C^n$ onto $M$)}
$$
is called a {\it Blaschke-Potapov factor}\,; an $n\times n$ matrix
function $D$ is called {\it a finite Blaschke-Potapov product} if
$D$ is of the form
$$
D \label{Dz}=\nu \prod_{m=1}^M \Bigl(b_m P_m + (I-P_m)\Bigr),
$$
where $\nu$ is an $n\times n$ unitary constant matrix, $b_m$ is a
Blaschke factor, and $P_m$ is an orthogonal projection in $\mathbb
C^n$ for each $m=1,\cdots, M$. \  In particular, a scalar-valued
function $D$ reduces to a finite Blaschke product $D=\nu
\prod_{m=1}^M b_m$, where $\nu=e^{i\omega}$. \  It is also known
(cf. \cite{Po}) that an $n\times n$ matrix function $D$ is rational
and inner if and only if it can be represented as a finite
Blaschke-Potapov product. \

\medskip

Write $\mathcal{Z}(\theta)\label{Ztheta}$ for the set of zeros of an inner
function $\theta$. \  We then have:

\begin{lemma}\label{lem2.2}
Let $B\in H^\infty_{M_n}$ be rational and $\Theta=I_\theta$ with a
finite Blaschke product $\theta$. \  Then the following statements
are equivalent:
\begin{itemize}
\item[(a)] $B(\alpha)$ is invertible for each $\alpha \in \mathcal{Z}(\theta)$;
\item[(b)] $B$ and  $\Theta$ are right coprime;
\item[(c)] $B$ and  $\Theta$ are left coprime.
\end{itemize}
\end{lemma}

\begin{proof} See \cite [Lemma 3.10] {CHL}. \  \end{proof}

\bigskip

If $\Theta\in H^\infty_{M_n}$ is an inner matrix function, we write
$$
\begin{aligned}
\mathcal H(\Theta)\label{HTheta}&:=(\Theta H^2_{\mathbb C_n})^{\perp};\\
\mathcal H_{\Theta}\label{H_Theta}&:=(\Theta H^2_{M_n})^{\perp};\\
\mathcal K_{\Theta}\label{K_Theta}&:=(H^2_{M_n}\Theta)^{\perp}.
\end{aligned}
$$
If $\Theta=I_\theta$ for an inner function $\theta$ then
$\mathcal{H}_\Theta=\mathcal{K}_\Theta$ and if $n=1$, then
$\mathcal{H}(\Theta)=\mathcal{H}_{\Theta}= \mathcal{K}_{\Theta}$\,.
\

\bigskip

The following lemma is useful in the sequel.

\smallskip

\begin{lemma}\label{lem2.8}
If $\Theta \in H_{M_n}^2$ is an inner matrix function then
$$
\hbox{\rm dim}\, \mathcal H(\Theta)< \infty \ \Longleftrightarrow \
\Theta\ \hbox{is a finite Blaschke-Potapov product}.
$$
\end{lemma}

\begin{proof}
Let
$$
\delta:=\hbox{GCD}\left\{\omega: \omega \ \hbox{is inner}, \Theta  \
\hbox{is a left inner divisor of} \ \Omega=I_{\omega}\right\} \quad
\hbox{and} \quad \Delta:=I_{\delta}\,,
$$
in other words, $\delta$ is a `minimal' inner function such that
$\Delta\equiv I_\delta=\Theta\Theta_1$ for some inner matrix
function $\Theta_1$. \  Note that
$$
\aligned \Theta  \ \hbox{is a finite Blaschke-Potapov product}\
&\Longrightarrow\ \delta \
\hbox{is a finite Blaschke product}\\
&\Longrightarrow\ \hbox{dim}\,\mathcal H (\Delta)< \infty.
\endaligned
$$
Observe that
$$
\mathcal H (\Delta)=\mathcal H (\Theta \Theta_1)=\mathcal H (\Theta)
\bigoplus \Theta \mathcal H(\Theta_1)\,.
$$
Thus if $\Theta$ is a finite Blaschke-Potapov product, then
$\hbox{dim}\,\mathcal H (\Theta)< \infty$. \  Conversely, we suppose
$\hbox{dim}\,\mathcal H (\Theta)< \infty$. \ Write
$\Theta:=[\theta_{ij}]_{ij=1}^n$. \  Since
$$
\hbox{rank}\,H_{\overline{\theta}_{ij}}^* \leq
\hbox{rank}\,H_{\Theta^*}^* =\hbox{dim}\,\mathcal H(\Theta)<
\infty\,,
$$
it follows that $\theta_{ij}$'s are rational functions. \  Thus
$\Theta$ is a rational inner matrix function and hence a finite
Blaschke-Potapov product. \
\end{proof}

\medskip

Lemma \ref{lem2.8} implies that every inner divisor of a rational inner
function (i.e., a finite Blaschke-Potapov product) is also a finite
Blaschke-Potapov product: indeed, if $\Theta$ is a finite
Blaschke-Potapov product and $\Theta_1$ is
an inner divisor of $\Theta$, then
$\hbox{dim}\, \mathcal H(\Theta_1) \leq \hbox{dim}\, \mathcal
H(\Theta)<\infty$, and hence by Lemma \ref{lem2.8},
$\Theta_1$ is a finite Blaschke-Potapov product.

\bigskip

From Lemma \ref{lem2.8}, we know that every inner divisor of
$B_{\lambda}:=I_{b_{\lambda}} \in H_{M_n}^\infty$ is a finite
Blaschke-Potapov product. \  However we can say more:

\smallskip

\begin{lemma} \label{lemma2.3}
Every inner divisor of $B_{\lambda}:=I_{b_{\lambda}} \in
H_{M_n}^\infty$ is a Blaschke-Potapov factor. \
\end{lemma}

\begin{proof} Suppose $D$ is an inner divisor of $B_{\lambda}$. \
By Lemma \ref{lem2.8}, $D$ is a Blaschke-Potapov product
of the form
$$
D=\nu \prod_{i=1}^m D_i \ \ \hbox{with}\ D_i:=b_{\lambda_i}P_i +
(I-P_i)\quad \hbox{($m \leq n$)}.
$$
We write $B_{\lambda}=ED$ for some $E\in H^2_{M_n}$. \  Observe
that $B_{\lambda}(\lambda_i)$ is not invertible, so that
$\lambda_i=\lambda$ for all $i=1,2\hdots,m$. \  We thus have
$$
P_m + b_{\lambda}(I-P_m)=B_{\lambda}D_m^*= E \cdot \nu
\prod_{i=1}^{m-1} D_i\, .
$$
Then we have
$$
\hbox{ker}\, P_m=\hbox{ker}\, (B_{\lambda}D_m^*)(\lambda)\supseteq
\hbox{ker}\, D_{m-1}(\lambda)=\hbox{ran}\, P_{m-1},
$$
which implies that $P_m P_{m-1}=0$, and hence $P_m$ and $P_{m-1}$
are orthogonal. \  Thus $D_{m-1}D_m$ is a Blaschke-Potapov factor. \
Now an induction shows that $D$ is a Blaschke-Potapov factor. \
\end{proof}

\medskip

By the aid of Lemma \ref{lemma2.3}, we can show that the equivalence
(b)$\Leftrightarrow$(c) in Lemma \ref{2.3} fails if $\Theta$ is not
a {\it constant} diagonal matrix. \  To see this, let
$$
\Theta_1:=\begin{bmatrix}
b_{\alpha}&0\\0&1\end{bmatrix}\quad\hbox{and}\quad
\Theta_2:=\frac{1}{\sqrt{2}}\begin{bmatrix}z&-z\\1&1\end{bmatrix}\,.
$$
Then $\Theta:=\Theta_1\Theta_2$ and $\Theta_1$ are not left coprime.
Observe that
$$
\widetilde{\Theta}_1:=\begin{bmatrix}
b_{\overline{\alpha}}&0\\0&1\end{bmatrix}, \ \
\widetilde{\Theta}=\frac{1}{\sqrt{2}}
\begin{bmatrix}zb_{\overline{\alpha}}&1\\-zb_{\overline{\alpha}}&1\end{bmatrix}.
$$
Since every right inner divisor $\Delta$ of $\widetilde{\Theta}_1$
is an inner divisor of
$B_{\overline{\alpha}}:=I_{b_{\overline{\alpha}}}$, it follows from
Lemma \ref{lemma2.3} that $\Delta=\widetilde{\Theta}_1$ (up to a unitary
constant right factor). Suppose that $\Theta$ and $\Theta_1$ are not
right coprime. Then $\widetilde{\Theta}$ and $\widetilde{\Theta}_1$
are not left coprime and hence $\widetilde{\Theta}_1$ is a left
inner divisor of $\widetilde{\Theta}$. Write
$$
\frac{1}{\sqrt{2}}\begin{bmatrix}zb_{\overline{\alpha}}&1\\
-zb_{\overline{\alpha}}&1\end{bmatrix}=\begin{bmatrix}
b_{\overline{\alpha}}&0\\0&1\end{bmatrix}\begin{bmatrix}
f_{11}&f_{12}\\f_{21}&f_{21}\end{bmatrix} \qquad (f_{ij} \in H^2).
$$
Then we have $\frac{1}{\sqrt{2}}=b_{\overline{\alpha}}f_{12}$, so
that
$f_{12}=\frac{1}{\sqrt{2}}\overline{b_{\overline{\alpha}}}\notin
H^2$, giving a contradiction.

\bigskip

For $\mathcal{X}$ a subspace of $H^2_{M_n}$, we write $P_{\mathcal X}$ for
the orthogonal projection from $H^2_{M_n}$ onto $\mathcal X$.

\medskip

\begin{lemma}\label{lem2.4}
Let $\Theta\in H^\infty_{M_n}$ be an inner matrix function and $A
\in H_{M_n}^2$. \  Then the following hold:
\begin{enumerate}
\item[(a)] $A \in \mathcal K_{\Theta} \Longleftrightarrow \Theta A^* \in H_0^2$;
\item[(b)] $A \in \mathcal H_{\Theta} \Longleftrightarrow A^* \Theta  \in H_0^2$;
\item[(c)] $P_{H_0^2}(\Theta A^*)=\Theta\Bigl(P_{\mathcal K_{\Theta}}A\Bigr)^*$.
\end{enumerate}
\end{lemma}

\begin{proof} Let $C \in H_{M_n}^2$ be arbitrary. \  We then have
$$
\begin{aligned}
A \in \mathcal K_{\Theta}
&\Longleftrightarrow \langle A, C\Theta \rangle =0\\
&\Longleftrightarrow \int_{\mathbb T} \hbox{tr}\,\Bigl((C\Theta)^* A\Bigr)\,d\mu=0\\
&\Longleftrightarrow \int_{\mathbb T}
\hbox{tr}\,(C^*A\Theta^*)\,d\mu=0
\quad \left(\hbox{since tr($AB$)}=\hbox{tr($BA$)}\right)\\
&\Longleftrightarrow \langle A \Theta^*, C \rangle =0\\
&\Longleftrightarrow \Theta A^* \in H_0^2,
\end{aligned}
$$
giving (a) and similarly, (b). \  For (c), we write $A=A_1+A_2$,
where $A_1:=P_{\mathcal K_{\Theta}}A$ and $A_2:=A_3\Theta$ for some
$A_3\in H^2_{M_n}$. \  We then have
$$
P_{H_0^2}(\Theta A^*)=P_{H_0^2} (\Theta A_1^*+\Theta A_2^*)
=P_{H_0^2}\Bigl(\Theta (P_{\mathcal
K_{\Theta}}A)^*+\Theta\Theta^*A_3^*\Bigr) =\Theta (P_{\mathcal
K_{\Theta}}A)^*\,,
$$
giving (c). \
\end{proof}

\bigskip

%
%



We next review the classical Hermite-Fej\' er interpolation problem,
following \cite{FF}; this approach will be useful in the sequel. \
Let $\theta$ be a finite Blaschke product of degree $d$:
$$
\theta=e^{i\xi}\prod_{i=1}^N (\widetilde{b}_i)^{m_i} \quad
\left(\widetilde b_i:=\frac{z-\alpha_i}{1-\overline{\alpha}_i z},\
\hbox{where}\ \alpha_i\in\mathbb D\right),
$$
where $d=\sum_{i=1}^N m_i$. \  For  our purposes rewrite $\theta$ in
the form
$$
\theta=e^{i\xi}\prod_{j=1}^d b_j,
$$
where
$$
b_j:=\widetilde{b}_k\quad\hbox{if}\ \sum_{l=0}^{k-1} m_l<j\le
\sum_{l=0}^k m_l
$$
and, for notational convenience, $m_0:=0$. \  Let
\begin{equation}\label{2.10}
\varphi_j:=\frac{q_j}{1-\overline{\alpha}_j z} b_{j-1}b_{j-2}\cdots
b_1\quad (1\le j\le d),
\end{equation}
where $\varphi_1:=q_1 (1-\overline{\alpha}_1 z)^{-1}$ and
$q_j:=(1-|\alpha_j|^2)^{\frac{1}{2}}$ ($1\le j\le d$). \  It is well
known (cf. \cite{Ta}) that $\{\varphi_j\}_{j=1}^d$ is an orthonormal basis for
$\mathcal{H}(\theta)$. \

For our purposes we concentrate on the data given by sequences of
$n\times n$ complex matrices. \  Given the sequence $\{K_{ij}:\ 1\le
i\le N,\ 0\le j<m_i\}$ of $n\times n$ complex matrices and a set of
distinct complex numbers $\alpha_1,\hdots,\alpha_N$ in $\mathbb{D}$,
the classical Hermite-Fej\' er interpolation problem entails finding
necessary and sufficient conditions for the existence of a
contractive analytic matrix function $K$ in $H^{\infty}_{M_n}$
satisfying
\begin{equation} \label{2.11}
\frac{K^{(j)}(\alpha_i)}{j!} = K_{i,j} \qquad (1 \leq i \leq N, \ 0
\leq j < m_i).
\end{equation}
To construct a matrix polynomial $K(z)\equiv P(z)$ satisfying
(\ref{2.11}), let $p_i(z)$ be the polynomial of order $d-m_i$
defined by
$$
p_i(z):=\prod_{k=1,\\ k \neq i}^{N} \Bigl(\frac{z-
\alpha_k}{\alpha_i - \alpha_k}\Bigr)^{m_k}.
$$
Consider the matrix polynomial $P(z)$ of degree $d-1$ defined by
\begin{equation} \label{2.12}
P(z):= \sum_{i=1}^N \Biggl(K_{i,0}^{\prime} +K_{i,1}^{\prime}(z-
\alpha_i)+K_{i,2}^{\prime}(z- \alpha_i)^2+\cdots
+K_{i,m_i-1}^{\prime} (z-\alpha_i)^{m_i-1}\Biggr)\,p_i(z),
\end{equation}
where the $K_{i,j}^{\prime}$ are obtained by the following
equations:
$$
K_{i,j}^\prime=K_{i,j}-\sum_{k=0}^{j-1} \frac{K_{i,k}^\prime\,
p_i^{(j-k)}(\alpha_i)}{(j-k)!} \ \ (1\le i\le N;\ 0\le j<m_i)
$$
and $K_{i,0}^\prime=K_{i,0}$ ($1\le i\le N$). \  Then $P(z)$
satisfies (\ref{2.11}). \

On the other hand, for an inner function $\theta$, let $U_\theta$ be
defined by the compression of the shift operator $U$ : i.e.,
$$
U_\theta\label{U_theta}= P_{\mathcal H (\theta)} U \vert_{\mathcal H (\theta)}.
$$
Let $\Theta=I_{\theta}$ and $W\label{W}$ be the unitary operator from
$\bigoplus_1^d \mathbb C^n $ onto $\mathcal H (\Theta)$ defined by
\begin{equation} \label{2.13}
W:=(I_{\varphi_1},I_{\varphi_2}, \cdots , I_{\varphi_d}),
\end{equation}
where the $\varphi_j$ are the functions in (\ref{2.10}). \  It is
known \cite[Theorem X.1.5]{FF} that if $\theta$ is the finite
Blaschke product of order $d$, then $U_\theta$ is unitarily
equivalent to the lower triangular matrix $M\label{M}$ on $\mathbb{C}^d$
defined by
\bigskip
\begin{equation} \label{2.14}
M:={\small \begin{bmatrix} \alpha_1&0&0&0&\cdots&\cdots\\
q_1 q_2&\alpha_2&0&0&\cdots&\cdots\\
-q_1 \overline{\alpha}_1 q_3&q_2 q_3&\alpha_3&0&\cdots&\cdots\\
q_1 \overline{\alpha}_2 \overline{\alpha}_3q_4&-q_2
\overline{\alpha}_3q_4&q_3 q_4&\alpha_4&\cdots&\cdots\\
-q_1 \overline{\alpha}_2 \overline{\alpha}_3
\overline{\alpha}_4q_5&q_2 \overline{\alpha}_3
\overline{\alpha}_4q_5&-q_3 \overline{\alpha}_4q_5&q_4
q_5&\alpha_5&\ddots\\
\vdots&\vdots&\vdots&\ddots&\ddots&\ddots
\end{bmatrix}.}
\end{equation}

\noindent If $L \in M_n$ and $M=[m_{i,j}]_{d\times d}$, then the
matrix $L \otimes M$ is the matrix on $\mathbb C^{n \times d}$
defined by the block matrix
$$
L \otimes M\label{LotimesM}:=\begin{bmatrix} Lm_{1,1}&Lm_{1,2}&\cdots&Lm_{1,d}\\
Lm_{2,1}&Lm_{2,2}&\cdots&Lm_{2,d}\\
\vdots&\vdots&\vdots&\vdots\\
Lm_{d,1}&Lm_{d,2}&\cdots&Lm_{d,d}\end{bmatrix}.
$$
Now let $P(z) \in H^{\infty}_{M_n}$ be a matrix polynomial of degree
$k$. \  Then the matrix $P(M)$ on $\mathbb C^{n \times d}$ is
defined by
\begin{equation} \label{2.15}
P(M):=\sum_{i=0}^{k} P_i \otimes M^i, \quad \text{where}\ P(z)=
\sum_{i=0}^{k} P_i z^i.
\end{equation}
For $\Phi \in H^\infty_{M_n}$ and $\Theta:=I_\theta$ with an inner
function $\theta$, we write, for brevity,
\begin{equation}\label{2.16}
\left(T_\Phi\right)_\Theta :=P_{\mathcal H (\Theta)}T_\Phi
\vert_{\mathcal H (\Theta)},
\end{equation}
which is called the compression of $T_\Phi$ to
$\mathcal{H}(\Theta)$. \  If $M$ is given by (\ref{2.14}) and $P$ is
the matrix polynomial defined by (\ref{2.12}) then the matrix $P(M)$
is called the {\it Hermite-Fej\' er matrix} determined by
(\ref{2.15}). \  In particular, it is known \cite[Theorem X.5.6]{FF}
that
\begin{equation}\label{2.17}
W^* (T_P)_\Theta W=P(M),
\end{equation}
which says that $P(M)$ is a matrix representation for
$(T_P)_\Theta\label{TPTheta}$. \

\medskip

\begin{lemma}\label{lem2.7} Let $A \in H^{\infty}_{M_n}$ and $\Theta=I_{\theta}$ for a
finite Blaschke product $\theta$. \  If $A(\alpha)$ is invertible
for all $\alpha \in \mathcal Z(\theta)$ then $(T_A)_\Theta\label{TPhiTheta}$ is
invertible. \ \end{lemma}

\begin{proof} Suppose $(T_A)_\Theta f=0$ for some $f\in \mathcal H (
\Theta)$, so that $P_{\mathcal H (\Theta)}(Af)=0$ and hence, $Af \in
\Theta H^2_{\mathbb C^n}$. \  Since $A(\alpha)$ is invertible for
all $\alpha \in \mathcal Z(\theta)$, it follows that $f \in \Theta
H^2_{\mathbb C^n}$ and hence, $f\in \Theta H^2_{\mathbb C^n}\cap
\mathcal H (\Theta)=\{0\}$. \  Thus $(T_A)_\Theta$ is one-one. \ But
since $(T_A)_\Theta$ is a finite dimensional operator (because
$\theta$ is a finite Blaschke product), it follows that
$(T_A)_\Theta$ is invertible. \ \end{proof}

\vskip 1cm

%
%
%
%

\section{Hyponormality of Block Toeplitz Operators}

\bigskip
%
%


To get a tractable criterion for the hyponormality of block Toeplitz
operators with bounded type symbols, we need a triangular
representation for compressions of the unilateral shift operator $U
\equiv T_z$. \  We refer to \cite{AC} and \cite{Ni} for details
on this representation. \  For an explicit criterion, we need to
introduce the triangularization theorem concretely. \  To do so,
recall that for an inner function $\theta$, $U_\theta$ is defined by
\begin{equation}\label{3.1}
U_\theta= P_{\mathcal H (\theta)} U \vert_{\mathcal H (\theta)}.
\end{equation}
There are three cases to consider. \

\medskip

\noindent {\it Case 1} : Let $B$ be a Blaschke product and let
$\Lambda:=\{\lambda_n : n \geq 1\}$ be the sequence of zeros of $B$
counted with their multiplicities. \  Write
$$
\beta_1:=1, \quad
\beta_k:=\prod_{n=1}^{k-1}\frac{\lambda_n-z}{1-\overline{\lambda}_n z}\cdot
\frac{|\lambda_n|}{\lambda_n}\qquad (k\geq 2),
$$
and let
$$
\delta_j\label{deltaj}:=\frac{d_j}{1-\overline{\lambda}_j z}\beta_j \qquad (j \geq
1),
$$
where $d_j\label{dj}:=(1-|\lambda_j|^2)^{\frac{1}{2}}$. \  Let $\mu_B\label{mub}$ be a
measure on $\mathbb N$ given by $\mu_B(\{n\}):=\frac{1}{2}d_n^2, (n
\in \mathbb N). \ $ Then the map $V_B\label{VB}: L^2(\mu_B)\to \mathcal H(B)$
defined by
\begin{equation}\label{3.2}
V_B(c):=\frac{1}{\sqrt{2}} \sum_{n \geq 1} c(n)d_n \delta_n, \quad c
\equiv \{c(n)\}_{n \geq 1},
\end{equation}
is unitary and $U_B$ is mapped onto the operator
\begin{equation}\label{3.3}
V_B^*U_B V_B=(I-J_B)M_B,
\end{equation}
where $(M_B\label{MB} c)(n):=\lambda_n c(n)$ ($n\in \mathbb N$) is a
multiplication operator and
$$
(J_B\label{JB} c)(n):=\sum_{k=1}^{n-1}c(k)|\lambda_k|^{-2} \cdot
\frac{\beta_n(0)}{\beta_k(0)}d_k d_n \ \ \hbox{$(n \in \mathbb N$)}
$$
is a lower-triangular Hilbert-Schmidt operator. \

\medskip

\noindent {\it Case 2} : Let $s$ be a singular inner function with
continuous representing measure $\mu\equiv \mu_{s}\label{mus}$. \  Let
$\mu_{\lambda}$ be the projection of $\mu$ onto the arc $\{\zeta:
\zeta \in \mathbb T, \ 0< \text{arg}\zeta \leq \text{arg}\lambda\}$
and let
$$
s_{\lambda}(\zeta)\label{slambdazeta}:=\text{exp}\Bigl(-\int_{\mathbb
T}\frac{t+\zeta}{t-\zeta}d\mu_{\lambda}(t) \Bigr) \ \ \hbox{($\zeta
\in \mathbb D$)}.
$$
Then the map $V_s\label{Vs}: L^2(\mu)\to \mathcal H (s)$ defined by
\begin{equation}\label{3.4}
(V_s c)(\zeta)=\sqrt{2}\int_{\mathbb
T}c(\lambda)s_{\lambda}(\zeta)\frac{\lambda
d\mu(\lambda)}{\lambda-\zeta}\quad\hbox{($\zeta \in \mathbb D$)}
\end{equation}
is unitary and $U_s$ is  mapped onto the operator
\begin{equation}\label{3.5}
V_s^*U_s V_s=(I-J_s)M_s,
\end{equation}
where $(M_s c)(\lambda):=\lambda c(\lambda)$ ($\lambda\in\mathbb T$)
is a multiplication operator and
$$
(J_s\label{Js} c)(\lambda)=2 \int _{\mathbb T} e^{\mu(t)-\mu(\lambda)}
c(t)d_{\mu_{\lambda}}(t)\ \ \hbox{($\lambda \in \mathbb T$)}
$$
is a lower-triangular Hilbert-Schmidt operator. \

\medskip

\noindent {\it Case 3} : Let $\Delta$ be a singular inner function
with pure point representing measure $\mu\equiv \mu_{\Delta}\label{mudelta}$. \  We
enumerate the set $\{t\in \mathbb T: \mu(\{t\})>0 \}$ as a sequence
$\{t_k\}_{k \in \mathbb N}$. \  Write $\mu_k:=\mu(\{t_k\}), \ k \geq
1$. \  Further, let $\mu_{\Delta}$ be a measure on $\mathbb
R_{+}=[0, \infty)$ such that
$d\mu_{\Delta}(\lambda)=\mu_{[\lambda]+1}d \lambda$ and define a
function $\Delta_{\lambda}\label{deltalambda}$ on the unit disk $\mathbb D$ by the
formula
$$
\Delta_{\lambda}(\zeta):=\text{exp}\Biggl\{-\sum_{k=1}^{[\lambda]}\mu_k
\frac{t_k+\zeta}{t_k-\zeta}-(\lambda-[\lambda])\mu_{[\lambda]+1}
\frac{t_{[\lambda]+1}+\zeta}{t_{[\lambda]+1}-\zeta}\Biggr \},
$$
where $[\lambda]$ is the integer part of $\lambda$  ($\lambda \in
\mathbb R_{+}$) and by definition $\Delta_0:=1$. \  Then the map
$V_{\Delta}\label{Vdelta}: L^2(\mu_{\Delta})\to \mathcal H (\Delta)$ defined by
\begin{equation}\label{3.6}
(V_{\Delta}c)(\zeta):=\sqrt{2} \int_{\mathbb
R_{+}}c(\lambda)\Delta_{\lambda}(\zeta)(1-\overline{t}_{[\lambda]+1}
\zeta)^{-1}d\mu_{\Delta}(\lambda) \ \hbox{($\zeta \in \mathbb D$)}
\end{equation}
is unitary and $U_\Delta$ is mapped onto the operator
\begin{equation}\label{3.7}
V_{\Delta}^*U_\Delta V_{\Delta}=(I-J_{\Delta})M_{\Delta},
\end{equation}
where $(M_{\Delta}c)(\lambda):= t_{[\lambda]+1} c(\lambda)$,
($\lambda\in \mathbb R_+$) is a multiplication operator and
$$
(J_{\Delta}\label{Jdelta} c)(\lambda) := 2 \int _{0}^{\lambda}c(t)
\frac{\Delta_{\lambda}(0)}{\Delta_{t}(0)} d\mu_{\Delta}(t) \ \
\hbox{($\lambda \in \mathbb R_+$)}
$$
is a lower-triangular Hilbert-Schmidt operator. \

\medskip

Collecting the above three cases we get:

\medskip

\noindent {\bf Triangularization theorem.} (\cite [p.123] {Ni}) Let
$\theta$ be an inner function with the canonical factorization
$\theta=B\cdot s \cdot \Delta$, where $B$ is a Blaschke product, and
$s$ and $\Delta$ are singular functions with representing measures
$\mu_{s}$ and $\mu_{\Delta}$ respectively, with $\mu_{s}$ continuous
and $\mu_{\Delta}$ a pure point measure. \  Then the map
$V:\,L^2(\mu_B)\times L^2(\mu_s) \times L^2(\mu_{\Delta})\to
\mathcal H(\theta)$ defined by
\begin{equation}\label{3.8}
V:=\begin{bmatrix}V_B&0&0\\0&BV_s&0\\0&0&BsV_{\Delta}\end{bmatrix}
\end{equation}
is unitary, where $V_B, \mu_B, V_S, \mu_S, V_\Delta, \mu_\Delta$ are
defined in (\ref{3.2}) - (\ref{3.7}) and $U_\theta$ is mapped onto
the operator
$$
M:= V^*U_\theta V =
\begin{bmatrix}M_B&0&0\\0&M_{s}&0\\0&0&M_{\Delta}\end{bmatrix}+J,
$$
where $M_B, M_S, M_{\Delta}$ are defined in (\ref{3.3}), (\ref{3.5})
and (\ref{3.7}) and
$$
J:=-\begin{bmatrix}J_BM_B&0&0\\0&J_sM_s&0\\0&0&J_{\Delta}M_\Delta\end{bmatrix}+A
$$
is a lower-triangular Hilbert-Schmidt operator, with $A^3=0$,
$\text{rank}A\leq 3$. \

\bigskip

%
%


If $\Phi \in L^\infty_{M_n}$, then by (\ref{1.3}),
$$
[T_\Phi^*, T_\Phi]= H_{\Phi^*}^* H_{\Phi^*} - H_{\Phi}^*H_\Phi
        +T_{\Phi^*\Phi-\Phi\Phi^*}\,.
$$
Since the normality of $\Phi$ is a necessary condition for the
hyponormality of $T_\Phi$, the positivity of $H_{\Phi^*}^*
H_{\Phi^*} - H_{\Phi}^*H_\Phi$ is an essential condition for the
hyponormality of $T_\Phi$. \  Thus, we isolate this property as a
new notion, weaker than hyponormality. \  The reader will notice at
once that this notion is meaningful for non-scalar symbols. \

\medskip \noindent
\begin{definition} Let $\Phi \in L^\infty_{M_n}$. \  The {\it pseudo-selfcommutator}
of $T_\Phi$ is defined by
$$
[T_\Phi^*, T_\Phi]_p\label{pseudo}:= H_{\Phi^*}^* H_{\Phi^*} - H_{\Phi}^*H_\Phi.
$$
$T_{\Phi}$ is said to be {\it pseudo-hyponormal} if $[T_\Phi^*,
T_\Phi]_p$ is positive semidefinite. \
\end{definition}

As in the case of hyponormality of scalar Toeplitz operators, we can see that the
pseudo-hyponormality of $T_\Phi$ is independent of the constant
matrix term $\Phi(0)$. \  Thus whenever we consider the
pseudo-hyponormality of $T_\Phi$ we may assume that $\Phi(0)=0$. \
Observe that if $\Phi\in L^\infty_{M_n}$ then
$$
[T_\Phi^*, T_\Phi]= [T_\Phi^*, T_\Phi]_p +
T_{\Phi^*\Phi-\Phi\Phi^*}.
$$
We thus have
$$
T_\Phi\ \hbox{is hyponormal}\ \Longleftrightarrow\ T_\Phi\ \hbox{is
pseudo-hyponormal and $\Phi$ is normal;}
$$
and (via \cite[Theorem 3.3]{GHR}) $T_\Phi$ is pseudo-hyponormal if
and only if $\mathcal{E}(\Phi)\ne \emptyset$. \

\medskip

For $\Phi\equiv  \Phi_-^* + \Phi_+ \in L^{\infty}_{M_n}$, we write
$$
\mathcal C(\Phi)\label{Cphi}:=\Bigl\{K \in H^{\infty}_{M_n}:\ \Phi-K\Phi^*\in
H^{\infty}_{M_n}\Bigr\}.
$$
Thus if $\Phi\in L^\infty_{M_n}$ then
$$
K\in \mathcal{E}(\Phi)\ \Longleftrightarrow\ K\in\mathcal{C}(\Phi)\
\hbox{and}\ ||K||_\infty\le 1.
$$
Also if $K\in\mathcal{C}(\Phi)$ then $H_{\Phi_-^*}=H_{K\Phi_+^*} =
T_{\widetilde K}^*H_{\Phi_+^*}$, which gives a necessary condition
for the nonempty-ness of $\mathcal C(\Phi)$ (and hence the
hyponormality of $T_\Phi$): in other words,
\begin{equation}\label{3.9}
K\in\mathcal{C}(\Phi)\ \Longrightarrow\ \text{ker}\, H_{\Phi_+^*}
\subseteq \text{ker}\, H_{\Phi_-^*}.
\end{equation}

\medskip

We begin with:

\begin{proposition}\label{pro3.2}
Let $\Phi\equiv \Phi_-^* + \Phi_+ \in L^{\infty}_{M_n}$ be such that
$\Phi$ and $\Phi^*$ are of bounded type. \  Thus we may write
$$
\Phi_+= \Theta_1 A^*\quad \hbox{and}\quad \Phi_- =\Theta_2 B^*,
$$
where $\Theta_i=I_{\theta_i}$ for an inner function $\theta_i$ {\rm
($i=1,2$)} and $A,B\in H^2_{M_n}$. \  If
$\mathcal{C}(\Phi)\not=\emptyset$, then $\Theta_2$ is an inner
divisor of $\Theta_1$, {\rm i.e.}, $\Theta_1=\Theta_0 \Theta_2$ for
some inner function $\Theta_0$. \
\end{proposition}

\begin{proof} In view of (\ref{2.6}) we may write
$$
\Phi_+\equiv \left[\theta_1\overline{a}_{ij}\right]_{n\times n}\quad
\hbox{and}\quad
\Phi_-=\left[\theta_2\overline{b}_{ij}\right]_{n\times n}=
\left[\theta_{ij}\overline{c}_{ij}\right]_{n\times n},
$$
where each $\theta_{ij}$ is an inner function, $c_{ij}\in H^2$,
$\theta_{ij}$ and $c_{ij}$ are coprime, and $\theta_2$ is the least
common multiple of $\theta_{ij}$'s. \  Suppose
$\mathcal{C}(\Phi)\not=\emptyset$. \  Then there exists a matrix
function $K \in H^{\infty}_{M_n}$ such that $\Phi_-^*-K \Phi_+^* \in
H^{2}_{M_n}$. \  Thus $B\Theta_2^*-KA\Theta_1^* \in H^{2}_{M_n}$,
which implies that
$$
B\Theta_2^*\Theta_1=\left[b_{ji}\overline{\theta}_2\theta_1\right]
\in H^2_{M_n}. \
$$
But since $\theta_2\overline{b}_{ij}=\theta_{ij}\overline{c}_{ij}$,
and hence $b_{ij}=\theta_2\overline{\theta}_{ij}c_{ij}$, it follows
that
$$
b_{ji}\overline{\theta}_2\theta_1
=\left(\theta_2\overline{\theta}_{ji}c_{ji}\right)\overline{\theta}_2\theta_1
=\overline{\theta}_{ji}c_{ji}\theta_1\in H^\infty.
$$
Since $\theta_{ji}$ and $c_{ji}$ are coprime, we have that
$$
\overline{\theta}_{ji}\theta_1\in H^\infty,\ \ \hbox{and hence}\ \
\overline{\theta}_2\theta_1\in H^\infty,
$$
which implies that $\Theta_2$ divides $\Theta_1$. \
\end{proof}

\medskip

Proposition \ref{pro3.2} shows that the hyponormality of $T_\varphi$
with scalar-valued rational symbol $\varphi$ implies
$$
\hbox{deg}\,(\varphi_-)\le \hbox{deg}\,(\varphi_+),
$$
which is a generalization of the well-known result for the cases of
the trigonometric Toeplitz operators, i.e., if
$\varphi=\sum_{n=-m}^N a_n z^n$ is such that $T_\varphi$ is
hyponormal then $m\le N$ (cf. \cite{FL}). \

\medskip

In view of Proposition \ref{3.2}, when we study the hyponormality of
block Toeplitz operators with {\it bounded type symbols} $\Phi$
(i.e., $\Phi$ and $\Phi^*$ are of bounded type) we may assume that
the symbol $\Phi\equiv \Phi_-^* + \Phi_+\in L^{\infty}_{M_n}$ is of
the form
$$
\Phi_+= \Theta_1 \Theta_0 A^*\quad \hbox{and}\quad \Phi_-=\Theta_1
B^*,
$$
where $\Theta_i:=I_{\theta_i}$ for an inner function $\theta_i$
($i=0,1$) and $A,B\in H^2_{M_n}$. \

\medskip

Our criterion is as follows:

\begin{theorem}\label{thm3.3}
Let $\Phi\equiv \Phi_-^* + \Phi_+ \in L^{\infty}_{M_n}$ be
normal such that $\Phi$ and $\Phi^*$ are of bounded type of the form
$$
\Phi_+= \Theta_1 \Theta_0 A^*\quad \hbox{and}\quad \Phi_- =\Theta_1
B^*,
$$
where $\Theta_i=I_{\theta_i}$ for an inner function $\theta_i$
\hbox{\rm ($i=0,1$)} and $A,B\in H^2_{M_n}$. \  Write
\begin{equation}\label{3.10}
\begin{cases}
V: L\equiv L^2(\mu_B) \times L^2(\mu_s)
         \times L^2(\mu_{\Delta})\to \mathcal H(\theta_1 \theta_0)\ \
            \hbox{\rm is unitary as in (3.8)};\\
M:= V^*U_{\theta_1\theta_0}V;\\
\mathcal L :=L \otimes \mathbb C^n\\
\mathcal V := V \otimes I_n.
\end{cases}
\end{equation}
If $K \in \mathcal C(\Phi)$ then
$$
[T_{\Phi}^*, \ T_{\Phi}]=(T_A)_{\Theta_1\Theta_0}^*\mathcal V\label{V}
\Bigl(I|_{\mathcal L\label{L}}-K(M)^* K(M) \Bigr)\mathcal
V^*(T_A)_{\Theta_1\Theta_0}\ \bigoplus\ 0|_{\Theta_1 \Theta_0
H^2_{\mathbb C^n}}\,,
$$
where  $K(M)$ is understood as an $H^\infty$-functional calculus.\
Hence, in particular,
$$
K(M)\label{KM}\ \hbox{is contractive}\ \Longrightarrow\ T_\Phi\ \hbox{is
hyponormal;}
$$
the converse is also true if $(T_A)_{\Theta_1\Theta_0}$ has dense
range, and in this case,
$$
\hbox{\rm rank}\,[T_{\Phi}^*, \ T_{\Phi}]= \hbox{\rm
rank}\,\Bigl(I|_{\mathcal L}-K(M)^* K(M) \Bigr).
$$
\end{theorem}

\begin{proof} Let $E, F\in \mathcal H (\Theta_1 \Theta_0)$ and $K \in
\mathcal C(\Phi)$. \  Since $\text{ker}\,H_{\Theta_1^* \Theta_0^*}=
\Theta_1 \Theta_0 H^2_{\mathbb C^n}$, we have
$$
H_{\Theta_1^* \Theta_0^*K}E = H_{\Theta_1^* \Theta_0^*}(P_{\mathcal
H (\Theta_1 \Theta_0)}(KE)).
$$
Since $H_{\Theta_1^* \Theta_0^*}^*H_{\Theta_1^*\Theta_0^*}$ is the
projection onto $\mathcal H(\Theta_1 \Theta_0)$, it follows that
$$
\begin{aligned}
\Bigl \langle H_{\Theta_1^*\Theta_0^* K}^* H_{\Theta_1^*\Theta_0^*
K}E,\  F \Bigr \rangle &= \Bigl \langle P_{\mathcal H (\Theta_1
\Theta_0)}KE,\
      P_{\mathcal H(\Theta_1 \Theta_0)}KF \Bigr\rangle\\
&= \Bigl \langle (T_K)_{\Theta_1\Theta_0}E,\
(T_K)_{\Theta_1\Theta_0}F \Bigr\rangle,
\end{aligned}
$$
which gives
$$
H_{\Theta_1^*\Theta_0^* K}^* H_{ \Theta_1^* \Theta_0^* K} \vert_{
\mathcal H ( \Theta_1 \Theta_0)} =  (T_K)_{\Theta_1\Theta_0}^*
(T_K)_{\Theta_1\Theta_0}.
$$
Observe that $[T_\Phi^* , T_\Phi]=H_{A\Theta_0^*\Theta_1^*}^*
H_{A\Theta_0^*\Theta_1^*} - H_{B\Theta_1^*}^*H_{B\Theta_1^*}$
because $\Phi$ is normal. \  But since
$$
\text{cl}\,\text{ran}\,(H_{A\Theta_0^*\Theta_1^*}^*
H_{A\Theta_0^*\Theta_1^*}) = \left(\text{ker}\,
H_{A\Theta_0^*\Theta_1^*}\right)^\perp \subseteq (\Theta_1 \Theta_0
H^2_{\mathbb C^n})^{\perp} =\mathcal H(\Theta_1\Theta_0)
$$
and
$$
\text{cl}\,\text{ran}\,(H_{B\Theta_1^*}^*H_{B\Theta_1^*}) =
\left(\text{ker}\,H_{B\Theta_1^*}\right)^\perp \subseteq (\Theta_1
H^2_{\mathbb C^n})^{\perp}=\mathcal H(\Theta_1),
$$
we have $\text{\rm ran}\,[T_\Phi^* , T_\Phi]\subseteq
\mathcal{H}(\Theta_1\Theta_0)$. \ But since $\Phi- K \Phi^* \in
H^{\infty}_{M_n}$, it follows that on $\mathcal H (\Theta_1
\Theta_0)$,
$$
\aligned ~[T_{\Phi}^*, T_{\Phi}]
&=\left.\Bigl(H_{\Phi^*}^*H_{\Phi^*}-H_{\Phi}^*H_{\Phi}\Bigr)\right|_{\mathcal
H
(\Theta_1 \Theta_0)}\\
&= \left.\Bigl(H_{\Phi^*}^*H_{\Phi^*}- H_{K \Phi^*}^*
H_{K\Phi^*}\Bigr)\right|_{\mathcal H (\Theta_1 \Theta_0)}\\
&=\left.\Bigl(H_{\Theta_1^* \Theta_0^* A+\Phi_-}^*H_{\Theta_1^*
\Theta_0^*A+\Phi_-} - H_{K(\Theta_1^* \Theta_0^* A+\Phi_-)}^*H_{K(
\Theta_1^* \Theta_0^* A +\Phi_-)}\Bigr)\right|_{\mathcal H
(\Theta_1 \Theta_0)}\\
&=\left.T_{(A+\Theta_1\Theta_0\Phi_-)}^* \Bigl(H_{\Theta_1^*
\Theta_0^*}^*H_{\Theta_1^* \Theta_0^*} - H_{K \Theta_1^*
\Theta_0^*}^*H_{K \Theta_1^*
\Theta_0^*}\Bigr)T_{(A+\Theta_1\Theta_0\Phi_-)}\right|_{\mathcal H
(\Theta_1 \Theta_0)} \\
&=(T_A)_{\Theta_1\Theta_0}^* \Bigl(I|_{\mathcal H (\Theta_1
\Theta_0)}-(T_K)_{\Theta_1 \Theta_0}^* (T_K)_{\Theta_1
\Theta_0}\Bigr) (T_A)_{\Theta_1\Theta_0} ,
\endaligned
$$
where $(T_A)_{\Theta_1\Theta_0}$ is understood in the sense that the
compression $(T_A)_{\Theta_1\Theta_0}$ is bounded even though $T_A$
is possibly unbounded; in fact,
$$
P_{\mathcal H (\Theta_1 \Theta_0)}T_{(A+\Theta_1\Theta_0 \Phi_-)}
\vert_{\mathcal H (\Theta_1\Theta_0)} =P_{\mathcal H (\Theta_1
\Theta_0)}T_{A} \vert_{\mathcal H (\Theta_1\Theta_0)}.
$$
On the other hand, since $K(z)\equiv \begin{bmatrix}
k_{rs}(z)\end{bmatrix}_{1\le r,s\le n} \in H^{\infty}_{M_n}$, we may
write
$$
K(z)=\sum_{i=0}^{\infty} K_i z^i \qquad (K_i \in M_n).
$$
We also write $k_{rs}(z):=\sum_{0}^\infty c_i^{(rs)} z^i$ and then
$K_i=\begin{bmatrix} c_i^{(rs)}\end{bmatrix}_{1\le r,s\le n}$. \ We
thus have

$$
\aligned (T_K)_{\Theta_1 \Theta_0}
&=P_{\mathcal H (\Theta_1 \Theta_0)}T_K\vert_{\mathcal H (\Theta_1 \Theta_0)}\\
&=\begin{bmatrix} P_{\mathcal H(\theta_1\theta_0)}T_{k_{rs}}
                   \vert_{\mathcal H(\theta_1\theta_0)}\end{bmatrix}_{1\le r,s\le n}\\
&=\left[\sum_{i=0}^{\infty}c^{(rs)}_i P_{\mathcal
H(\theta_1\theta_0)}T_{z^i}
                   \vert_{\mathcal H(\theta_1\theta_0)}\right]_{1\le r,s\le n}\\
&=\left[\sum_{i=0}^{\infty}c^{(rs)}_i \Bigl(P_{\mathcal
H(\theta_1\theta_0)}T_{z}
                \vert_{\mathcal H(\theta_1\theta_0)}\Bigr)^i\right]_{1\le r,s\le n}
                \ \ \hbox{(because $\theta_1\theta_0 H^2\subseteq \hbox{\rm Lat}\,T_z$)}\\
&=\sum_{i=0}^\infty \left(U_{\theta_1\theta_0}^i\, \otimes\,
\begin{bmatrix} c_i^{(rs)}
                          \end{bmatrix}_{1\le r,s\le n}\right)\\
&=\sum_{i=0}^{\infty} \left( U_{\theta_1\theta_0}^i\, \otimes\,
K_i\right),
\endaligned
$$
Let $\{\phi_j\}$ be an orthonormal basis for $\mathcal H (\theta_1
\theta_0)$ and put $e_j:=V^* \phi_j$. \  Then $\{e_j\}$ forms an
orthonormal basis for $L^2(\mu_B) \times L^2(\mu_s) \times
L^2(\mu_{\Delta})$. \  Thus for each $f \in \mathbb C^n$, we have
$\mathcal V (e_j \otimes f)= \phi_j \otimes f$. \  It thus follows
that
$$
\aligned \Bigl \langle (T_K)_{\Theta_1\Theta_0} (\phi_j  \otimes
f)\, , \ \phi_k \otimes g \Bigr \rangle &= \sum_{i=0}^{\infty} \Bigl
\langle (U_{\theta_1 \theta_0}^i \otimes K_i)
 (\phi_j \otimes f), \  \phi_k\otimes g \Bigr \rangle \\
&= \sum_{i=0}^{\infty} \Bigl \langle
(U_{\theta_1 \theta_0}^i \phi_j)\otimes (K_i f) , \ \phi_k \otimes g \Bigr \rangle \\
&=\sum_{i=0}^{\infty} \Bigl \langle U_{\theta_1 \theta_0}^i V e_j ,
\ V e_k \Bigr \rangle
               \Bigl \langle K_i f, g \Bigr \rangle \\
&=\sum_{i=0}^{\infty} \Bigl \langle (M^i \otimes K_i) (e_j \otimes
f),
                           \ e_k \otimes g \Bigr \rangle \\
&=\Bigl \langle K(M) (e_j \otimes f)\, , \ e_k \otimes g \Bigr
\rangle,
\endaligned
$$
which implies that
\begin{equation}\label{3.11}
\mathcal V^*(T_K)_{\Theta_1 \Theta_0}\mathcal V =K(M). \
\end{equation}
Here $K(M)$ is understood as a $H^\infty$-functional calculus (so
called the Sz.-Nagy-Foia\c s functional calculus) because $M$ is an
absolutely continuous contraction: in fact, we claim that
$$
\hbox{every compression of the shift operator is completely
non-unitary.}
$$
To see this, write $P_{\mathcal X}UP_{\mathcal X}$ for the
compression of $U$ to $\mathcal X\equiv P_{\mathcal X}H^2$ with some
projection $P_{\mathcal X}$. \  We assume to the contrary that
$P_{\mathcal X}UP_{\mathcal X}$ has a unitary summand $W$ acting on
a closed subspace $\mathcal Y\subseteq \mathcal X$. \  But since
$||U||=1$, we must have that $P_{\mathcal Y^\perp}U\vert_{\mathcal
Y}=0$. \  Thus we can see that $\mathcal Y$ is an invariant subspace
of $U$. \  Thus, by Beurling's Theorem, $\mathcal Y=\theta H^2$ for
some inner function $\theta$. \  But then $W(\theta H^2)=z\theta
H^2$, and hence $W$ is not surjective, a contradiction. \  Hence
every compression of the shift operator is completely non-unitary. \
We can therefore conclude that
$$
[T_{\Phi}^*, \ T_{\Phi}]=(T_A)_{\Theta_1 \Theta_0}^*\mathcal V
\Bigl(I|_{\mathcal L}-K(M)^* K(M) \Bigr)\mathcal V^* (T_A)_{\Theta_1
\Theta_0}\ \bigoplus\ 0|_{\Theta_1 \Theta_0 H^2_{\mathbb C^n}}.
$$
The remaining assertions follow trivially from the first assertion.
\
\end{proof}

\medskip

If $\Phi$ is a scalar-valued function then Theorem \ref{thm3.3}
reduces to the following corollary. \

\begin{cor}\label{lem3.4}\label{cor3.4}
Let $\varphi\equiv \varphi_-^*+\varphi_+\in L^\infty$ be such that
$\varphi$ and $\overline\varphi$ are of bounded type of the form
$$
\varphi_+=\theta_1\theta_0\overline a\quad\hbox{and}\quad
\varphi_-=\theta_1\overline b,
$$
where $\theta_1$ and $\theta_0$ are inner functions and $a,b\in
H^2$. \  If $k\in\mathcal{C}(\varphi)$ then
$$
T_\varphi\ \hbox{is hyponormal}\ \Longleftrightarrow\ k(M)\ \hbox{is
contractive},
$$
where $M$ is defined as in $(\ref{3.10})$. \
\end{cor}

\begin{proof} By Theorem \ref{3.3}, it suffices to show that
$(T_a)_{\theta_1\theta_0}$ has dense range. \  To prove this suppose
$(T_a)_{\theta_1\theta_0}^* f=0$ for some $f\in
\mathcal{H}(\theta_1\theta_0)$. \  Then
$P_{\mathcal{H}(\theta_1\theta_0)}(\overline a f)=0$, i.e.,
$\overline a f=\theta_1\theta_0 h$ for some $h\in H^2$. \  Thus we
have $\overline a \overline {\theta}_1 \overline{\theta}_0 f\in
\left(H^2\right)^\perp \cap H^2=\{0\}$, which implies that $f=0$. \
Therefore $(T_a)_{\theta_1\theta_0}^*$ is 1-1, which gives the
result. \
\end{proof}

\medskip
\begin{remark}\label{remark3.5}We note that in Corollary
\ref{3.4}, if $\varphi$ is a rational function then $M$ is a finite
matrix. \  Indeed, if $\varphi$ is a rational function, and hence
$\theta_1\theta_0$ is a finite Blaschke product of the form
$$
\theta_1\theta_0=\prod_{j=1}^d
\frac{z-\alpha_j}{1-\overline{\alpha_j}z}\,,
$$
then $M$ is obtained by (\ref{2.14}). \qed
\end{remark}
\medskip

\begin{remark}\label{rem3.5} We mention that $K\in\mathcal C(\Phi)$ may not be
contractive, i.e., there might exist a function $K\in \mathcal
C(\Phi)\setminus \mathcal{E}(\Phi)$. \  In spite of this, Theorem
\ref{3.3} guarantees that
$$
I-K(M)^*K(M)
$$
is unchanged regardless of the particular choice of $K$ in
$\mathcal{C}(\Phi)$. \  We will illustrate this phenomenon with a
scalar-valued Toeplitz operator with a trigonometric polynomial
symbol. \  For example, let
$$
\Phi(z)=z^{-2}+2z^{-1}+z+2z^2.
$$
If $K(z)=\frac{1}{2}+\frac{3}{4} z$, then $\Phi_-^*-K\Phi_+^*
=(z^{-2}+2z^{-1})-(\frac{1}{2}+ \frac{3}{4}
z)(z^{-1}+2z^{-2})=-\frac{3}{4}\in H^\infty$, so that $K\in\mathcal
C(\Phi)$, but $||K||_\infty=\frac{5}{4}>1$. \  However by
(\ref{2.14}) we have
$$
M=\begin{bmatrix} 0&0\\ 1&0\end{bmatrix} \quad\hbox{and hence},\quad
K(M)=\begin{bmatrix} \frac{1}{2}&0\\ \frac{3}{4}&\frac{1}{2}
\end{bmatrix},
$$
so that
$$
I-K(M)^*K(M)=\begin{bmatrix} 1&0\\ 0&1\end{bmatrix}
- \begin{bmatrix} \frac{13}{16}&\frac{3}{8}\\
\frac{3}{8}&\frac{1}{4}\end{bmatrix}
= \begin{bmatrix} \frac{3}{16}&-\frac{3}{8}\\
-\frac{3}{8}&\frac{3}{4}\end{bmatrix} \ge 0,
$$
which implies that $T_\Phi$ is hyponormal even though
$||K||_\infty>1$. \  Of course, in view of Cowen's Theorem, there
exists a function $b\in\mathcal{E}(\Phi)$: indeed,
$b(z)=\frac{z+\frac{1}{2}}{1+\frac{1}{2}z}\in\mathcal{E}(\Phi)$. \qed
\end{remark}

\bigskip

We now provide some revealing examples that illustrate Theorem
\ref{thm3.3}. \

\begin{example}\label{ex3.8}
Let $\Delta$ be a singular inner function of the form
$$
\Delta:=\hbox{exp}\Bigl(\frac{z+1}{z-1}\Bigr)
$$
and consider the matrix-valued function
$$
\Phi:=\begin{bmatrix}
\overline{\Delta}&\overline{z\Delta}+z\Delta\\
\overline{z\Delta}+z\Delta&\overline{\Delta}\end{bmatrix}.
$$
We now use Theorem \ref{thm3.3} to determine the hyponormality of
$T_\Phi$. \  Under the notation of Theorem \ref{thm3.3} we have
$$
\Theta_1=I_{z\Delta},\ \ \Theta_0=I_2,\ \ A=\begin{bmatrix} 0&1\\
1&0\end{bmatrix},\ \ B=\begin{bmatrix} z&1\\ 1&z\end{bmatrix}.
$$
If we put
$$
K(z):=\begin{bmatrix}1&z\\z&1\end{bmatrix},
$$
then a straightforward calculation shows that $K\in \mathcal
C(\Phi)$. \  Under the notation of Theorem \ref{thm3.3} we can see
that
$$
(T_A)_{I_{z\Delta}}=\begin{bmatrix}0&I|_{L}\\I|_{L}&0\end{bmatrix}:\
\mathcal L\to \mathcal L\ \ \hbox{is invertible.}
$$
But since
$$
K(M)=\begin{bmatrix}I|_{L}& M\\ M&I|_{L}\end{bmatrix},
$$
it follows that
$$
I\vert_{\mathcal L} - K(M)^*K(M) =- \begin{bmatrix} M^* M& M+ M^*\\
M+ M^*& M^* M\end{bmatrix},
$$
which is not positive (simply by looking at the upper-left entry). \
It therefore follows from Theorem \ref{thm3.3} that $T_{\Phi}$ is
not hyponormal. \qed
\end{example}

\medskip

\begin{example}\label{ex3.9}
Let $\Delta$ be a singular inner function of the form
$$
\Delta:=\hbox{exp}\left(\frac{z+1}{z-1}\right)
$$
and let
$$
\Omega:=\Delta^{\frac{1}{2}}=\hbox{exp}\left(\frac{1}{2}\cdot\frac{z+1}{z-1}\right).
$$
Consider the function
$$
\varphi:=\frac{4}{5}\overline{z}+\frac{8}{5}\overline{\Delta}+\frac{37}{50}
\Bigl(z^2\overline{\Omega}-\frac{1}{\sqrt{e}}z^2+\frac{1}{\sqrt{e}}z\Bigr)+
\frac{29}{25}(z+2\Delta).
$$
Observe that
$$
\overline{\varphi_-}=
\frac{4}{5}\overline{z}+\frac{8}{5}\overline{\Delta}
+\frac{37}{50}(I-P)(z^2\overline{\Omega})+c \quad (c \in \mathbb C).
$$
We thus have
$$
k :=\frac{25}{29}\Bigl(\frac{4}{5}+\frac{37}{100}z^2\Omega \Bigr)\in
\mathcal C(\varphi),
$$
but
$$
||k||_{\infty}=\frac{117}{116} > 1.
$$
Thus by the aid of such a function $k$, we cannot determine the
hyponormality of $T_\varphi$. \  Using the notation in the
triangularization theorem we can show that

\medskip

\begin{itemize}
\item[(i)] $L^2(\mu_{\Delta})=L^2(0,1)$ and $L^2(\mu_{B})=\mathbb C$;
\item[(ii)] $M_{\Delta}=I$ and $M_B=0$;
\item[(iii)] $(J_{\Delta}c)(\lambda)=2\int_0^{\lambda}e^{t-\lambda}c(t)dt$ for
$\lambda \in (0,1)$ and $c \in L^2(0,1)$;
\item[(iv)]
$V_z=\frac{1}{\sqrt{2}}$ and $(V_{\Delta}c)(\zeta)=\sqrt{2}\int_0^1
c(\lambda)\hbox{exp}\bigl(-\lambda\frac{1+\zeta}{1-\zeta}\bigr)(1-\zeta)^{-1}d\lambda,
\quad \zeta \in \mathbb D,  \ \ c\in L^2(0,1)$;
\item[(v)] $M\equiv M_B \times M_{\Delta}+J: L\to L$, where
$L\equiv \mathbb C\oplus L^2(0,1)\cong \mathcal{H}(I_{z\Delta})$.
\end{itemize}
\medskip

\noindent Note that $\mathcal H(z\Delta)=\mathcal H(z) \oplus
z\mathcal H(\Delta)$, so that $P_{\mathcal H (z\Delta)}=P_{\mathcal
H(z)}+ zP_{\mathcal H(\Delta)}\overline{z}$. \  We then have
$$
U_{z\Delta}=\begin{bmatrix}U_{z}&0\\a&zU_{\Delta}\overline
z\end{bmatrix}\ :\ \begin{bmatrix}\mathcal H(z)\\ z\mathcal H
(\Delta)\end{bmatrix} \to
\begin{bmatrix}\mathcal H(z)\\ z\mathcal H (\Delta)\end{bmatrix}. \
$$
Since $U_z=0$, it follows that
$$
M=\begin{bmatrix}V_z&0\\0&zV_{\Delta}\end{bmatrix}^*\begin{bmatrix}0&0\\a&zU_{\Delta}\overline
z\end{bmatrix}\begin{bmatrix}V_z&0\\0&zV_{\Delta}\end{bmatrix}
=\begin{bmatrix}0&0\\
(zV_{\Delta})^*a V_z&I-J_{\Delta}\end{bmatrix}.
$$
By using the fact that $P_{\mathcal H(\theta)}=I-\theta P
\overline\theta$, we can compute:
$$
a=P_{z\mathcal H(\Delta)}(z\cdot1)=zP_{\mathcal H(\Delta)}\overline
z (z\cdot1) =zP_{\mathcal H (\Delta)}(1)=z(I-\Delta
P\overline{\Delta})(1)=z\left(1-\frac{\Delta}{e}\right).
$$
Thus we have
$$
M=\begin{bmatrix}0&0\\
V_{\Delta}^*\left(1-\frac{\Delta}{e}\right)V_z
&I-J_{\Delta}\end{bmatrix}.
$$
Write
$$
A:=V_{\Delta}^*\Bigl(1-\frac{\Delta}{e}\Bigr)V_z.
$$
Then we have
$$
\begin{aligned}
||A||^2 &=\Bigl|\Bigl|1-\frac{\Delta}{e}\Bigr|\Bigr|^2
=\Bigl|\Bigl| 1-\frac{\Delta(0)}{e}-\frac{1}{e}\Bigl(\Delta-\Delta(0)\Bigr)\Bigr|\Bigr|^2\\
&=\Bigl(1-\frac{\Delta(0)}{e}\Bigr)^2 +
\frac{1}{e^2}\left(||\Delta||^2-|\Delta(0)|^2\right)\\
&=\Bigl(1-\frac{1}{e^2}\Bigr)^2+\frac{1}{e^2}\Bigl(1-\frac{1}{e^2}\Bigr)=1-\frac{1}{e^2}.
\end{aligned}
$$
A straightforward calculation shows that
$$
k(M)=\frac{25}{29}
\begin{bmatrix}\frac{4}{5}&0\\S&\Bigl(\frac{4}{5}
+\frac{37}{100}z^2\Omega\Bigr)(I-J_{\Delta})\end{bmatrix},
$$
where
$$
S:=\Biggl(\Bigl(\frac{37}{100}z\Omega\Bigr)(I-J_{\Delta})\Biggr)A.
$$
On the other hand, consider the function
$$
\varphi_1(z):=(I-P)(z\overline{\Omega})+\Delta.
$$
Then $q:=z\Omega\in \mathcal E(\varphi_1)$. \  Since $M_\Delta=1$,
it follows from Corollary \ref{cor3.4} that
$q(M)=(z\Omega)(I-J_{\Delta})$ is a contraction. \  Thus we have
$$
||S||=\frac{37}{100}||(z\Omega)(I-J_{\Delta})A|| \leq
\frac{37}{100}\sqrt{1-\frac{1}{e^2}}
$$
Also we consider the function
$$
\varphi_2(z):=\frac{4}{5}\overline{\Delta}+\frac{9}{25}
(I-P)(z^2\,\overline{\Omega})+\Delta.
$$
Put
$$
B(z):=\frac{z^2\Omega+\frac{4}{5}}{1+\frac{4}{5}z^2\Omega}.
$$
Since $B(z)=\frac{4}{5}+ \frac{9}{25}z^2\Omega +\Delta g$ for some
$g\in H^2$, we have
$$
\begin{aligned}
(\varphi_2)_-\, - \,B\,\overline{(\varphi_2)_+}
&=\frac{4}{5}\overline{\Delta}+\frac{9}{25}
(I-P)(z^2\,\overline{\Omega})
     - \left(\frac{4}{5}+\frac{9}{25}z^2\Omega +\Delta g\right)\overline{\Delta}\\
&=-\frac{9}{25} P(z^2\Omega)-g\in H^2,
\end{aligned}
$$
Since $||B||_\infty =1$, it follows that $T_{\varphi_2}$ is
hyponormal. \ In particular, since
$$
r:= \frac{4}{5}+\frac{9}{25}z^2\Omega \in \mathcal C(\varphi_2),
$$
it follows from Corollary \ref{cor3.4} that
$r(M)=(\frac{4}{5}+\frac{9}{25}z^2\Omega)(I-J_{\Delta})$ is a
contraction. \  Thus we have that
$$
\left|\left|\Bigl(\frac{4}{5}+\frac{37}{100}z^2\Omega\Bigr)(I-J_{\Delta})\right|\right|
\leq
\left|\left|\Bigl(\frac{4}{5}+\frac{9}{25}z^2\Omega\Bigr)(I-J_{\Delta})\right|\right|
+\frac{1}{100} ||(z^2\Omega)(I-J_{\Delta})|| \leq \frac{101}{100}.
$$
Using the observation that if $A,B,C$ and $D$ are operators then
$$
\left|\left|\begin{bmatrix} A&B\\ C&D
\end{bmatrix}\right|\right|\le \left|\left|\begin{bmatrix}
||A||&||B||\\ ||C||&||D|| \end{bmatrix}\right|\right|,
$$
we can see that
$$
\begin{aligned}
||k(M)||
&\leq
\frac{25}{29} \Biggl|\Biggl|\begin{bmatrix} \frac{4}{5}&0\\
||S||&
\left|\left|\Bigl(\frac{4}{5}+\frac{9}{25}z^2\Omega\Bigr)(I-J_{\Delta})
\right|\right|\end{bmatrix}\Biggl|\Biggl|\\
&\leq
\frac{25}{29} \Biggl|\Biggl|\begin{bmatrix} \frac{4}{5}&0\\
\frac{37}{100}\sqrt{1-\frac{1}{e^2}}&\frac{101}{100}\end{bmatrix}\Biggl|\Biggl|
\approx 0.968<1,
\end{aligned}
$$
which, by Corollary \ref{cor3.4}, implies that $T_{\varphi}$ is
hyponormal. \qed
\end{example}

\medskip

We now consider the condition ``$\mathcal C(\Phi)\not=\emptyset$",
i.e., the existence of a function $K\in H_{M_n}^\infty$ such that
$\Phi-K\Phi^*\in H_{M_n}^\infty$. \  In view of (\ref{3.9}), we may
assume that
\begin{equation}\label{3.14}
\hbox{ker}\,H_{\Phi_+^*} \subseteq \hbox{ker}\,H_{\Phi_-^*}
\end{equation}
whenever we study the hyponormality of $T_\Phi$. \ Recall (\cite
[Corollary 2]{Gu1}) that
\begin{equation} \label{3.15}
H_{\Phi^*}^* H_{\Phi^*} - H_{\Phi}^*H_\Phi\ge 0 \
\Longleftrightarrow\ \exists\, K\in H_{M_n}^\infty\ \hbox{with}\
||K||_\infty\le 1\ \ \hbox{such that}\ H_{\Phi_-^*}=T_{\widetilde
K}^*H_{\Phi_+^*}.
\end{equation}
We thus have
\begin{equation}\label{3.16}
\aligned \mathcal C(\Phi) \neq \emptyset
&\Longleftrightarrow \exists\, K\in H_{M_n}^\infty\ \hbox{such that}\ \Phi_-^*-K\Phi_+^*\in H_{M_n}^2\\
&\Longleftrightarrow H_{\Phi_-^*}=T_{\widetilde K}^*H_{\Phi_+^*}\ \ \hbox{for some}\ K\in H_{M_n}^\infty\\
&\Longleftrightarrow H_{\alpha\Phi_+^*}^* H_{\alpha\Phi_+^*} -
H_{\Phi_-^*}^*H_{\Phi_-^*}\ge 0\ \ \hbox{for some}\
               \alpha>0\ \ \hbox{(by (\ref{3.15}))}\\
&\Longleftrightarrow \hbox{ker}\,H_{\Phi_+^*} \subseteq
\hbox{ker}\,H_{\Phi_-^*}\ \
          \hbox{and}\\
          &\quad\qquad \hbox{sup}\Biggl \{\frac{||H_{\Phi_-^*}F||}{||H_{\Phi_+^*}F||} : F\in
                  \text{ker}\,(H_{\Phi_+^*})^{\perp}, \ ||F||=1 \Biggr \}\leq \alpha.
\endaligned
\end{equation}
If $\Phi \in L^{\infty}_{M_n}$ is a rational function then by
Kronecker's lemma (cf. \cite [p.183] {Ni}), $\hbox{ran}\,
H_{\Phi_+^*}$ is finite dimensional. \  Thus by (\ref{3.16}) we can
see that
$$
\mathcal C (\Phi)\ \neq\ \emptyset \Longleftrightarrow
\hbox{ker}\,H_{\Phi_+^*} \subseteq \hbox{ker}\,H_{\Phi_-^*}.
$$
Consequently, if $\Phi\in L_{M_n}^\infty$ is a rational function
then there always exists a function $K\in\mathcal C(\Phi)$ under the
kernel assumption (\ref{3.14}). \  We record this in

\begin{proposition} \label{pro3.10} If $\Phi \in L^{\infty}_{M_n}$ is a rational function
satisfying $\hbox{\rm ker}\,H_{\Phi_+^*} \subseteq \hbox{\rm
ker}\,H_{\Phi_-^*}$, then $\mathcal C(\Phi)\ne \emptyset$. \
\end{proposition}

\medskip

We remark that there is an explicit way to find a function (in fact,
a matrix-valued polynomial) $K$ in $\mathcal C(\Phi)$ for the
rational symbol case. \  To see this, in view of Proposition
\ref{pro3.2}, suppose $\Phi\in L^{\infty}_{M_n}$ is of the form
$$
\Phi_+= \Theta_1\Theta_0 A^*\quad\hbox{and}\quad \Phi_-= \Theta_1
B^*,
$$
where $\Theta_i=I_{\theta_i}$ for a finite Blaschke product
$\theta_i$ ($i=0,1$). \  We observe first that
\begin{equation}\label{3.17}
K \in \mathcal C (\Phi)\ \Longleftrightarrow\ \Phi-K\Phi^*\in
H^{\infty}_{M_n} \ \Longleftrightarrow\
 \Theta_0B-KA \in \Theta_1\Theta_0H^{\infty}_{M_n}.
\end{equation}
Suppose $\theta_1\theta_0$ is a finite Blaschke product of degree
$d$ of the form
$$
\theta_1\theta_0=\prod_{i=1}^N
\left(\frac{z-\alpha_i}{1-\overline{\alpha}_i z}\right)^{m_i} \qquad
(d:=\sum_{i=1}^N m_i).
$$
Then the last assertion in (\ref{3.17}) holds if and only if the
following equations hold: for each $i=1,\hdots, N$,
\begin{equation}\label{3.18}
\begin{bmatrix} B_{i,0}\\
B_{i,1}\\
B_{i,2}\\
\vdots\\
B_{i,m_i -1}
\end{bmatrix}
= \begin{bmatrix}
K_{i,0}&0&0&\cdots&0\\
K_{i,1}&K_{i,0}&0&\cdots&0\\
K_{i,2}&K_{i,1}&K_{i,0}&\cdots&0\\
\vdots&\ddots&\ddots&\ddots&\vdots \\
K_{i,m_i -1}&K_{i,m_i -2}&\hdots&K_{i,1}&K_{i,0}
\end{bmatrix}
\begin{bmatrix}
A_{i,0}\\
A_{i,1}\\
A_{i,2}\\
\vdots\\
A_{i,m_i -1}
\end{bmatrix},
\end{equation}
where
$$
K_{i,j}:= \frac{K^{(j)}(\alpha_i)}{j!},\quad A_{i,j}:=
\frac{A^{(j)}(\alpha_i)}{j!} \quad \text{and} \quad
B_{i,j}:=\frac{(\theta_0B)^{(j)}(\alpha_i)}{j!}.
$$
Thus $K$ is a function in $H^{\infty}_{M_n}$ for which
\begin{equation}\label{3.19}
\frac{K^{(j)}(\alpha_i)}{j!} = K_{i,j} \qquad (1 \leq i \leq N, \ 0
\leq j < m_i),
\end{equation}
where the $K_{i,j}$ are determined by the equation (\ref{3.18}). \
This is exactly the classical Hermite-Fej\' er interpolation problem
which we have introduced in Section 2. \  Thus the solution
(\ref{2.12}) for the classical Hermite-Fej\' er interpolation
problem provides a polynomial $K\in \mathcal C (\Phi)$. \

\medskip

Therefore we get:

\begin{proposition}\label{pro3.11}
If $\Phi \in L^{\infty}_{M_n}$ is a rational function such that
$\mathcal C(\Phi)\ne \emptyset$, then $\mathcal C(\Phi)$ contains a
polynomial. \
\end{proposition}

\medskip

However, by comparison with the rational symbol case, there may not
exist a function $K\in\mathcal C(\Phi)$ if $\Phi$ is of bounded
type. \  But we guarantee the existence of a function $K\in\mathcal
C(\Phi)$ if the bounded type symbol $\Phi$ satisfies a certain
determinant property. \  To see this, we recall the notion of the
reduced minimum modulus. \  If $T\in \mathcal{B(H)}$ then the {\it
reduced minimum modulus} of $T$ is defined by
$$
\gamma(T)\label{gammaT}=\begin{cases} \inf \bigl\{||Tx||: \ x \in
\bigl(\hbox{ker}\,T\bigr)^\perp, \ ||x||=1 \bigr\} \quad &\text{if} \ T \neq 0\\
\qquad\qquad  0&\text{if}\ T=0.
\end{cases}
$$
It is well known (\cite{Ap}) that if $T\ne 0$ then $\gamma(T)>0$ if
and only if $T$ has closed range. \  We can easily show that if
$S,T\in\mathcal{B(H)}$ and $S$ is one-one then
\begin{equation}\label{3.20}
 \gamma(ST)\ge \gamma(S)\gamma(T).
\end{equation}

We then have:

\begin{proposition}\label{pro3.12}
Let $\Phi \in L^{\infty}_{M_n}$ be such that $\Phi$ and $\Phi^*$ are
of bounded type satisfying
$$
\hbox{\rm ker}\, H_{\Phi_+^*} \subseteq \hbox{\rm
ker}\,H_{\Phi_-^*}.
$$
If there exists $\delta>0$ such that $\mathcal{M}:=\Bigl\{t:\
|\hbox{\rm det}\,\Phi_+(e^{it})|\ <\delta\Bigr\}$ has measure zero
then $\mathcal C(\Phi)\not=\emptyset$. \
\end{proposition}

\begin{proof} Suppose $\mathcal M$ has measure zero for some $\delta >0$. \
Write
$$
\Phi_+=\Theta A^*\quad\hbox{(right coprime factorization).}
$$
Since $\hbox{det}\,\Theta$ is inner, we have
$|\hbox{det}\,\Phi_+|=|\hbox{det}\,A|$ a.e. on $\mathbb T$. \  Then
by the well-known result \cite[Theorem XXIII.2.4]{GGK}, our
determinant condition shows that the multiplication operator $M_{A}$
is invertible and $\gamma(M_{A})>0$, where $M_A f:=Af$ for $f\in L^2_{\mathbb C^n}$. \
Since $A\in H^\infty_{M_n}$,
the Toeplitz operator $T_{A}$ is a restriction of $M_{A}$. \  Thus
it follows that $\gamma(T_{A}) \geq \gamma(M_{A})>0$. \  Since
$\hbox{ran}\,H_{\Theta^*}=\mathcal H(\widetilde{\Theta})$, it
follows that
$$
H_{\Phi_+^*}=H_{A \Theta^*}=
T_{\widetilde{A}}^*H_{\Theta^*}=T_{\widetilde{A}}^*|_{\mathcal H
(\widetilde{\Theta})}H_{\Theta^*}.
$$
Observe that
$$
\begin{aligned}
\gamma(H_{\Theta^*}) &= \inf\, \Bigl\{ ||H_{\Theta^*}F||: F \in
   \mathcal{H}(\Theta), ||F||=1 \Bigr\} \\
&= \inf\, \Bigl\{ ||\Theta^*F||:
F \in \mathcal{H}(\Theta) , ||F||=1 \Bigr\}\\
&=1.
\end{aligned}
$$
We now claim that
\begin{equation}\label{3.21}
T_{\widetilde{A}}^*|_{\mathcal H (\widetilde{\Theta})}\ \hbox{is
one-one}. \
\end{equation}
Indeed, since
$$
\Theta H^2_{\mathbb C^n}=\hbox{ker}\, H_{A\Theta^*}=\hbox{ker}\,
T_{\widetilde A}^*H_{\Theta^*}\quad\hbox{and}\quad \hbox{ker}\,
H_{\Theta^*}=\Theta H^2_{\mathbb C^n}\,,
$$
it follows that $T_{\widetilde A}^*\vert_{\hbox{ran}\,H_{\Theta^*}}
=T_{\widetilde A}^*\vert_{\mathcal{H}(\widetilde\Theta)}$ is
one-one, which gives (\ref{3.21}). \  Now since $\gamma(T_A)>0$ it
follows from (\ref{3.20}) and (\ref{3.21}) that
$$
\gamma(H_{\Phi^*})=\gamma(H_{\Phi_+^*})=\gamma(T_{\widetilde{A}}^*|_{\mathcal
H (\widetilde{\Theta})}H_{\Theta^*}) \geq
\gamma(T_{\widetilde{A}}^*|_{\mathcal H (\widetilde{\Theta})})\geq
\gamma(T_{\widetilde{A}}^*)= \gamma(T_{\widetilde{A}})=\gamma(T_{A})
>0.
$$
We thus have
\begin{equation}\label{3.22}
\hbox{sup}\Biggl \{ \frac{||H_{\Phi_-^*}F||}{||H_{\Phi_+^*}F||} : \
F\in \bigl(\hbox{ker}\,H_{\Phi_+^*}\bigr)^\perp, \ ||F||=1 \Biggr \}
\le \frac{||H_\Phi||}{\gamma(H_{\Phi^*})}<\alpha\quad\hbox{for
some}\ \alpha>0.
\end{equation}
Therefore, by (\ref{3.16}) we can conclude that $\mathcal C(\Phi)\ne
\emptyset$. \  \end{proof}

\vskip 1cm

%
%
%
%

\section{Subnormality of Block Toeplitz Operators}

%
%
%
%


As we saw in Introduction, the Bram-Halmos criterion on subnormality
(\cite{Br}, \cite{Con}) says that $T\in\mathcal{B(H)}$ is subnormal
if and only if the positive test (\ref{1.6}) holds. \  It is easy to
see that (\ref{1.6}) is equivalent to the following positivity test:
\begin{equation}\label{4.1}
\begin{bmatrix}
I& T^*& \hdots& T^{*k}\\
T& T^*T&  \hdots & T^{*k}T\\
\vdots & \vdots & \ddots & \vdots\\
T^k & T^*T^k & \hdots & T ^{*k}T^k
\end{bmatrix}\ge
0\quad\text{(all $k\ge 1$)}.
\end{equation}
Condition (\ref{4.1}) provides a measure of the gap between
hyponormality and subnormality. \  In fact the positivity condition
(\ref{4.1}) for $k=1$ is equivalent to the hyponormality of $T$,
while subnormality requires the validity of (\ref{4.1}) for all $k$.
\ For $k\ge 1$, an operator $T$ is said to be {\it $k$-hyponormal}\label{khypo}
if $T$ satisfies the positivity condition (\ref{4.1}) for a fixed
$k$. \  Thus the Bram-Halmos criterion can be stated as: $T$ is
subnormal if and only if $T$ is $k$-hyponormal for all $k\ge 1$. \
The $k$-hyponormality has been considered by many authors with an
aim at understanding the gap between hyponormality and subnormality.
\ For instance, the Bram-Halmos criterion on subnormality indicates
that $2$-hyponormality is generally far from subnormality. \  There
are special classes of operators, however, for which these two
notions are equivalent. \  A trivial example is given by the class
of operators whose square is compact (e.g., compact perturbations of
nilpotent operators of nilpotency 2). \  Also in \cite [Example 3.1]
{CuL1}, it was shown that there is no gap between 2-hyponormality
and subnormality for back-step extensions of recursively generated
subnormal weighted shifts. \

On the other hand, in 1970, P.R. Halmos posed the following problem,
listed as Problem 5 in his lectures ``Ten problems in Hilbert space"
\cite{Hal1}, \cite{Hal2}:
$$
\hbox{Is every subnormal Toeplitz operator either normal or
analytic\,?}
$$
A Toeplitz operator $T_\varphi$ is called {\it analytic} if
$\varphi\in H^\infty$. \  Any analytic Toeplitz operator is easily
seen to be subnormal: indeed, $T_\varphi h=P(\varphi h)=\varphi h
=M_\varphi h$ for $h\in H^2$, where $M_\varphi\label{Mphi}$ is the normal
operator of multiplication by $\varphi$ on $L^2$. \  The question is
natural because the two classes, the normal and analytic Toeplitz
operators, are fairly well understood and are subnormal. \  Halmos's
Problem 5 has been partially answered in the affirmative by many
authors (cf. \cite{Ab}, \cite{AIW}, \cite{Co2}, \cite{Co3},
\cite{CuL1}, \cite{CuL2}, \cite{NT}, and etc). \  In 1984, Halmos's
Problem 5 was answered in the negative by C. Cowen and J. Long
\cite{CoL}: they found an analytic function $\psi$ for which
$T_{\psi+\alpha\overline\psi}$ ($0<\alpha<1$) is subnormal - in
fact, this Toeplitz operator is unitarily equivalent to a subnormal
weighted shift $W_\beta$ with weight sequence
$\beta\equiv\{\beta_n\}$, where
$\beta_n=(1-\alpha^{2n+2})^{\frac{1}{2}}$ for $n=0,1,2,\hdots$. \
Unfortunately, Cowen and Long's construction does not provide an
intrinsic connection between subnormality and the theory of Toeplitz
operators. \  Until now researchers have been unable to characterize
subnormal Toeplitz operators in terms of their symbols. \  Thus the
following question is very interesting and challenging:
\begin{equation}\label{4.1*}
\hbox{\rm Which Toeplitz operators are subnormal\,?}
\end{equation}
The most interesting partial answer to Halmos's Problem 5 was given
by M. Abrahamse \cite{Ab}. \  M. Abrahamse gave a general sufficient
condition for the answer to Halmos's Problem 5 to be affirmative. \
Abrahamse's Theorem can be then stated as follows: {\it Let
$\varphi=\overline{g}+f\in L^\infty$ ($f,g\in H^2$) be such that
$\varphi$ or $\overline \varphi$ is of bounded type. \  If
$T_\varphi$ is subnormal then $T_\varphi$ is normal or analytic. \ }
In fact, it was also shown (cf. \cite{CuL2}, \cite{CuL3}) that every
2-hyponormal Toeplitz operator with a bounded type symbol is normal
or analytic, and hence subnormal. \  On the other hand, very
recently, the authors of \cite{CHL} have extended Abrahamse's
Theorem to block Toeplitz operators. \

\begin{theorem}
\rm (Extension of Abrahamse's Theorem) (Curto-Hwang-Lee
\cite{CHL})\label{thm4.1}
{\it Suppose $\Phi=\Phi_-^*+ \Phi_+\in
L^\infty_{M_n}$ is such that $\Phi$ and $\Phi^*$ are of bounded type
of the form
$$
\Phi_- = B^*\Theta\quad (B\in H^{2}_{M_n};\ \Theta=I_\theta\
          \hbox{with an inner function}\ \theta)\,,
$$
where $B$ and $\Theta$ are coprime. \
If $T_\Phi$ is hyponormal and $\hbox{ker}\, [T_\Phi^*, T_\Phi]$ is invariant under
$T_\Phi$ then $T_{\Phi}$ is normal or analytic. \
Hence, in particular, if $T_\Phi$ is subnormal then
$T_{\Phi}$ is normal or analytic.
} \
\end{theorem}

\medskip

We note that if $n=1$ (i.e., $T_\Phi$ is a scalar-valued Toeplitz
operator), then $\Phi_-=\overline{b}\theta$ with $b\in H^2$. \  Thus, it automatically holds that $b$ and $\theta$ are coprime. \  Consequently,
if $n=1$ then Theorem \ref{thm4.1} reduces to Abrahamse's Theorem.

On the other hand,
the study of square-hyponormality originated in \cite
[Problem 209]{Hal3}. \ It is easy to see that every power of a normal
operator is normal and the same statement is true for every
subnormal operator. \ How about hyponormal operators? \ \cite[Problem 209]{Hal3} shows that there exists a
hyponormal operator whose
square is not hyponormal (e.g., $U^*+2U$ for the unilateral shift
$U$). \ However, as we remarked in the preceding, there exist special
classes of operators for which square-hyponormality and subnormality
coincide. \ For those classes of operators, it suffices to check the
square-hyponormality to show subnormality. \ This certainly gives
a nice answer to question (\ref{4.1*}). \ Indeed, in \cite{CuL1},
it was shown that every hyponormal trigonometric Toeplitz operator
whose square is hyponormal must be either normal or analytic, and
hence subnormal. \ In \cite{Gu3}, C. Gu showed that this
result still holds for Toeplitz operators $T_\varphi$ with rational
symbols $\varphi$ (more generally, the cases where both $\varphi$ and
$\overline\varphi$ are of bounded type). \

The aim of this section is to prove that this result can be extended
to the block Toeplitz operators whose symbols are matrix-valued
rational functions. \

\medskip

We begin with:

\begin{lemma}\label{lem4.2}
Suppose $\Phi=\Phi_-^*+ \Phi_+\in L^\infty_{M_n}$ is a
matrix-valued rational function of the form
$$
\Phi_- = B^*\Theta \quad (\hbox{coprime factorization})\quad
\hbox{and}\quad \Phi_+=\Theta\Theta_0A^*,
$$
where $\Theta =I_{\theta}$ and $\Theta_0=I_{\theta_0}$ with finite
Blaschke products $\theta, \theta_0$ and $A, B\in H^{2}_{M_n}$. \ If
$T_{\Phi}$ is hyponormal then $A(\alpha)$ is invertible for each
$\alpha \in \mathcal Z(\theta)\setminus \mathcal Z(\theta_0)$. \
\end{lemma}

\begin{proof} Assume to the contrary that $A(\alpha)$ is not invertible for
some $\alpha \in \mathcal Z(\theta)\setminus \mathcal Z(\theta_0)$.
\ Then by Lemma \ref{lem2.2}, $A$ and $B_{\alpha}:= I_{b_{\alpha}}$
are not right coprime. \  Thus there exists a nonconstant inner
matrix function $\Delta$ such that
$$
B_{\alpha}=\Delta_1\Delta=\Delta\Delta_1\quad\hbox{and}\quad
A=A_1\Delta.
$$
Write $\Theta:=B_{\alpha} \Theta^{\prime} = \Theta^{\prime}
B_{\alpha}$. \  Then we may write
$\Phi_+=\Theta_0\Theta^{\prime}\Delta_1 A_1^*$. \  Since $T_{\Phi}$
is hyponormal, it follows that
$$
{\Theta}_0{\Theta}^{\prime} {\Delta_1} H_{\mathbb C^n}^2 \subseteq
\hbox{ker}\, H_{\Phi_+^*}\subseteq \hbox{ker}\, H_{\Phi_-^*}
={\Theta}H_{\mathbb C^n}^2,
$$
which implies that ${\Theta}$ is a (left) inner divisor of
${\Theta}_0{\Theta}^{\prime} {\Delta_1}$ (cf. \cite [Corollary
IX.2.2] {FF}). \  Observe that
$$
\aligned {\Theta}\ \ \hbox{is a (left) inner divisor of} \
{\Theta}_0{\Theta}^{\prime} {\Delta_1}
&\Longrightarrow {\Theta}^* {\Theta}_0{\Theta}^{\prime} {\Delta_1}\in H_{M_n}^2\\
&\Longrightarrow \Theta_0\Delta_1\Theta^{\prime} \Theta^* \in H_{M_n}^2\\
&\Longrightarrow \Theta_0\Delta^* \in H_{M_n}^2,
\endaligned
$$
which implies that $\Delta$ is a (right) inner divisor of
$\Theta_0$. \  But since $B_{\alpha}$ and $\Theta_0$ are coprime, it
follows that $\Delta$ and $\Theta_0$ are coprime. \  Thus $\Delta$
is a constant unitary, a contradiction. \  This completes the proof.
\
\end{proof}

\medskip

\begin{lemma}\label{lem4.3}
Suppose $F, B_{\lambda} \in H_{M_n}^{\infty}$ \hbox{\rm
($B_{\lambda}:=I_{b_\lambda}$).} \ If $G=\hbox{\rm GCD}_{\ell}\{F,
B_{\lambda}\}$, then $G$ is a Blaschke-Potapov factor of the form $
G=b_\lambda P_N +(I-P_N)$ with
$$
N:=\bigl(\hbox{\rm ran}\,{F}(\lambda)\bigr)^{\perp}.
$$
\end{lemma}

\begin{proof} By assumption, $\widetilde{G}=\hbox{GCD}_{r}\{\widetilde{F},
\widetilde{B}_{\lambda}\}$. \  Then by Lemma \ref{lemma2.3},
$$
\widetilde{G}=b_{\overline\lambda} P_N +(I-P_N)\quad \hbox{for a
closed subspace $N$}\,.
$$
Thus $\widetilde F=\widetilde L\widetilde G$ for some $\widetilde L\in
H^2_{M_n}$, where $\widetilde L$ and $\widetilde B_{\lambda} \widetilde
G^*=P_N+b_{\overline\lambda}(I-P_N)$ are right coprime. \  We argue
that $\hbox{ker}\,\widetilde
L(\overline\lambda)\,\cap\,\hbox{ran}\,(I-P_N)=\{0\}$. \  Indeed, if
$\hbox{ker}\,\widetilde
L(\overline\lambda)\,\cap\,\hbox{ran}\,(I-P_N)=:N_0\ne\{0\}$ then
$P_{N_0^\perp}+b_{\overline\lambda}(I-P_{N_0^\perp})$ would be a
right inner divisor of $\widetilde L$ and
$P_N+b_{\overline\lambda}(I-P_N)$ as follows:
$$
\widetilde L(\overline\lambda)=\begin{bmatrix} \ast&0\\
\ast&0\end{bmatrix} \begin{matrix} N_0^\perp\\ N_0 \end{matrix}
=\begin{bmatrix} \ast&0\\
\ast&0\end{bmatrix} \begin{bmatrix} 1&0\\
0&b_{\overline\lambda}\end{bmatrix} \begin{matrix} N_0^\perp\\
N_0\end{matrix}\,,
$$
and hence,
$$
\widetilde{L}=\widetilde{L}(\overline\lambda)+CB_{\overline\lambda}
=D\begin{bmatrix} 1&0\\
0&b_{\overline\lambda}\end{bmatrix} \begin{matrix} N_0^\perp\\
N_0\end{matrix}\quad\hbox{(some $C,D\in H^2_{M_n}$)}
$$
and
$$
P_N+b_{\overline\lambda}(I-P_N)=
\begin{bmatrix}
1&0&0\\
0&b_{\overline\lambda}&0\\
0&0&b_{\overline\lambda}
\end{bmatrix} \begin{matrix} N\\ N^\prime\\
N_0\end{matrix} =
\begin{bmatrix}
1&0&0\\
0&b_{\overline\lambda}&0\\
0&0&1\end{bmatrix} \begin{bmatrix}
1&0&0\\
0&1&0\\
0&0&b_{\overline\lambda}
\end{bmatrix}\begin{matrix} N\\ N^\prime\\
N_0\end{matrix}\,,
$$
where $N^\prime:=N^\perp\ominus N_0$. \  But since $\widetilde
F(\overline\lambda)=\widetilde L(\overline\lambda)
\Bigl(b_{\overline\lambda}P_N+(I-P_N)\Bigr)(\overline\lambda)
=\widetilde L(\overline\lambda)(I-P_N)$, it follows that
$$
N=\hbox{ker}\,(I-P_N)=\hbox{ker}\,
\widetilde{F}(\overline{\lambda})=\hbox{ker}\,
F^*(\lambda)=\bigl(\hbox{ran}\, F(\lambda)\bigr)^{\perp}.
$$
Therefore $G=b_\lambda P_N+(I-P_N)$ with $N:=\bigl(\hbox{\rm
ran}\,{F}(\lambda)\bigr)^{\perp}$. \
\end{proof}

\medskip

\begin{lemma} \label{lem4.4}
Let $\Phi\equiv\Phi_+\in H^\infty_{M_n}$ be a matrix-valued
rational function of the form
$$
\aligned
\Phi &=\Theta\Delta_r A_r^*,  \quad (\hbox{right coprime factorization}),\\
&=A_{\ell}^*\Omega,  \qquad (\hbox{left coprime factorization}),
\endaligned
$$
where $\Theta =I_\theta$ with a finite Blaschke product $\theta$ and
$\Delta_r, \Omega$ are inner matrix functions. \  Then $\Theta$ is
an inner divisor of $\Omega$. \
\end{lemma}

\begin{proof} Since $\Delta_r$ is a finite Blaschke-Potapov product, we may
write
$$
\Delta_r=\nu \prod_{m=1}^M \Bigl(b_m P_m + (I-P_m)\Bigr) \quad
(b_m:=\frac{z-\alpha_m}{1-\overline \alpha_m z}).
$$
Without loss of generality we may assume that $\nu=I_n$. \  Define
$$
\theta_0:=\hbox{GCD}\,\Bigl\{\omega:\ \omega\ \hbox{is inner},\
\Delta_r \ \hbox{is an inner divisor of}\ \Omega=\omega I_n\Bigr\}.
$$
Then $\theta_0=\prod_{m=1}^M b_m$. \ Observe that
$$
\aligned \Phi
&=\Theta\Delta_r A_r^*\\
&=\Theta\prod_{m=1}^M \Bigl(b_mP_m + (I-P_m)\Bigr) A_r^*\\
&=\prod_{m=1}^{M-1} \Bigl(b_mP_m + (I-P_m)\Bigr)B_M \Bigl(P_M +
b_M(I-P_M)\Bigr)^*A_r^* \Theta \quad
(B_m:=I_{b_m})\\
&=\prod_{m=1}^{M-1} \Bigl(b_mP_m + (I-P_m)\Bigr)
\Bigl[A_r\Bigl(P_M + b_M(I-P_M)\Bigr)\Bigr]^*  B_M\Theta\\
\endaligned
$$
If $P_M=I$, then
$$
\Phi=\prod_{m=1}^{M-1} \Bigl(b_mP_m + (I-P_m)\Bigr) A_r^* B_M\Theta,
$$
where $\Theta$ and $A_r$ are coprime. \  If instead $P_M\ne I$, then
there are two cases to consider. \

\medskip

Case 1: Let $\alpha_M \notin \mathcal Z(\theta)$. \  Then
$$
\Phi =\prod_{m=1}^{M-1} \Bigl(b_mP_m + (I-P_m)\Bigr) A_1^* B_M\Theta
\quad(\hbox{with}\ A_1:=A_r(P_M+b_M(I-P_M))),
$$
where $\Theta$ and $A_1$ are coprime (by passing to Lemma
\ref{lem2.2}). \

\medskip

Case 2: Let $\alpha_M \in \mathcal Z(\theta)$. \  Write
$\Omega_M:=\hbox{GCD}_{\ell}\{B_M,\ A_r\bigl(P_M + b_M(I-P_M)\bigr)\}$.
\  Then we can write
\begin{equation}\label{4.2}
B_M=\Omega_M\Omega_M^{\prime}\quad\hbox{and}\quad
         A_r\Bigl(P_M + b_M(I-P_M)\Bigr)=\Omega_M\Gamma_M
\end{equation}
for some $\Omega_M^{\prime}, \Gamma_M\in H^\infty_{M_n}$. \  By
Lemma \ref{lem4.3}, $\Omega_M=b_MP_N+(I-P_N)$ with
$N:=\bigl(\hbox{ran}\, (A_r(\alpha_M)P_M)\bigr)^{\perp}$. \  We now
claim that
\begin{equation}\label{4.2-1}
\hbox{$\Gamma_M(\alpha_M)$ is invertible. \ }
\end{equation}
Since
$$
\hbox{det}\,\Bigl[ A_r\Bigl(P_M + b_M(I-P_M)\Bigr)\Bigr]
=b_M^{\hbox{rank}(I-P_M)}\cdot \hbox{det}\, A_r
$$
and
$$
\hbox{det}\,\Omega_M\Gamma_M=(b_M)^{\hbox{dim}\, N} \cdot
\hbox{det}\, \Gamma_M,
$$
it follows from (\ref{4.2}) that
\begin{equation} \label{4.3}
\hbox{det}\, A_r \cdot
(b_M)^{\hbox{rank}(I-P_M)}=(b_M)^{\hbox{dim}\, N} \cdot \hbox{det}\,
\Gamma_M\,.
\end{equation}
But since $A_r$ and $\Theta$ are right coprime, and hence
$A_r(\alpha_M)$ is invertible, it follows that
$$
\hbox{dim}\, N=\hbox{dim}\, \bigl(\hbox{ran}\,
(A_r(\alpha_M)P_M)\bigr)^{\perp} =\hbox{dim}(\hbox{ran}\,
P_M)^{\perp}=\hbox{rank}(I-P_M),
$$
which together with (\ref{4.3}) implies that $\hbox{det}\,
\Gamma_M=\hbox{det}\, A_r$. \  This proves (\ref{4.2-1}). \
Therefore we have
$$
\Phi =\prod_{m=1}^{M-1} \Bigl(b_mP_m + (I-P_m)\Bigr) \Gamma_M^*  \Omega_M^\prime \Theta,\\
$$
where $\Gamma_M(\alpha_M)$ is invertible. \  Thus $\Gamma_M(\alpha)$
is invertible for all $\alpha\in\mathcal{Z}(\theta)$, and hence by
Lemma \ref{lem2.2}, $\Theta$ and $\Gamma_M$ are coprime. \

If we repeat this argument then after $M$ steps we get the left
coprime factorization of $\Phi=A_l^* \Omega$, where $\Omega$ still
has $\Theta$ as an inner divisor. \
\end{proof}

\smallskip

Our main theorem of this section now follows:

\begin{theorem} \label{thm4.5}
Let $\Phi\in L^\infty_{M_n}$ be a matrix-valued rational function. \
Then we may write
$$
\Phi_- = B^*\Theta\,,
$$
where $B\in H^2_{M_n}$ and $\Theta:=I_\theta$ with a finite Blaschke
product $\theta$. \  Suppose $B$ and $\Theta$ are coprime. \  If
both $T_{\Phi}$ and $T_{\Phi}^2$ are hyponormal then $T_{\Phi}$ is
either normal or analytic. \
\end{theorem}

\begin{proof} Suppose ${\Phi}$ is not analytic. \  Then $\Theta$ is not
constant unitary. \  Since $T_{\Phi}$ is hyponormal, it follows that
$\hbox{ker}H_{\Phi_+^*} \subseteq \hbox{ker}\, H_{\Phi_-^*}=\Theta
H_{\mathbb C^n}^2$. \  Thus we can write
$$
\Phi_+=\Theta\Delta_r A_r^*   \qquad (\hbox{right coprime
factorization}),
$$
where $\Delta_r$ is an inner matrix function. \  Let $\theta_0$ be a
minimal inner function such that $\Theta_0\equiv
I_{\theta_0}=\Delta_r\Theta_1$ for some inner matrix function
$\Theta_1$. \  We also write $A:=A_r\Theta_1$, and hence
$$
\Phi_+=\Theta\Theta_0A^*.
$$
On the other hand,
we need to keep in mind that $\Theta=I_\theta$ and $\Theta_0=I_{\theta_0}$
are inner functions, {\it constant} along the diagonal,
so that these factors commute with all other matrix functions
in the computations below.
Note that $\Phi^2\Theta^2\Theta_0^2\in H^\infty_{M_n}$ and
$\Phi^{*2}\Theta^2\Theta_0^2\in H^\infty_{M_n}$. \  We thus have
$$
\aligned T_{\Theta^2\Theta_0^2}^* [T_{\Phi}^{2*},
T_{\Phi}^2]T_{\Theta^2\Theta_0^2}
&=T_{\Theta^2\Theta_0^2}^*T_{\Phi}^{2*}T_{\Phi}^2T_{\Theta^2\Theta_0^2}
-T_{\Theta^2\Theta_0^2}^*T_{\Phi}^2T_{\Phi}^{2*}T_{\Theta^2\Theta_0^2}\\
&=T_{\Phi^{*2}\Theta^{*2}\Theta_0^{*2}}T_{\Phi^2\Theta^2\Theta_0^2}
-T_{\Phi^{2}\Theta^{*2}\Theta_0^{*2}}T_{\Phi^{*2}\Theta^2\Theta_0^2}\\
&=T_{\Phi^{*2}\Theta^{*2}\Theta_0^{*2}\Phi^2\Theta^2\Theta_0^2}
-T_{\Phi^{2}\Theta^{*2}\Theta_0^{*2}\Phi^{*2}\Theta^2\Theta_0^2}\\
&=T_{\Phi^{*2}\Phi^2}-T_{\Phi^2\Phi^{*2}}=0
\quad (\hbox{since} \ \Phi \ \hbox{is normal}).
\endaligned
$$
The positivity of $[T_{\Phi}^{2*}, T_{\Phi}^2]$ implies that
$[T_{\Phi}^{2*}, T_{\Phi}^2]T_{\Theta^2\Theta_0^2}=0$. \  We thus
have
\begin{equation}\label{4.4}
\aligned
0&=[T_{\Phi}^{2*}, T_{\Phi}^2]T_{\Theta^2\Theta_0^2}\\
&=T_{\Phi^*}^2T_{\Phi^2\Theta^2\Theta_0^2}-T_{\Phi}^2T_{\Phi^{*2}\Theta^2\Theta_0^2}\\
&=T_{\Phi^*}T_{\Phi^*\Phi^2\Theta^2\Theta_0^2}-T_{\Phi}T_{\Phi\Phi^{*2}\Theta^2\Theta_0^2}\\
&=\Bigl(T_{\Phi^{*2}\Phi^2\Theta^2\Theta_0^2}
-H_{\Phi}^*H_{\Phi^*\Phi^2\Theta^2\Theta_0^2}\Bigr)-\Bigl(T_{\Phi^2\Phi^{*2}\Theta^2\Theta_0^2}
-H_{\Phi^*}^*H_{\Phi\Phi^{*2}\Theta^2\Theta_0^2}\Bigr)\quad \hbox{(by (\ref{1.3}))}\\
&=H_{\Phi^*}^*H_{\Phi\Phi^{*2}\Theta^2\Theta_0^2}-H_{\Phi}^*H_{\Phi^*\Phi^2\Theta^2\Theta_0^2}\\
&=H_{\Phi_+^*}^*H_{\Phi\Phi^{*2}\Theta^2\Theta_0^2}-H_{\Phi^*_-}^*H_{\Phi^*\Phi^2\Theta^2\Theta_0^2}.
\endaligned
\end{equation}
Let $\Omega:=\hbox{GCD}\, (\Theta_0, \Theta)$. \  Then by Lemma
\ref{lem2.1}, $\Omega=I_\omega$ for an inner function $\omega$. \
Thus we can write
$$
\Theta=\Theta^{\prime}\Omega\quad\hbox{and}\quad
\Theta_0=\Theta_0^{\prime}\Omega ,
$$
where $\Theta^{\prime}=I_{\theta^{\prime}}$ and
$\Theta_0^{\prime}=I_{\theta_0^{\prime}}$ for some inner functions
$\theta^\prime$ and $\theta_0^\prime$. \  Observe that
\begin{equation}\label{4.5}
\aligned H_{\Phi^*\Phi^2\Theta^2\Theta_0^2} &=H_{(\Theta
B^*+\Theta^*\Theta_0^* A)
       (\Theta^* B+\Theta\Theta_0 A^*)^2\Theta^2\Theta_0^2}\\
&=H_{(\Theta B^*+\Theta^*\Theta_0^* A)(B+\Theta^2\Theta_0A)^2\Theta_0^2}\\
&=H_{\Theta^*\Theta_0^* A(B+\Theta^2\Theta_0A)^2\Theta_0^2}\\
&=H_{ A(B+\Theta^2\Theta_0A)^2\Theta_0\Theta^*}\\
&=H_{ A(B+\Theta^2\Theta_0A)^2\Theta_0^{\prime}\Theta^{\prime*}}.
\endaligned
\end{equation}
Since $\Phi$ is normal we also have
\begin{equation}\label{4.6}
\aligned H_{\Phi\Phi^{*2}\Theta^2\Theta_0^2}&=H_{(\Theta
B^*+\Theta^*\Theta_0^* A)^2
(\Theta^* B+\Theta\Theta_0A^*)\Theta^2\Theta_0^2}\\
&=H_{(\Theta^2 \Theta_0 B^*+ A)(\Theta^* B+\Theta\Theta_0A^*)}\\
&=H_{(\Theta^2 \Theta_0 B^*+ A)B\Theta^*}\\
&=H_{AB\Theta^*}\,.
\endaligned
\end{equation}
We now claim that
\begin{equation}\label{4.7}
\hbox{$\theta^{\prime}$ is not constant. \ }
\end{equation}
Toward (\ref{4.7}) we assume to the contrary that $\theta^{\prime}$
is a constant. \  Then by (\ref{4.5}) we have
$$
H_{\Phi^*\Phi^2\Theta^2\Theta_0^2}=H_{
A(B+\Theta^2\Theta_0A)^2\Theta_0^{\prime}\Theta^{\prime*}}=0\,,
$$
so that $H_{\Phi_-^*}^*H_{\Phi^*\Phi^{2}\Theta^2\Theta_0^2}=0$, and
by (\ref{4.4}) we have
\begin{equation}\label{4.8}
H_{\Phi_+^*}^*H_{\Phi\Phi^{*2}\Theta^2\Theta_0^2}=0.
\end{equation}
Observe that
$$
\aligned
H_{AB\Theta^*}=0 &\Longleftrightarrow AB\Theta^* \in H_{M_n}^2\\
&\Longleftrightarrow AB \in \Theta H_{M_n}^2\\
&\Longleftrightarrow A=\Theta A^{\prime} \quad (\hbox{since} \ B \
\hbox{and}  \ \Theta \ \hbox{are coprime})\,,
\endaligned
$$
which implies that $A(\alpha)=0$ for each
$\alpha\in\mathcal{Z}(\theta)$, a contradiction. \  Therefore
$$
H_{\Phi\Phi^{*2}\Theta^2\Theta_0^2}=H_{AB\Theta^*} \neq
0\quad\hbox{and}\quad \hbox{cl
ran}\,H_{\Phi\Phi^{*2}\Theta^2\Theta_0^2}=\mathcal
H(\widetilde{\Delta})
$$
for some nonconstant (left) inner divisor $\Delta$ of $\Theta$. \
Thus it follows from Lemma \ref{lem4.4} and (\ref{4.8}) that
$$
\mathcal H(\widetilde{\Delta})=\hbox{cl
ran}\,H_{\Phi\Phi^{*2}\Theta^2\Theta_0^2} \subseteq \hbox{ker}\,
H_{\Phi_+^*}^*\subseteq \widetilde{\Theta} H_{\mathbb C^n}^2,
$$
giving a contradiction. \  This proves (\ref{4.7}). \  Observe
$$
\hbox{cl ran}\, H_{\Phi^*\Phi^2\Theta^2\Theta_0^2}=\hbox{cl
ran}\,H_{ A(B+\Theta^2\Theta_0A)^2\Theta_0^{\prime}\Theta^{\prime*}}
\subseteq \mathcal H(\widetilde{\Theta}^{\prime})\perp
\widetilde{\Theta}H_{\mathbb C^n}^2= \hbox{ker}\, H_{\Phi_-^*}^*,
$$
and
$$
\hbox{cl ran}\, H_{\Phi\Phi^{*2}\Theta^2\Theta_0^2}=\hbox{cl
ran}\,H_{AB\Theta^*} \subseteq \mathcal H(\widetilde\Theta)\perp
\widetilde{\Theta} H_{\mathbb C^n}^2\supseteq \hbox{ker}\,
H_{\Phi_+^*}^*.
$$
Thus by (\ref{4.4}) we have
\begin{equation}\label{4.9}
\aligned \hbox{ker}\, H_{\Phi^*\Phi^2\Theta^2\Theta_0^2}
&=\hbox{ker}\,
H_{\Phi^*_-}^*H_{\Phi^*\Phi^2\Theta^2\Theta_0^2}\\
&=\hbox{ker}\, H_{\Phi_+^*}^*H_{\Phi\Phi^{*2}\Theta^2\Theta_0^2}\\
&=\hbox{ker}\, H_{\Phi\Phi^{*2}\Theta^2\Theta_0^2}.
\endaligned
\end{equation}
Observe that for all $\alpha \in \mathcal Z(\theta)$,
$$
\Bigl(A(B+\Theta^2\Theta_0A)^2\Theta_0^{\prime}\Bigr)(\alpha)
=A(\alpha)B(\alpha)^2\Theta_0^{\prime}(\alpha)\,.
$$
Since $B(\alpha)$ and $\Theta_0^{\prime}(\alpha)$ are invertible, we
have
$$
\hbox{dim}\,\hbox{ker}\,
\bigl(A(B+\Theta^2\Theta_0A)^2\Theta_0^{\prime}\bigr)
(\alpha)=\hbox{dim ker}\, A(\alpha)=\hbox{dim}\,
\hbox{ker}\,(AB)(\alpha).
$$
By (\ref{4.5}), (\ref{4.6}) and (\ref{4.9}), we have that
$A(\alpha)=0$ for all $\alpha \in \mathcal Z(\omega)$ and hence
$\omega$ is a constant. \  Thus $\Omega$ is a constant unitary, and
hence $\Theta=\Theta^\prime$ and $\Theta_0=\Theta_0^\prime$. \
Therefore $\mathcal Z(\theta)=\mathcal Z(\theta) \setminus \mathcal
Z(\theta_0)$ and hence, by Lemma \ref{lem4.2}, $A(\alpha)$ is
invertible for each $\alpha \in \mathcal Z(\theta)$. \  Since for
each $\alpha \in \mathcal Z(\theta)$,
$$
\hbox{ $\Bigl(A(B+\Theta^2\Theta_0A)^2\Theta_0\Bigr)(\alpha)
=A(\alpha)B(\alpha)^2\Theta_0(\alpha)$ and $(AB)(\alpha)$ are
invertible,}
$$
it follows that $A(B+\Theta^2\Theta_0A)^2\Theta_0$ and $\Theta$ are
coprime, and $AB$ and $\Theta$ are coprime. \  Thus by (\ref{4.5}),
(\ref{4.6}), and (\ref{4.9}) we have
\begin{equation}\label{4.10}
\hbox{cl ran}\, H_{\Phi^*\Phi^2\Theta^2\Theta_0^2}=\hbox{cl ran}\,
H_{\Phi\Phi^{*2}\Theta^2\Theta_0^2}= \mathcal H(\widetilde{\Theta}).
\end{equation}
By the well-known result of C. Cowen \cite [Theorem 1]{Co1}\, -\, if
$\varphi\in L^\infty$ and $b$ is a finite Blaschke product of degree
$n$ then $T_{\varphi\circ b}\cong \oplus_n T_\varphi$, we may,
without loss of generality, assume that $0 \in \mathcal Z(\theta)$.
\ Since $T_{\Phi}$ is hyponormal, by \cite{GHR} (cf. p.4) there
exists $K \in H_{M_n}^{\infty}$ with $||K||_{\infty} \leq 1$ such
that
$$
H_{\Phi_-^*}=H_{K\Phi_+^*}=T_{\widetilde{K}}^*H_{\Phi_+^*}.
$$
Since $\Phi\Phi^*\Theta^2\Theta_0^2\in H^\infty_{M_n}$, we have
$$
\begin{aligned}
T_{\widetilde{K}}H_{\Phi^2\Phi^*\Theta^2\Theta_0^2}
&=T_{\widetilde{K}}H_{\Phi}T_{\Phi\Phi^*\Theta^2\Theta_0^2}
=T_{\widetilde{K}}H_{\Phi_-^*}T_{\Phi\Phi^*\Theta^2\Theta_0^2}\\
&=T_{\widetilde{K}}(T_{\widetilde{K}}^*H_{\Phi_+^*})T_{\Phi\Phi^*\Theta^2\Theta_0^2}
=T_{\widetilde{K}}T_{\widetilde{K}}^*H_{\Phi\Phi^{*2}\Theta^2\Theta_0^2}\,.
\end{aligned}
$$
Thus by (\ref{4.4}) we have
\begin{equation}\label{4.11}
\begin{aligned}
0&=H_{\Phi_+^*}^*H_{\Phi\Phi^{*2}\Theta^2\Theta_0^2}
       -H_{\Phi^*_-}^*H_{\Phi^*\Phi^2\Theta^2\Theta_0^2}\\
&=H_{\Phi_+^*}^* \Bigl(H_{\Phi\Phi^{*2}\Theta^2\Theta_0^2}
-T_{\widetilde{K}}H_{\Phi^*\Phi^2\Theta^2\Theta_0^2}\Bigr)\\
&=H_{\Phi_+^*}^*
(I-T_{\widetilde{K}}T_{\widetilde{K}}^*)H_{\Phi\Phi^{*2}\Theta^2\Theta_0^2}\,.
\end{aligned}
\end{equation}
It thus follows from (\ref{4.10}), (\ref{4.11}) and Lemma
\ref{lem4.4} that \begin{equation}\label{4.11-1}
(I-T_{\widetilde{K}}T_{\widetilde{K}}^*)\bigl(\mathcal
H(\widetilde{\Theta})\bigr) =\hbox{cl
ran}\Bigl((I-T_{\widetilde{K}}T_{\widetilde{K}}^*)
   H_{\Phi^{*2}\Phi\Theta^2\Theta_0^2}\Bigr)\subseteq
     \hbox{ker}\, H_{\Phi_+^*}^* \subseteq \widetilde{\Theta}H_{\mathbb
C^n}^2.
\end{equation}
Since
$||K||_\infty=||\widetilde{K}||_\infty=||\widetilde{K}^*||_\infty$,
it follows that $||T_{\widetilde{K}}||=||T_{\widetilde{K}^*}|| \leq
1$. \  For each $i=1,2,\cdots,n$, put
$$
E_i\label{E_i}:=(0,0, \cdots, 1, 0, \cdots,0,0)^t.
$$
Since $0 \in \mathcal Z(\theta) \cap  \mathcal
Z(\widetilde{\theta})$, we have $E_i \in \mathcal H(\Theta) \cap
\mathcal H(\widetilde{\Theta})$ and by (\ref{4.11-1}),
$$
E_i-T_{\widetilde{K}}T_{\widetilde{K}}^*E_i=\widetilde{\Theta} F_i
\quad (\hbox{some}\ F_i \in H_{\mathbb C^n}^2).
$$
Observe that
$$
\aligned
E_i-T_{\widetilde{K}}T_{\widetilde{K}}^*E_i=\widetilde{\Theta} F_i
&\Longrightarrow
T_{\widetilde{K}}T_{\widetilde{K}}^*E_i=E_i-\widetilde{\Theta} F_i\\
&\Longrightarrow ||T_{\widetilde{K}}T_{\widetilde{K}}^*E_i||_2^2
     =||E_i-\widetilde{\Theta}F_i||_2^2\\
&\Longrightarrow ||T_{\widetilde{K}}T_{\widetilde{K}}^*E_i||_2^2
           =||E_i||_2^2+||\widetilde{\Theta}F_i||_2^2
              \quad(\hbox{since} \ E_i \in \mathcal H(\widetilde{\Theta}))\\
&\Longrightarrow
||T_{\widetilde{K}}T_{\widetilde{K}}^*E_i||_2^2=1+||\widetilde{\Theta}F_i||_2^2.\\
\endaligned
$$
But since $||T_{\widetilde{K}}T_{\widetilde{K}}^*|| \leq 1$, it
follows that $||T_{\widetilde{K}}T_{\widetilde{K}}^*E_i||_2=1$ and
$F_i=0$ for all $i=1,2,\cdots,n$. \  We thus have
$||T_{\widetilde{K}}^*E_i||_2=1$ for all $i=1,2,\cdots,n$. \  Write
$$
K(z):=\begin{bmatrix} k_{ij}(z) \end{bmatrix}\quad\hbox{and}\quad
k_{ij}(z)=  \sum_{m=0}^{\infty}k_{ij}^{(m)}z^m.
$$
Then $\widetilde{K}^*(z) =K(\overline{z})=\begin{bmatrix}
k_{ij}(\overline{z})
\end{bmatrix}$\,,
and hence
$$
T_{\widetilde{K}}^* E_i=[P(k_{1i}(\overline{z})),
P(k_{2i}(\overline{z})), \cdots,
P(k_{ni}(\overline{z}))]^t=[k_{1i}^{(0)}, k_{2i}^{(0)},\cdots,
k_{ni}^{(0)}]^t.
$$
But since $||T_{\widetilde{K}}^* E_i||_2 =1$ for all
$i=1,2,\cdots,n$, it follows that
$$
||[k_{1i}^{(0)}, k_{2i}^{(0)},\cdots, k_{ni}^{(0)}]^t||_2=1  \ \
\hbox{for all} \ i=1,2,\cdots n.
$$
Therefore
$$
1=\frac{1}{n}\Bigl|\Bigl|\bigl[k_{ij}^{(0)}\bigr]\Bigr|\Bigr|_2^2
\leq
\frac{1}{n}\sum_{m=0}^{\infty}\Bigl|\Bigl|\bigl[k_{ij}^{(m)}\bigr]
\Bigl|\Bigl|_2^2=\frac{1}{n}||K||_2^2 \leq ||K||_{\infty}^2 \leq 1,
$$
which implies that $\bigl[k_{ij}^{(m)}\bigr]=0$ for all $m \geq 1$.
\ Hence $K=\bigl[k_{ij}^{(0)}\bigr]$, so that $\widetilde K=K^*$. \
Observe that
$$
\aligned (I-T_{\widetilde{K}}T_{\widetilde{K}}^*)\mathcal
H(\widetilde{\Theta})=0
&\Longrightarrow (I-T_{K^*K})\mathcal H(\widetilde{\Theta})=0\\
&\Longrightarrow K^*K=I_{n} \quad (\hbox{since} \ 0 \in \mathcal
Z(\widetilde{\theta})).
\endaligned
$$
Therefore  $\widetilde{K}=K^*$ is a constant unitary and hence we
have
$$
\hbox{$[T_{\Phi}^*, T_{\Phi}]$}
=H_{\Phi^*}^*H_{\Phi^*}-H_{\Phi}^*H_{\Phi}
=H_{\Phi^*}^*H_{\Phi^*}- H_{\Phi^*}^*T_{\widetilde{K}\widetilde{K}^*}H_{\Phi^*}=0,
$$
which implies that $T_{\Phi}$ is normal. \
\end{proof}

\smallskip

\begin{cor} \label{cor4.6}
Let $\Phi\in L^\infty_{M_n}$ be a matrix-valued trigonometric
polynomial whose co-analytic outer coefficient is invertible. \ If
$T_\Phi$ and $T_\Phi^2$ are hyponormal then $T_\Phi$ is normal. \
\end{cor}

\begin{proof} Immediate from Theorem \ref{thm4.5} together with the observation that
$\Phi_-=B^*\Theta$ with $\Theta=I_{z^m}$ is a coprime factorization
if and only if $B(0)$ is a co-analytic outer coefficient and is
invertible. \
\end{proof}

\smallskip

\begin{remark}\label{remark4.7} In Theorem \ref{thm4.5}, the ``coprime" condition is essential. \  To see
this, let
$$
T_\Phi:=
\begin{bmatrix}
T_b+T_b^* &0\\ 0& T_b
\end{bmatrix}\quad \hbox{($b$ is a finite Blaschke product)}.
$$
Since $T_b+T_b^*$ is normal and $T_b$ is analytic, it follows that
$T_\Phi$ and $T_\Phi^2$ are both hyponormal. \  Obviously, $T_\Phi$
is neither normal nor analytic. \  Note that
$\Phi_-\equiv\left[\begin{smallmatrix} b&0\\
0&0\end{smallmatrix}\right]=\left[\begin{smallmatrix} 1&0\\
0&0\end{smallmatrix}\right]^*\cdot I_b$,
where $\left[\begin{smallmatrix} 1&0\\
0&0\end{smallmatrix}\right]$ and $I_b$ are not coprime. \qed
\end{remark}

On the other hand, we have not been able to determine whether this
phenomenon is quite accidental. \  In fact we would guess that if
$\Phi\in L^\infty_{M_n}$ is a matrix-valued rational function such
that $T_\Phi$ is subnormal then $T_\Phi=T_A\oplus T_B$, where $T_A$
is normal and $T_B$ is analytic. \

\vskip 1cm

%
%
%
%

\section{Subnormal Toeplitz Completions}

\bigskip
%
%


Given a partially specified operator matrix with some known entries,
the problem of finding suitable operators to complete the given
partial operator matrix so that the resulting matrix satisfies
certain given properties is called a {\it completion problem}. \
Dilation problems are special cases of completion problems: in other
words, the dilation of $T$ is a completion of the partial operator
matrix $\left[\begin{smallmatrix} T&?\\?&?\end{smallmatrix}\right]$.
\  In recent years, operator theorists have been interested in the
subnormal completion problem for $\left[\begin{smallmatrix}
U^*&?\\?&U^*\end{smallmatrix}\right]$, where $U$ is the shift on
$H^2$. \  In this section, we solve this completion problem. \

A {\it partial block Toeplitz matrix} is simply an $n\times n$
matrix some of whose entries are specified Toeplitz operators and
whose remaining entries are unspecified. \  A {\it subnormal
completion} of a partial operator matrix is a particular
specification of the unspecified entries resulting in a subnormal
operator. \  For example
\begin{equation} \label{6.1}
\left[\begin{matrix} T_z& 1-T_z T_{\overline z}\\ 0&T_{\overline z}
\end{matrix}\right]
\end{equation}
is a subnormal (even unitary) completion of the $2\times 2$ partial
operator matrix
$$
\left[\begin{matrix} T_z& ?\\ ?&T_{\overline z}\end{matrix}\right].
$$
A {\it subnormal Toeplitz completion} of a partial block Toeplitz
matrix is a subnormal completion whose unspecified entries are
Toeplitz operators. \  Then the following question comes up at once:
Does there exist a subnormal Toeplitz completion of
$\left[\begin{smallmatrix} T_z& ?\\ ?&T_{\overline
z}\end{smallmatrix}\right]$ ?  \ Evidently, (\ref{6.1}) is not such
a completion. \  To answer this question, let
$$
\Phi\equiv \left[\begin{matrix} z&\varphi\\ \psi&\overline
z\end{matrix}\right] \quad (\varphi, \psi\in L^\infty).
$$
If $T_\Phi$ is hyponormal then by \cite{GHR} (cf. p.4), $\Phi$
should be normal. \  Thus a straightforward calculation shows that
$$
|\varphi|=|\psi|\quad\hbox{and}\quad \overline
z(\varphi+\overline\psi)=z(\varphi+\overline\psi),
$$
which implies that $\varphi=-\overline\psi$. \  Thus a direct
calculation shows that
$$
[T_\Phi^*,\, T_\Phi]= \left[\begin{matrix} \ast&\ast\\ \ast& T_z
T_{\overline z}-1\end{matrix}\right],
$$
which is not positive semi-definite because $T_z T_{\overline z}-1$
is not. \  Therefore, there are no hyponormal Toeplitz
completions of $\left[\begin{smallmatrix} T_z& ?\\ ?&T_{\overline
z}\end{smallmatrix}\right]$. \
However the following problem
has remained unsolved until now:
\bigskip

\noindent{\bf Problem A.} Let $U$ be the shift on $H^2$. \  Complete
the unspecified Toeplitz entries of the partial block Toeplitz
matrix $A:=\left[\begin{smallmatrix} U^*& ?\\
?&U^*\end{smallmatrix}\right]$ to make $A$ subnormal. \

\bigskip

In this section we give a complete answer to Problem A. \

\begin{theorem} \label{thm6.1}Let $\varphi, \psi \in L^{\infty}$ and consider
$$
A:=\left[\begin{matrix} T_{\overline z}& T_{\varphi}\\ T_{\psi}& T_{\overline
z}\end{matrix}\right].
$$
The following statements are equivalent.

{\rm (i)} \ $A$ is normal.

{\rm (ii)} \ $A$ is subnormal.

{\rm (iii)} \ $A$ is $2$-hyponormal.

{\rm (iv)} \ One of the following conditions
holds:
\medskip
\begin{itemize}
\item[1.] $\varphi=e^{i\theta}z+\beta$\quad and\quad $\psi=e^{i\omega}\varphi$\quad
\hbox{\rm ($\beta\in\mathbb C$; $\theta,\omega\in [0,2\pi)$);}
\item[2.] $\varphi=\alpha\, \overline z+ e^{i\theta}\sqrt{1+|\alpha|^2}\,z+ \beta$\quad and\quad
$\psi=e^{i\,(\pi-2\,{\rm arg}\,\alpha)}\varphi$
\end{itemize}
\qquad\qquad\qquad\hbox{\rm ($\alpha,\beta\in\mathbb C$, $\alpha\ne
0$; $\theta\in [0,2\pi)$).}
\end{theorem}

\medskip

Theorem \ref{thm6.1} says that the unspecified entries of the matrix
$\left[\begin{smallmatrix} T_{\overline z}& ?\\ ?&T_{\overline
z}\end{smallmatrix}\right]$ are Toeplitz operators with symbols
which are both analytic or trigonometric polynomials of degree 1. \
In fact, as we will see in the proof of Theorem \ref{thm6.1}, our
solution is just the normal completion. \  However the solution is
somewhat more intricate than one would expect. \

\bigskip

%
%
%


To prove Theorem \ref{thm6.1} we need several technical lemmas. \

\begin{lemma}\label{lem6.2} For $j=1,2,3$, let $\theta_j$ be an inner function. \ If
$\theta_1 \mathcal H (\theta_2) \subseteq \mathcal H (\theta_3)$
then either $\theta_2$ is constant or $\theta_1\theta_2$ is a
divisor of $\theta_3$. \  In particular, if $\theta_1\mathcal H
(\theta_2)\subseteq \mathcal H (\theta_1)$ then $\theta_1$ or
$\theta_2$ is constant. \
\end{lemma}

\begin{proof} Suppose $\theta_2$ is not constant. \  If $\theta_1 \mathcal H
(\theta_2) \subseteq \mathcal H (\theta_3)$ then by Lemma
\ref{lem2.4}, for all $f\in \mathcal{H}(\theta_2)$,
$\overline{\theta}_1\overline{f}\theta_3\in zH^2$, and hence
$\overline f\theta_3\in zH^2$, so that $f\in \mathcal{H}(\theta_3)$,
which implies that $\mathcal{H}(\theta_2)\subseteq
\mathcal{H}(\theta_3)$, and therefore $\theta_3 H^2\subseteq
\theta_2 H^2$. \  Thus $\theta_2$ is a divisor of $\theta_3$. \  We
can then write $\theta_3=\theta_0\theta_2$ for some inner function
$\theta_0$. \  It suffices to show that $\theta_1$ is a divisor of
$\theta_0$. \  Observe that
$$
\aligned \theta_1 \mathcal H (\theta_2) \subseteq \mathcal
H(\theta_0 \theta_2) &\Longrightarrow
\hbox{ran}\,(T_{\theta_1}H_{\overline{\theta}_2}^*) \subseteq
\mathcal H
(\theta_0 \theta_2)\\
&\Longrightarrow \theta_0 \theta_2H^2 \subseteq
\hbox{ker}H_{\overline{\theta}_2}T_{\overline{\theta}_1}\\
&\Longrightarrow
H_{\overline{\theta}_2}T_{\overline{\theta}_1\theta_0\theta_2}=0\\
&\Longrightarrow
H_{\overline{\theta}_1\theta_0}-T_{\widetilde{\theta_2}}
H_{\overline{\theta}_1\theta_0\theta_2}=0 \\
&\Longrightarrow
H_{\overline{\theta}_1\theta_0}-T_{\widetilde{\theta}_2}
T_{\overline{\widetilde{\theta}}_2}H_{\overline{\theta}_1\theta_0}=0\\
&\Longrightarrow
H_{\overline{\widetilde{\theta}}_2} H_{\theta_0\overline{\theta}_1}=0\,,
\endaligned
$$
where the fourth implication follows from the fact that
$H_{\varphi\psi}=T_{\widetilde\varphi}^*H_\psi+H_\varphi T_\psi$ for
any $\varphi,\psi\in L^\infty$. \  But since $\theta_2$ is not
constant it follows that  $\theta_1$ is a divisor of $\theta_0$. \
The second assertion follows at once from the first. \
\end{proof}

\bigskip


Suppose $\Phi\equiv \Phi_-^*+\Phi_+\in L^\infty_{M_n}$ is such that
$\Phi$ and $\Phi^*$ are of bounded type, with
$$
\Phi_+=A^*\Theta\quad\hbox{and}\quad \Phi_-=B_{\ell}^*\Omega_2\
\hbox{(left coprime factorization)},
$$
where $\Theta=I_\theta$ for an inner function $\theta$. \  If
$T_\Phi$ is hyponormal, then in view of Proposition \ref{pro3.2},
$\Phi$ can be written as:
\begin{equation} \label{6.2}
\Phi_+=A^*\Omega_1\Omega_2\quad\hbox{and}\quad
\Phi_-=B_{\ell}^*\Omega_2\,,
\end{equation}
where $\Omega_1\Omega_2=\Theta=I_\theta$. \  We also note that
$\Omega_1\Omega_2=\Theta=\Omega_2\Omega_1$. \

\bigskip

The following lemma will be extensively used in the proof of Theorem
\ref{thm6.1}. \

\medskip

\begin{lemma}\label{lem6.3} Let $\Phi\equiv \Phi_-^*+\Phi_+\in L^\infty_{M_n}$ be such
that $\Phi$ and $\Phi^*$ are of bounded type of the form
$(\ref{6.2})$:
$$
\Phi_+=A^*\Omega_1\Omega_2=A^*\Theta\quad\hbox{and}\quad
\Phi_-=B_{\ell}^*\Omega_2\ \hbox{\rm (left coprime factorization)},
$$
where $\Theta=I_\theta$ for an inner function $\theta$. \  If
$\hbox{\rm ker}\,[T_\Phi^*, T_\Phi]$ is invariant under $T_\Phi$,
then
$$
\Omega_1 H^2_{\mathbb C^n}\ \subseteq\ \hbox{\rm ker}\,[T_\Phi^*,
T_\Phi],
$$
and therefore
$$
\hbox{\rm cl ran}\,[T_\Phi^*, T_\Phi]\ \subseteq\
\mathcal{H}(\Omega_1).
$$
Assume instead that we decompose $\Phi\in L^\infty_{M_n}$ as:
$$
\Phi_+=\Delta_2\Delta_0 A_r^*\ \ \hbox{\rm (right coprime
factorization)}
$$
and
$$
\Phi_-=\Delta_2 B_r^*\ \ \hbox{\rm (right coprime factorization)}\,.
$$
If $T_\Phi$ is hyponormal then
$$
\Delta_2\mathcal{H}(\Delta_0)\ \subseteq\ \hbox{\rm cl
ran}\,[T_\Phi^*, T_\Phi].
$$
Hence, in particular, if $T_\Phi$ is hyponormal and $\hbox{ker}\, [T_\Phi^*, T_\Phi]$ is invariant under
$T_\Phi$, then
\begin{equation}\label{5.2-1}
\Delta_2\mathcal{H}(\Delta_0)\ \subseteq\ \hbox{\rm cl
ran}\,[T_\Phi^*, T_\Phi]\ \subseteq\ \mathcal{H}(\Omega_1).
\end{equation}
\end{lemma}

\begin{proof} See \cite [Lemma 3.2 and Theorem 3.7] {CHL}. \end{proof}


\medskip

\begin{lemma}\label{lem6.4} Let
$$
\Phi_-=\begin{bmatrix} z&\theta_1 \overline{b}\\
\theta_0 \overline{a}& z\end{bmatrix} \quad (a \in \mathcal
H(\theta_0),\ b\in \mathcal H(\theta_1)\ \hbox{and $\theta_j$ inner}
\  (j=0,1)).
$$
If  $\theta_0=z^n \theta_0^{\prime}$ {\rm ($n\ge 1$;
$\theta_0'(0)\ne 0$)} and $\theta_1(0) \neq 0$, then $\hbox{\rm
ker}\,H_{\Phi_-^*}=\Delta H^2_{\mathbb C^2}$, where
$$
\Delta=
\begin{cases}\qquad\qquad
\begin{bmatrix}z\theta_1&0\\0&\theta_0\end{bmatrix} \quad &(n=1);\\
\frac{1}{\sqrt{|\alpha|^2+1}}\begin{bmatrix}z\theta_1&\alpha
\theta_1\\-\overline{\alpha}\theta_0&z^{n-1}\theta_0'\end{bmatrix}
\quad &(n \geq 2)\quad
\left(\alpha:=-\frac{a(0)}{\theta_1(0)}\right).
\end{cases}
$$
\end{lemma}

\smallskip

\begin{proof} Observe that for $f,g\in H^2$,
$$
\Phi_-^*\begin{bmatrix}f\\g\end{bmatrix}
= \begin{bmatrix} \overline{z}&\overline{z}^n\overline{\theta_0'} a\\
\overline{\theta}_1 b&\overline{z}\end{bmatrix}\begin{bmatrix}f\\g\end{bmatrix}
\in H^2_{\mathbb C^2} \ \Longleftrightarrow\
\begin{bmatrix} \overline{z}f+\overline{z}^n\overline{\theta_0'} ag\\
\overline{\theta}_1 bf+\overline{z}g \end{bmatrix} \in H^2_{\mathbb
C^2}.
$$
Thus if
$\Phi_-^*\left[\begin{smallmatrix}f\\g\end{smallmatrix}\right]\in
H^2_{\mathbb C^2}$, then $\overline{\theta}_1 bf+\overline{z}g  \in
H^2$. \  Since $\theta_1(0)\ne 0$, we have $\overline{\theta}_1 bfz
\in H^2$, and hence $f=\theta_1 f_1$ for some $f_1\in H^2$. \  In
turn, $bf_1+\overline{z}g \in H^2$, so that $g=zg_1$ for some
$g_1\in H^2$. \  We therefore have
\begin{equation} \label{6.3}
\overline{z}\theta_1
f_1+\overline{z}^{n-1}\overline{\theta_0'} ag_1\in H^2.
\end{equation}
If $n=1$, then (\ref{6.3}) implies $g_1=\theta_0'g_2$ and $f_1=zf_2$
for some $g_2, f_2\in H^2$. \  Thus $f=z\theta_1f_2$ and
$g=\theta_0g_2$, which implies
$$
\hbox{ker}\,H_{\Phi_-^*}=\begin{bmatrix}z\theta_1&0\\0&\theta_0\end{bmatrix}
H^2_{\mathbb C^2}.
$$
If instead $n\geq 2$, then (\ref{6.3}) implies that
$\overline{z}^{n-2}\overline{\theta_0'}ag_1 \in H^2$, so that
$g_1=z^{n-2}\theta_0'g_2$ for some $g_2\in H^2$. \  We thus have
$$
\aligned \overline{z}\theta_1
f_1+\overline{z}^{n-1}\overline{\theta_0'}ag_1 \in H^2
&\Longleftrightarrow \overline{z}\theta_1
f_1+\overline{z}ag_2 \in H^2\\
&\Longleftrightarrow \theta_1(0)f_1(0)+a(0)g_2(0)=0\\
&\Longleftrightarrow g_2(0)=\frac{1}{\alpha} f_1(0)\quad
(\hbox{recall that}\ \alpha=-\frac{a(0)}{\theta_1(0)}).
\endaligned
$$
Therefore we have
\begin{equation} \label{6.4}
\begin{bmatrix}f\\g\end{bmatrix} \in \hbox{ker}H_{\Phi_-^*}
\Longleftrightarrow \ f=\theta_1 f_1,\ g=z^{n-1}\theta_0'g_2,\
\hbox{and}\ g_2(0)=\frac{1}{\alpha}f_1(0).
\end{equation}
Put
$$
\Delta:=\frac{1}{\sqrt{|\alpha|^2+1}}\begin{bmatrix}z\theta_1&\alpha
\theta_1\\-\overline{\alpha}\theta_0&z^{n-1}\theta_0'\end{bmatrix}.
$$
Then $\Delta$ is inner, and for $h_1,h_2\in H^2$,
$$
\Delta \begin{bmatrix}h_1\\h_2\end{bmatrix}=
\frac{1}{\sqrt{|\alpha|^2+1}}
\begin{bmatrix}z\theta_1h_1+\alpha\theta_1h_2\\
-\overline{\alpha}z^n\theta_0'h_1+z^{n-1}\theta_0'h_2\end{bmatrix}
=\frac{1}{\sqrt{|\alpha|^2+1}}
\begin{bmatrix}\theta_1\bigl(zh_1+\alpha
h_2\bigr)\\z^{n-1}\theta_0'\bigl(-\overline{\alpha}zh_1+h_2\bigr)\end{bmatrix}.
$$
But since $\frac{1}{\alpha}\bigl(zh_1+\alpha
h_2\bigr)(0)=\bigl(-\overline{\alpha}zh_1+h_2 \bigr)(0)$, it follows
from (\ref{6.4}) that $\hbox{ker}\, H_{\Phi_-^*}=\Delta H^2_{\mathbb
C^2}$. \
\end{proof}

\medskip

\begin{lemma}\label{lem6.5} Let
$$
\Phi_-=\begin{bmatrix} z&\theta_1 \overline{b}\\
\theta_0 \overline{a}& z\end{bmatrix} \quad (a \in \mathcal
H(\theta_0),\ b\in \mathcal H(\theta_1)\ \hbox{and $\theta_j$ inner}
\ (j=0,1)).
$$
If  $\theta_1=z^n \theta_1^{\prime}$ {\rm ($n\ge 1$;
$\theta_1'(0)\ne 0$)} and $\theta_0(0) \neq 0$, then $\hbox{\rm
ker}\,H_{\Phi_-^*}=\Delta H^2_{\mathbb C^2}$, where
$$
\Delta=
\begin{cases}\qquad\qquad
\begin{bmatrix}\theta_1&0\\0&z\theta_0\end{bmatrix} \quad &(n=1);\\
\frac{1}{\sqrt{|\alpha|^2+1}}\begin{bmatrix}z^{n-1}\theta_1'&
-\overline{\alpha}\theta_1\\ \alpha\theta_0&z\theta_0\end{bmatrix}
\quad &(n \geq 2)\quad
\left(\alpha:=-\frac{b(0)}{\theta_0(0)}\right).
\end{cases}
$$
\end{lemma}

\smallskip

\begin{proof} Same as the proof of Lemma \ref{lem6.4}. \  \end{proof}

\medskip

\begin{lemma}\label{lem6.6} Let
$$
\Phi_-=\begin{bmatrix} z&\theta_1 \overline{b}\\
\theta_0 \overline{a}& z\end{bmatrix} \quad (a \in \mathcal
H(\theta_0),\ b\in \mathcal H(\theta_1)\ \hbox{and $\theta_j$ inner}
\ (j=0,1)). \
$$
If $\theta_0=z \theta_0'$ and $\theta_1 =z\theta_1'$ then $\hbox{\rm
ker}\,H_{\Phi_-^*}=\Delta H^2_{\mathbb C^2}$, where
$$
\Delta=
\begin{cases}
\qquad\qquad
\begin{bmatrix} \theta_1 &0\\0&\theta_0\end{bmatrix}
&\Bigl((ab)(0)\neq (\theta_0'\theta_1')(0)\Bigr);\\
\small{\frac{1}{\sqrt{|\alpha|^2+1}}}
\begin{bmatrix} \theta_1&\alpha\theta_1^{\prime}\\
-\overline{\alpha}\theta_0&\theta_0^{\prime}\end{bmatrix} \quad
&\Bigl((ab)(0)=(\theta_0'\theta_1')(0)\Bigr) \quad
\left(\alpha:=-\frac{a(0)}{\theta_1'(0)}\right).
\end{cases}
$$
\end{lemma}
\medskip

\begin{remark}
Since $a(0), b(0)\ne 0$, the second part of the above
assertion makes sense because by assumption, $\theta_0'(0),
\theta_1'(0)\ne 0$. \qed
\end{remark}

\medskip

\begin{proof}[Proof of Lemma \ref{lem6.6}] \
Observe that for $f,g\in H^2$,
$$
\aligned \Phi_-^*\begin{bmatrix} f\\g\end{bmatrix}=
\begin{bmatrix} \overline{z}&\overline{z\theta_0'}a\\
\overline{z\theta_1'}b&\overline{z}\end{bmatrix}
\begin{bmatrix} f\\g\end{bmatrix}\in H^2_{\mathbb C^2}
&\Longleftrightarrow \begin{bmatrix}
\overline{z}(f+\overline{\theta_0'}ag)\\ \overline{z}
(g+\overline{\theta_1'}bf) \end{bmatrix} \in H^2_{\mathbb C^2}\\
&\Longrightarrow
\begin{bmatrix}f+\overline{\theta_0'}ag\\g+\overline{\theta_1'}bf
\end{bmatrix} \in zH^2_{\mathbb C^2}\\
 &\Longrightarrow \begin{bmatrix}
\overline{\theta_0'}ag\\ \overline{\theta_1'}bf\end{bmatrix} \in H^2_{\mathbb C^2}\\
 &\Longrightarrow g=\theta_0'g_1\ \hbox{and}\ f=\theta_1' f_1\ \hbox{for some}\
 g_1,f_1\in H^2.
\endaligned
$$
Thus if $\Phi_-^*\left[\begin{smallmatrix}
f\\g\end{smallmatrix}\right]\in H^2_{\mathbb C^2}$ then
$\overline{z}(\theta_1'f_1+ag_1)\in H^2$ and
$\overline{z}(\theta_0'g_1+bf_1)\in H^2$, so that
$$
\theta_1'(0)f_1(0)=-a(0)g_1(0)\quad\hbox{and}\quad
\theta_0'(0)g_1(0)=-b(0)f_1(0).
$$
If $(ab)(0)=(\theta_0'\theta_1')(0)$, then
$$
\theta_1'(0)f_1(0)=-a(0)g_1(0) \Longleftrightarrow
\theta_0'(0)g_1(0)=-b(0)f_1(0).
$$
Put
$$
\Delta:= \small{\frac{1}{\sqrt{|\alpha|^2+1}}}
\begin{bmatrix} \theta_1 &\alpha\theta_1^{\prime}\\
-\overline{\alpha}\theta_0&\theta_0^{\prime}\end{bmatrix}.
$$
Then we can see that $\Delta$ is inner and $\hbox{\rm
ker}\,H_{\Phi_-^*}=\Delta H^2_{\mathbb C^2}$. \

If $(ab)(0)\neq (\theta_0'\theta_1')(0)$, put
$\theta_0=z^m\theta_0''$ and $\theta_1 =z^n\theta_1''$
($\theta_0''(0), \theta_1''(0)\ne 0$). \  If $n=m$ then
$$
\Phi_-
=\begin{bmatrix}z&z^n\theta_1''\overline{b}\\z^m\theta_0''\overline{a}&z\end{bmatrix}
=I_{z^n \theta_0''\theta_1''}
\,\begin{bmatrix}\overline{z}^{n-1}\overline{\theta_0''\theta_1''}
&\overline{\theta_0''b}\\
\overline{\theta_1'' a}& \overline{z}^{n-1}\overline{\theta_0''
\theta_1''}\end{bmatrix}\equiv  I_{z^n \theta_0'' \theta_1''}\,B^*.
$$
Since $B(0)$ is invertible, it follows from Lemma \ref{lem2.2} that
$I_{z^n}$ and $B$ are coprime. \  Observe that
$$
\Phi_-^*\begin{bmatrix} z^n f\\ z^n g\end{bmatrix}
= \begin{bmatrix} \overline{z}&\overline{z}^n \overline{\theta_0''} a\\
\overline{z}^n \overline{\theta_1'' }b&
\overline{z}\end{bmatrix}\begin{bmatrix} z^nf\\z^ng\end{bmatrix} \in
H^2_{\mathbb C^2} \Longleftrightarrow \begin{bmatrix} z^{n-1}f +
\overline{\theta_0''}ag\\ \overline{\theta_1''}bf+z^{n-1}g
\end{bmatrix} \in H^2_{\mathbb C^2},
$$
which implies
$$
\hbox{$f=\theta_1'' f_1$ and $g=\theta_0'' g_1$ for some $f_1,g_1\in
H^2$.}
$$
We thus have
$$
\hbox{\rm ker}\,
H_{\Phi_-^*}=\begin{bmatrix}\theta_1&0\\0&\theta_0\end{bmatrix}H^2_{\mathbb
C^2}
$$
If instead $n \neq m$, then
$$
\begin{aligned}
\Phi_-^*\begin{bmatrix} f\\ g\end{bmatrix}
&=\begin{bmatrix} \overline{z}&\overline{z}^m\overline{\theta_0''}a\\
\overline{z}^n\overline{\theta_1''}b&\overline{z} \end{bmatrix}
\begin{bmatrix}f\\g\end{bmatrix}=\begin{bmatrix}\overline{z}(f+\overline{z}^{m-1}\overline{
\theta_0''}ag)\\ \overline{z}(g+\overline{z}^{n-1} \overline{\theta_1''}bf)
\end{bmatrix} \in H^2_{\mathbb C^2}\\
&\Longleftrightarrow \begin{cases}
f+\overline{z}^{m-1}\overline{\theta_0''}ag \in
zH^2\\g+\overline{z}^{n-1}\overline{\theta_1''}bf \in
zH^2,\end{cases}
\end{aligned}
$$
which implies
$$
f=z^{n-1}\theta_1''f_1\quad\hbox{and}\quad g=z^{m-1}\theta_0''g_1.
$$
Suppose $n > m$, and hence $n \geq 2$. \  We thus have
$$
z^{n-2}\theta_1''f_1+\overline{z}ag_1 \in H^2 \Longrightarrow
\overline{z}ag_1 \in H^2\Longrightarrow g_1=zg_2 \Longrightarrow
g=\theta_0g_2.
$$
In turn,
$$
g+\overline{z}^{n-1}\overline{\theta_1''}bf \in zH^2\ \Longrightarrow\
z^m\theta_0''g_2 + b f_1 \in z H^2 \Longrightarrow f_1=zf_2,
$$
which implies
$$
\hbox{ker}\,H_{\Phi_-^*}=\begin{bmatrix}\theta_1&0\\0&\theta_0\end{bmatrix}H^2_{\mathbb
C^2}.
$$
If $m>n$, a similar argument gives the result.
\end{proof}

\bigskip

%
%
%


We are ready for:

\bigskip

\begin{proof}[Proof of Theorem \ref{thm6.1}] \
Clearly (i) $\Rightarrow$ (ii) and (ii) $\Rightarrow$ (iii). \ Moreover, a simple calculation shows that (iv) $\Rightarrow$ (i).

(iii) $\Rightarrow$ (iv): \ Write
$$
\Phi\equiv \left[\begin{matrix} \overline z&\varphi\\ \psi&\overline
z\end{matrix}\right] \equiv \Phi_-^*+\Phi_+
=\begin{bmatrix} z& \psi_-\\
\varphi_-& z\end{bmatrix}^{\ast}+\begin{bmatrix} 0&\varphi_+\\
\psi_+& 0\end{bmatrix}
$$
and assume that $T_\Phi$ is $2$-hyponormal. \  Since $\hbox{ker}\, [T^*, T]$
is invariant under $T$ for every $2$-hyponormal
operator $T \in \mathcal{B}(\mathcal{H})$, we note that Theorem \ref{thm4.1}
and Lemma \ref{lem6.3} hold for $2$-hyponormal operators $T_{\Phi}$. \
We claim that
\medskip
\begin{align}
&|\varphi|=|\psi|, \; \textrm{and} \label{6.5}\\
&\Phi \ \hbox{and} \ \Phi^* \ \hbox{are of bounded type.}
\label{6.6}
\end{align}

\noindent Indeed, if $T_\Phi$ is hyponormal then $\Phi$ is normal,
so that a straightforward calculation gives (\ref{6.5}). \  Also
there exists a matrix function
$K\equiv \left[\begin{smallmatrix} k_1&k_2\\
k_3&k_4\end{smallmatrix}\right] \in H_{M_2}^\infty$ with
$||K||_\infty\le 1$ such that $\Phi-K\Phi^*\in H^\infty_{M_2}$,
i.e.,
$$
\left[\begin{matrix} \overline{z}&\overline{\varphi_-}\\
\overline{\psi_-}& \overline
z\end{matrix}\right]\, -\, \left[\begin{matrix} k_1&k_2\\
k_3&k_4\end{matrix}\right]\, \left[\begin{matrix} 0&\overline {\psi_+}\\
\overline {\varphi_+}&0\end{matrix}\right] \in H^2_{M_2},
$$
which implies that
$$
\begin{cases}
H_{\overline{z}}=H_{k_2\overline {\varphi_+}}=H_{\overline {\varphi_+}}T_{k_2};\\
H_{\overline{\varphi_-}}=H_{k_1\overline {\psi_+}}=H_{\overline {\psi_+}}T_{k_1};\\
H_{\overline{\psi_-}}=H_{k_4\overline {\varphi_+}}=H_{\overline {\varphi_+}}T_{k_4};\\
H_{\overline z}=H_{k_3\overline{\psi_+}}=H_{\overline
{\psi_+}}T_{k_3}.
\end{cases}
$$
If $\overline{\varphi_+}$ is not of bounded type then
$\hbox{ker}\,H_{\overline{\varphi_+}}=0$, so that $k_2=0$, a
contradiction; and if $\overline{\psi_+}$ is not of  bounded type
then $\hbox{ker}\,H_{\overline{\psi_+}}=0$, so that $k_3=0$, a
contradiction. \  Therefore we should have $\Phi^*$ of bounded type.
\ Since $T_{\Phi}$ is hyponormal, $\Phi$ is also of bounded type,
giving (\ref{6.6}). \  Thus we can write
$$
\varphi_-:=\theta_0\overline{a}\quad\hbox{and}\quad \psi_-:=\theta_1
\overline{b}\qquad (a\in \mathcal H(\theta_0),\  b \in \mathcal
H(\theta_1)).
$$
Put
$$
\theta_0=z^m\theta_0^{\prime}\quad\hbox{and}\quad \theta_1=z^n
\theta_1^{\prime}\qquad (m,n\ge 0;\ \theta_0^{\prime}(0)\ne 0 \ne
\theta_1^{\prime}(0)).
$$
We now claim that
\begin{equation} \label{6.7}
m=n=0\quad\hbox{or}\quad m=n=1.
\end{equation}
We split the proof of (\ref{6.7}) into three cases. \

\medskip


{\bf Case 1} ($m\ne 0$ and $n=0$)\quad In this case, we have
$a(0)\ne 0$ because $\theta_0(0)=0$ and $\theta_0$ and $a$ are
coprime. \  We first claim that $m=1$. \  To show this we assume to
the contrary that $m \geq 2$. \  Write
$$
\alpha:=-\frac{a(0)}{\theta_1(0)}\quad\hbox{and}\quad
\nu:=\frac{1}{\sqrt{|\alpha|^2+1}}.
$$
By Lemma \ref{lem6.4}, we can write:
\medskip
$$
\Phi_- =\begin{bmatrix} z&\theta_1 \overline{b}\\
\theta_0\overline{a}&z \end{bmatrix} =\Delta_2 B_r^*\quad
\hbox{(right coprime factorization)}\,,
$$
where
$$
\Delta_2:=
\nu \begin{bmatrix} z\theta_1&\alpha \theta_1\\
-\overline{\alpha}\theta_0&z^{m-1}\theta_0'
\end{bmatrix}
\quad\hbox{and}\quad B_r:= \nu
\begin{bmatrix} \theta_1 -\overline{\alpha}a&\alpha\theta_1\overline{z}+a\overline{z}\\
zb-\overline\alpha z^{m-1}\theta_0^{\prime}&\alpha b+
z^{m-2}\theta_0'\end{bmatrix}\,.
$$

\medskip
\noindent To get the left coprime factorization of $\Phi_-$,
applying Lemma \ref{lem6.4} for $\widetilde{\Phi_-}$ gives
\medskip
$$
\widetilde{\Phi_-}
=\begin{bmatrix} z&\widetilde{\theta_0} \overline{\widetilde{a}}\\
\widetilde{\theta_1}\overline{\widetilde{a}}&z \end{bmatrix}
=\widetilde\Omega_2\widetilde B_{\ell}^*\quad\hbox{(right coprime
factorization)}\,,
$$
where
$$
\Omega_2:= \nu
\begin{bmatrix} z^{m-1}\theta_0'&\alpha \theta_1\\
-\overline{\alpha} \theta_0 &z\theta_1
\end{bmatrix}
\quad\hbox{and}\quad B_{\ell}:= \nu \begin{bmatrix}
z^{m-2} \theta_0'+ \alpha b&a\overline z+\alpha\theta_1 \overline z\\
-\overline\alpha z^{m-1}\theta_0^{\prime}+zb& -\overline\alpha a+
\theta_1\end{bmatrix}\,,
$$
which gives
$$
\Phi_- =\begin{bmatrix} z&\theta_1 \overline{b}\\
\theta_0\overline{a}&z \end{bmatrix} = B_{\ell}^*\Omega_2
\quad\hbox{(left coprime factorization)}\,.
$$
On the other hand, since $\Phi_-^*-K\Phi_+^*\in H^2_{M_n}$, and
hence
$$
\begin{bmatrix} \overline{z}&\overline{\theta}_0 a\\
\overline{\theta}_1 b &\overline{z} \end{bmatrix}-\begin{bmatrix}
k_2\overline{\varphi_+}&k_1\overline{\psi_+}\\k_4
\overline{\varphi_+}&k_3\overline{\psi_+}\end{bmatrix} \in
H^2_{M_2},
$$
we have
$$
\begin{cases}
\overline{z}-k_2 \overline{\varphi_+}\in H^2,\quad
\overline{\theta}_1 b-k_4\overline{\varphi_+}\in H^2\\
\overline{z}-k_3 \overline{\psi_+}\in H^2,\quad
\overline{\theta}_0 a-k_1\overline{\psi_+}\in H^2,
\end{cases}
$$
which via Cowen's Theorem gives that the following Toeplitz
operators are all hyponormal:
$$
T_{\overline z+\varphi_+},\ \ T_{\overline{\theta}_1 b + \varphi_+},\
\ T_{\overline z+\psi_+},\ \ T_{\overline{\theta}_0 a + \psi_+}.
$$
Thus by Proposition \ref{pro3.2} we can see that
$$
\varphi_+=z \theta_1 \theta_3\overline{d}\ \ \hbox{and}\ \
\psi_+=\theta_0\theta_2\overline{c}\ \ \hbox{for some inner
functions $\theta_2,\theta_3$.}
$$
We thus have
$$
\Phi_+=\begin{bmatrix}z \theta_1
\theta_3&0\\0&\theta_0\theta_2\end{bmatrix}
\begin{bmatrix}0& c\\ d&0
\end{bmatrix}^{\ast}\equiv \Delta_2\Delta_0A_r^*  \quad \hbox{(right coprime factorization)},
$$
where
$$
\begin{aligned}
A_r&:=\begin{bmatrix}0& c\\ d&0\end{bmatrix};\\
\Delta_0&:=\nu\begin{bmatrix} 1&-\alpha \\
    \overline{\alpha}z &z \end{bmatrix}\begin{bmatrix} \theta_3&0 \\
       0&\theta_2\end{bmatrix};\\
\Delta_2&:= \nu\begin{bmatrix} z\theta_1 & \alpha\theta_1\\
   -\overline\alpha \theta_0& z^{m-1} \theta_0^\prime \end{bmatrix}=
            \begin{bmatrix}\theta_1&0\\ 0&z^{m-1} \theta_0'\end{bmatrix}\cdot
        \nu\begin{bmatrix} z&\alpha\\
        -\overline{\alpha}z&1\end{bmatrix}\,.
\end{aligned}
$$
Write
$$
\theta_2=z^p\theta_2'\quad\hbox{and}\quad \theta_3=z^q\theta_3'\quad
(p,q\ge 0;\ \theta_2'(0),\theta_3'(0)\ne 0).
$$
If $q+1 \geq m+p$, then $\hbox{LCM}\,(z \theta_1 \theta_3, \
\theta_0 \theta_2)$ is an inner divisor of
$z^{q+1}\theta_0'\theta_1\theta_2'\theta_3'$. \
Thus we can write
$$
\Phi_+=\begin{bmatrix} 0&x\\ y&0\end{bmatrix}^*
I_{z^{q+1}\theta_0'\theta_1\theta_2'\theta_3'} \equiv
A^*\Omega_1\Omega_2,
$$
where
$$
\begin{aligned}
A&:=\begin{bmatrix} 0&x\\ y&0\end{bmatrix}\ \ \hbox{(some $x,y\in
H^2$)};\\
\Omega_1&:=
\begin{bmatrix} z^{q+1-m}\theta_1\theta_2'\theta_3'&0\\
0&z^q\theta_0'\theta_2'\theta_3'
\end{bmatrix}\cdot
\nu\begin{bmatrix} z&-\alpha\\ \overline{\alpha}z&1\end{bmatrix}.
\end{aligned}
$$
It thus follows from Lemma \ref{lem6.3} that
\begin{equation}\label{6.7-1}
\Delta_2\mathcal{H}(\Delta_0)\ \subseteq\ \hbox{\rm cl
ran}\,[T_\Phi^*, T_\Phi]\ \subseteq\ \mathcal{H}(\Omega_1).
\end{equation}
But since in general, $\Theta_2\mathcal{H}(\Theta_1)\subseteq
\mathcal{H}(\Theta_1\Theta_2)$ for inner matrix functions
$\Theta_1,\Theta_2$, we have
$$
\nu\begin{bmatrix} 1&-\alpha\\ \overline\alpha z& z\end{bmatrix}
\mathcal{H}\Bigl(\begin{bmatrix} \theta_3&0 \\
0&\theta_2\end{bmatrix}\Bigr)
\subseteq \mathcal{H}(\Delta_0).
$$
Thus by (\ref{6.7-1}), we have
$$
\Delta_2\cdot \nu \begin{bmatrix} 1&-\alpha\\ \overline\alpha z&
z\end{bmatrix}
\mathcal{H}\Bigl(\begin{bmatrix} \theta_3&0 \\
0&\theta_2\end{bmatrix}\Bigr) \subseteq \mathcal{H}(\Omega_1),
$$
or equivalently,
\begin{equation}\label{6.7-2}
\begin{bmatrix}z\theta_1&0\\ 0&z^{m}
\theta_0'
\end{bmatrix}\mathcal H \Bigl(
\begin{bmatrix} \theta_3&0 \\
0&\theta_2\end{bmatrix}\Bigr)\\
\subseteq \mathcal H (\Omega_1).
\end{equation}
Since in general, $F\in\mathcal{H}(\Theta)$ if and only if
$\Theta^*F\in (H_{\mathbb C^n}^{2})^\perp$, (\ref{6.7-2}) implies
\begin{equation}\label{6.7-3}
\nu\begin{bmatrix} 1&\alpha\\-\overline{\alpha}z&z
\end{bmatrix} \mathcal H
\Bigl(\begin{bmatrix} \theta_3&0 \\
0&\theta_2\end{bmatrix}\Bigr)\\
\subseteq \mathcal H \Bigl(\begin{bmatrix} z^{q+1-m}\theta_2'\theta_3'&0\\
0&z^{q+1-m}\theta_2'\theta_3'
\end{bmatrix}\Bigr).
\end{equation}
Also since for inner matrix functions $\Theta_1,\Theta_2$ and any
closed subspace $F$ of $H^2_{\mathbb C^n}$,
$$\Theta_1F\subseteq
\mathcal{H}(\Theta_1\Theta_2)\ \ \hbox{and}\ \
\Theta_1\Theta_2=\Theta_2\Theta_1\ \Longrightarrow\ F\subseteq
\mathcal{H}(\Theta_1\Theta_2)\,,
$$
it follows from (\ref{6.7-3})
that
$$
\begin{bmatrix} \mathcal{H}(\theta_3)\\
\mathcal{H}(\theta_2)\end{bmatrix} \subseteq
\begin{bmatrix} \mathcal{H}(z^{q+1-m}\theta_2'\theta_3')\\
\mathcal{H}(z^{q+1-m}\theta_2'\theta_3')\end{bmatrix}.
$$
But since $\theta_3=z^q\theta_3'$, it follows that $q+1-m \ge q$,
giving a contradiction. \

\medskip

If instead $q+1 < m+p$, then $\hbox{LCM}\,(z \theta_1 \theta_3, \
\theta_0 \theta_2)$ is an inner divisor of
$z^{m+p}\theta_0'\theta_1\theta_2'\theta_3'$. \  Thus we can write
$$
\Phi_+=\begin{bmatrix} 0&x\\ y&0\end{bmatrix}^*
I_{z^{m+p}\theta_0'\theta_1\theta_2'\theta_3'} \equiv
A_1^*\Omega_1^\prime\Omega_2,
$$
where
$$
\begin{aligned}
A_1&:=\begin{bmatrix} 0&x\\ y&0\end{bmatrix}\ \ \hbox{(some $x,y\in
H^2$)};\\
\Omega_1^\prime&:=
\begin{bmatrix} z^{p}\theta_1\theta_2'\theta_3'&0\\
0&z^{m+p-1}\theta_0'\theta_2'\theta_3'
\end{bmatrix}\cdot
\nu \begin{bmatrix} z&-\alpha\\
\overline{\alpha}z&1 \end{bmatrix}\,.
\end{aligned}
$$
It thus follows from Lemma \ref{lem6.3} with $\Omega_1^\prime$ in
place of $\Omega_1$ that
$$
\Delta_2\mathcal{H}(\Delta_0)\ \subseteq\ \hbox{\rm cl
ran}\,[T_\Phi^*, T_\Phi]\ \subseteq\ \mathcal{H}(\Omega_1^\prime).
$$
Since
$$
\begin{aligned}
\mathcal{H}(\Delta_0)
&=\mathcal{H}\Bigl(\nu\begin{bmatrix} 1&-\alpha \\
   \overline{\alpha}z &z \end{bmatrix}
   \begin{bmatrix} \theta_3&0 \\ 0&\theta_2\end{bmatrix}\Bigr)\\
&=\mathcal{H}\Bigl(\nu \begin{bmatrix} 1&-\alpha \\
   \overline{\alpha}z &z \end{bmatrix}\Bigr)\ \bigoplus\
  \nu \begin{bmatrix} 1&-\alpha \\
   \overline{\alpha}z &z \end{bmatrix}
   \mathcal{H}\Bigl(\begin{bmatrix} \theta_3&0 \\
   0&\theta_2\end{bmatrix}\Bigr),
\end{aligned}
$$
we have
\begin{equation}\label{6.8}
\Delta_2 \mathcal H \Bigl(\nu\begin{bmatrix} 1&-\alpha \\
             \overline{\alpha}z &z \end{bmatrix}\Bigr)
     \bigoplus
        \begin{bmatrix}z\theta_1&0\\ 0&z^{m} \theta_0'
\end{bmatrix}\mathcal H \Bigl(\begin{bmatrix} \theta_3&0 \\
0&\theta_2\end{bmatrix}\Bigr) \subseteq \mathcal H
(\Omega_1^\prime).
\end{equation}
Then by the same argument as
(\ref{6.7-2}) and (\ref{6.7-3}), we can see that
\begin{equation}\label{6.7-4}
\nu\begin{bmatrix}
1&\alpha\\-\overline{\alpha}z&z
\end{bmatrix}
\mathcal H \Bigl(\begin{bmatrix} \theta_3&0 \\
0&\theta_2\end{bmatrix}\Bigr)\\
\subseteq \mathcal H \Bigl(\begin{bmatrix} z^{p}\theta_2'\theta_3'&0\\
0&z^{p}\theta_2'\theta_3'
\end{bmatrix}\Bigr)\,,
\end{equation}
which gives
$$
z\mathcal H (\theta_2) \subseteq \mathcal H (z^p
\theta_2'\theta_3'),
$$
which by Lemma \ref{lem6.2} implies that $\theta_2$ should be a
constant. \ Thus (\ref{6.7-4}) can be written as
$$
\nu\begin{bmatrix} 1&\alpha\\-\overline{\alpha}z&z \end{bmatrix}
\mathcal H \Bigl(\begin{bmatrix} \theta_3&0 \\
0&1\end{bmatrix}\Bigr)\\
\subseteq \mathcal H \Bigl(\begin{bmatrix} \theta_3'&0\\
0&\theta_3'
\end{bmatrix}\Bigr),
$$
which gives $z \mathcal H (\theta_3) \subseteq \mathcal
H(\theta_3')$. \ It follows from again Lemma \ref{lem6.2} that
$\theta_3$ is a constant. \  Thus by (\ref{6.8}), we have
$$
\begin{bmatrix}\theta_1&0\\ 0&z^{m-1} \theta_0'
\end{bmatrix}\nu\begin{bmatrix} z&\alpha\\ -\overline{\alpha}z&1
\end{bmatrix}\mathcal H \Bigl(\nu\begin{bmatrix} 1&-\alpha \\
\overline{\alpha}z &z \end{bmatrix}\Bigr)
\subseteq \mathcal H \Bigl(\begin{bmatrix} \theta_1&0\\
0&z^{m-1}\theta_0'
\end{bmatrix}\nu\begin{bmatrix} z&-\alpha\\
\overline{\alpha}z&1 \end{bmatrix}\Bigr),
$$
so that
$$
\nu\begin{bmatrix} z&\alpha\\ -\overline{\alpha}z&1
\end{bmatrix}\mathcal H \Bigl(\nu\begin{bmatrix} 1&-\alpha \\
\overline{\alpha}z &z
\end{bmatrix}\Bigr)\subseteq
\mathcal H \Bigl(\nu\begin{bmatrix} z&-\alpha\\
\overline{\alpha}z&1 \end{bmatrix}\Bigr),
$$
giving a contradiction by Lemma 5.3. \  Therefore we should have
$$
m=1,\quad\hbox{i.e.,}\quad \theta_0=z\theta_0^{\prime}.
$$
Thus by Lemmas \ref{6.4} and \ref{6.5}, we have
$$
\Phi_-=\begin{bmatrix} z&\theta_1\overline{b}\\
z\theta_0'\overline{a}& z\end{bmatrix}=\begin{bmatrix}
z\theta_1&0\\0&\theta_0\end{bmatrix}
\begin{bmatrix} \theta_1& a\\
zb&\theta_0'\end{bmatrix}^*\quad \hbox{(right coprime
factorization)}
$$
and
$$
\Phi_-=\begin{bmatrix} z&\theta_1\overline{b}\\
z\theta_0'\overline{a}& z\end{bmatrix}
=\begin{bmatrix} \theta_0'& a\\
zb&\theta_1\end{bmatrix}^* \begin{bmatrix}
\theta_0&0\\0&z\theta_1\end{bmatrix} \quad \hbox{(left coprime
factorization)}.
$$
Recall that
$$
\psi_+=\theta_0\theta_2\overline{c}\ \ \hbox{and}\ \
\varphi_+=z\theta_1\theta_3\overline{d}\ \ \hbox{for some inner
functions $\theta_2$ and $\theta_3$.}
$$
We can thus write
$$
\Phi_+=\begin{bmatrix}
0&z\theta_1\theta_3\overline{d}\\\theta_0\theta_2\overline{c}&0\end{bmatrix}=\begin{bmatrix}
z\theta_1\theta_3&0\\0&\theta_0\theta_2\end{bmatrix}\begin{bmatrix}0&c\\d&0\end{bmatrix}^*
\quad(\hbox{right coprime factorization}).
$$
Note that $\hbox{LCM}\,(z\theta_1 \theta_3, \theta_0\theta_2)$ is an
inner divisor of $\theta_0\theta_1\theta_2\theta_3$. \  Thus we can
write
$$
\Phi_+=\begin{bmatrix} 0&x\\ y&0\end{bmatrix}^*
I_{\theta_0\theta_1\theta_2\theta_3}\quad\hbox{($x,y\in H^2$).}
$$
It follows from Lemma \ref{lem6.3} that
$$
\begin{bmatrix} z\theta_1\mathcal H (\theta_3)\\\theta_0\mathcal H
(\theta_2) \end{bmatrix} \subseteq
\text{cl}\,\text{ran}\,[T_{\Phi}^*, T_{\Phi}]\subseteq
\begin{bmatrix} \mathcal H (\theta_1\theta_2\theta_3)\\ \mathcal H
(\theta_0'\theta_2\theta_3)\end{bmatrix},
$$
which implies
$$
z\mathcal H (\theta_3) \subseteq \mathcal H
(\theta_2\theta_3)\quad\hbox{and}\quad z\mathcal  H(\theta_2)
\subseteq \mathcal H (\theta_2\theta_3).
$$
By Lemma \ref{lem6.2},
\begin{equation}\label{6.9}
\begin{cases}
\hbox{either $\theta_3$ is constant or $z\theta_3$ is a divisor of $\theta_2\theta_3$;}\\
\hbox{either $\theta_2$ is constant or $z\theta_2$ is a divisor of
$\theta_2\theta_3$}\, ,
\end{cases}
\end{equation}
If $\theta_2$ or $\theta_3$ is not constant then it follows from
(\ref{6.9}) that $z$ is a divisor of $\theta_2$ and $\theta_3$. \
Thus we have $p,q\ge 1$. \  Let $N:=\hbox{max}(p,q)$. \  Then
$\hbox{LCM}\,(z \theta_1 \theta_3, \theta_0\theta_2)$ is an inner
divisor of $z^{N}\theta_0\theta_1\theta_2'\theta_3'$. \  Thus we can
write
$$
\Phi_+=\begin{bmatrix} 0&x\\ y&0\end{bmatrix}^*
I_{z^{N}\theta_0\theta_1\theta_2'\theta_3'}\quad\hbox{($x,y\in
H^2$).}
$$
It follows from Lemma \ref{lem6.3} that
$$
\aligned
\begin{bmatrix}
z\theta_1\mathcal H(z^q\theta_3')\\
\theta_0\mathcal H (z^p\theta_2')
\end{bmatrix}
\subseteq \text{cl}\,\text{ran}\,[T_{\Phi}^* , T_{\Phi}] \subseteq
\begin{bmatrix} \mathcal H (z^N\theta_1\theta_2'\theta_3')\\ \mathcal H
(z^N\theta_0'\theta_2'\theta_3')\end{bmatrix}\\
\Longrightarrow
\begin{cases}
\hbox{$z^q$ is a divisor of $z^{N-1}\theta_2'$}\\
\hbox{$z^p$ is a divisor of $z^{N-1}\theta_3'$},
\end{cases}
\endaligned
$$
giving a contradiction. \  Therefore
$$
\hbox{$\theta_2$ and $\theta_3$ are constant.}
$$
We observe that $\hbox{LCM}(z \theta_1 , \theta_0)$ is an inner
divisor of $z\theta_0'\theta_1$. \  It follows from Lemma
\ref{lem6.3} that
$$
\begin{bmatrix}\theta_1 H^2\\ \theta_0' H^2 \end{bmatrix}\ \subseteq\
\hbox{ker}\,[T_{\Phi}^* , T_{\Phi}].
$$
In particular, $\left[\begin{smallmatrix} 0\\
\theta_0'\end{smallmatrix}\right]\in \hbox{ker}\,[T_{\Phi}^* ,
T_{\Phi}]$. \  Observe that
$$
\Phi_-^* \begin{bmatrix} 0\\ \theta_0'\end{bmatrix}
=\begin{bmatrix} \overline{z}&\overline{\theta}_0 a\\
\overline{\theta}_1 b& \overline{z}\end{bmatrix}
\begin{bmatrix} 0\\ \theta_0'\end{bmatrix}
=\begin{bmatrix} \overline{z} a\\ \overline{z} \theta_0'
\end{bmatrix},
$$
so that
$$
H_{\Phi_-^*}\begin{bmatrix} 0\\ \theta_0'\end{bmatrix}
=\begin{bmatrix} a(0)\\ \theta_0'(0)\end{bmatrix}.
$$
We thus have
$$
H_{\Phi_-^*}^*H_{\Phi_-^*}\begin{bmatrix} 0\\ \theta_0'\end{bmatrix}
=\begin{bmatrix} \theta_0'(0)J(I-P)(\overline{\widetilde{\theta}}_1
\widetilde{b})+ a(0)\\ * \end{bmatrix}.
$$
A similar calculation shows that
$$
H_{\Phi_+^*}^*H_{\Phi_+^*}\begin{bmatrix} 0\\ \theta_0^{\prime}
\end{bmatrix}=\begin{bmatrix} 0\\ * \end{bmatrix}.
$$
Since $[T_{\Phi}^* ,
T_{\Phi}]=H_{\Phi_+^*}^*H_{\Phi_+^*}-H_{\Phi_-^*}^*H_{\Phi_-^*}$, it
follows that
$$
\theta_0'(0)J(I-P)(\overline{\widetilde{\theta}}_1 \widetilde{b})=-a(0).
$$
But since $a(0)\ne 0$, we must have that
$\overline{\widetilde{\theta_1}}\widetilde{b}\in \overline z H^2\cap
(H^2)^{\perp}$, which implies that $\theta_1=c\,z$ for a nonzero
constant $c$, giving a contradiction because $\theta_1(0)\ne 0$. \
Therefore this case cannot occur. \


\bigskip

{\bf Case 2} ($m=0$ and $n\ne 0$) \quad This case is symmetrical to
Case 1. \  Thus the proof is identical to that of Case 1. \
Therefore this case cannot occur either. \


\bigskip

{\bf Case 3} ($m\ne 0, n \neq 0$ and $m\ge 2$ or $n\ge 2$)\quad In
this case we have $a(0)\ne 0$ and $b(0)\ne 0$ because
$\theta_0(0)=\theta_1(0)=0$, $\theta_0$ and $a$ are coprime, and
$\theta_1$ and $b$ are coprime. \  By Lemma \ref{lem6.6} we have
$$
\Phi_-=\begin{bmatrix}\theta_1&0\\0& \theta_0\end{bmatrix}
\begin{bmatrix} z^{n-1}\theta_1'& a\\
b & z^{m-1}\theta_0'\end{bmatrix}^* \qquad \hbox{(right coprime
factorization)}.
$$
Similarly, we have
$$
\Phi_-=\begin{bmatrix} z^{m-1}\theta_0'& a\\
b& z^{n-1}\theta_1'\end{bmatrix}^*
\begin{bmatrix}\theta_0&0\\0&
\theta_1\end{bmatrix} \qquad \hbox{(left coprime factorization)}
$$
and
$$
\Phi_+=\begin{bmatrix}
\theta_1\theta_3&0\\0&\theta_0\theta_2\end{bmatrix}
\begin{bmatrix} 0& c\\
d&0\end{bmatrix}^* \qquad \hbox{(right coprime factorization)}.
$$
We then claim that
\begin{equation}\label{6.10}
\hbox{$\theta_2$ and $\theta_3$ are constant.}
\end{equation}
Assume $\theta_2$ is not constant. \  Put $\theta_2=z^p\theta_2'$
and $\theta_3=z^q\theta_3'$ ($\theta_2^{\prime}\ne 0\ne
\theta_3^{\prime}(0)$ and $p,q\ge 0$) and let
$N:=\hbox{max}\,(m+p,n+q). \ $ Then
$\hbox{LCM}(\theta_0\theta_2,\theta_1\theta_3)$ is a divisor of
$z^N\theta_0'\theta_1'\theta_2'\theta_3'$. \  Thus we can write
$$
\Phi_+=\begin{bmatrix} 0&x\\ y&0\end{bmatrix}^*
I_{z^N\theta_0'\theta_1'\theta_2'\theta_3'}\quad\hbox{($x,y\in
H^2$).}
$$
By Lemma \ref{lem6.3} we have
$$
\begin{bmatrix}\theta_1 \mathcal H(\theta_3)\\ \theta_0 \mathcal
H(\theta_2)\end{bmatrix} \subseteq
\begin{bmatrix} \mathcal{H}(z^{N-m}\theta_1'\theta_2'\theta_3')\\
\mathcal{H}(z^{N-n}\theta_0'\theta_2'\theta_3')\end{bmatrix}.
$$
If $\theta_3$ is a constant, then $q=0$ and $m+p\le N-n$, giving a
contradiction because $m,n\ge 2$. \  If $\theta_3$ is not constant,
then $n+q\le N-m$ and $m+p\le N-n$, giving a contradiction because
$m,n\ge 2$. \  Therefore we should have that $\theta_2$ is constant.
\ Similarly, we can show that $\theta_3$ is also constant. \  This
proves (\ref{6.10}). \

We now suppose $n \leq m$. \  By Lemma 5.3 we have
$$
\begin{bmatrix}\theta_1'H^2\\z^{m-n} \theta_0'H^2 \end{bmatrix} \subseteq
\hbox{ker}\,[T_{\Phi}^* , T_{\Phi}].
$$
In particular, $\left[\begin{smallmatrix}\theta_1'\\0
\end{smallmatrix}\right] \in \hbox{ker}\,[T_{\Phi}^* , T_{\Phi}]$. \
Observe that
$$
\Phi_-^* \begin{bmatrix} \theta_1'\\ 0\end{bmatrix} =
\begin{bmatrix} \overline{z}&\overline{\theta}_0 a\\
\overline{\theta}_1 b& \overline{z}
\end{bmatrix}\begin{bmatrix}\theta_1'\\0
\end{bmatrix}=\begin{bmatrix} \overline{z}
\theta_1'\\\overline{z}^n b \end{bmatrix},
$$
so that
$$
H_{\Phi_-^*}\begin{bmatrix}\theta_1'\\0
\end{bmatrix}=\begin{bmatrix}\theta_1'(0)\\z^{n-1}\overline{\widetilde{b}}_1
\end{bmatrix}  \quad (b_1:=P_{\mathcal H(z^n)}(b)).
$$
Put $b_3:=z^{n-1}\overline{\widetilde{b}}_1$. \  We also have
$$
\widetilde{\Phi_-^*} \begin{bmatrix} \theta_1'(0)\\ b_3\end{bmatrix}
=
\begin{bmatrix} \overline{z}&\overline{\widetilde{\theta}}_1\widetilde{b}\\
\overline{\widetilde{\theta}}_0\widetilde{a}& \overline{z}
\end{bmatrix}\begin{bmatrix}\theta_1'(0)\\b_3
\end{bmatrix}=\begin{bmatrix} *\\\theta_1'(0)\overline{\widetilde{\theta}}_0
\widetilde{a}+ b_3\overline{z}\end{bmatrix},
$$
so that
$$
H_{\Phi_-^*}^*H_{\Phi_-^*}\begin{bmatrix} \theta_1'\\0 \end{bmatrix}
=\begin{bmatrix} \ast\\
\theta_1'(0)J(I-P)(\overline{\widetilde{\theta}}_0
\widetilde{a})+b_3(0)
\end{bmatrix}.
$$
A similar calculation shows that
$$
H_{\Phi_+^*}^*H_{\Phi_+^*}\begin{bmatrix} \theta_1'\\0
\end{bmatrix}=\begin{bmatrix}*\\0\end{bmatrix}.
$$
It thus follows that
$$
\theta_1'(0)J(I-P)(\overline{\widetilde{\theta}}_0
\widetilde{a})=-b_3(0) \Longrightarrow
\overline{\widetilde{\theta}}_0 \widetilde{a} \in \overline{z}H^2
\Longrightarrow \overline{\theta}_0 a\in
\overline{z}H^2\Longrightarrow \overline{\theta}_0 a \in
\overline{z} H^2 \cap \left({H^2}\right)^{\perp}.
$$
Since $n\le m$ and ($m\ge 2$ or $n\ge 2$), it follows that $m\ge 2$.
\ Thus $\overline{\theta_0}\,a\in \overline z H^2\,\cap\,
(H^2)^{\perp}$ implies $\overline{z}^{m-1} \overline{\theta_0^{\prime}}\,a=c$
(a constant), which forces $a=0$, giving a contradiction. \  If
instead $n >m$ then the same argument leads a contradiction. \
Therefore this case cannot occur. \  This completes the proof of
(\ref{6.7}). \

\vskip .5cm

Now in view of (\ref{6.7}) it suffices to consider the case $m=n=0$
and the case $m=n=1$. \

\bigskip

{\bf Case A ($m=n=0$)}.\quad In this case, we first claim that
\begin{equation}\label{6.11}
\varphi_-=\psi_-=0,\quad \hbox{i.e.,\ \  $\varphi$ and $\psi$ are
analytic.}
\end{equation}
Put $\theta:=\hbox{GCD}(\theta_0, \theta_{1})$. \  Then
$\theta_0=\theta_0'\theta$ and $\theta_1=\theta_1'\theta$ for some
inner functions $\theta_0', \theta_1'$, and hence
$\hbox{LCM}(\theta_0,\theta_1)=\theta\theta_0'\theta_1'$. \  We thus
have
$$
\Phi_- =I_{z\theta\theta_0'\theta_1'}\,\begin{bmatrix}
\theta\theta_0'\theta_1'&z\theta_1'a\\z\theta_0'b&\theta\theta_0'\theta_1'
\end{bmatrix}^*\equiv I_{z\theta\theta_0'\theta_1'}\, B^*\,.
$$
Since $B(0)$ is invertible it follows from Lemma \ref{lem2.2} that
$I_z$ and $B$ are coprime. \  Observe that
$$
\begin{aligned}
\Phi_-^* \begin{bmatrix} zf\\ zg\end{bmatrix} =
\begin{bmatrix} \overline{z}&\overline{\theta}_0 a\\
\overline{\theta}_1 b& \overline{z}\end{bmatrix}
\begin{bmatrix}zf\\zg \end{bmatrix} \in H^2_{\mathbb C^2}
&\Longleftrightarrow \begin{bmatrix} \overline{\theta}_0 azg\\
\overline{\theta}_1 bzf \end{bmatrix} \ \in H^2_{\mathbb C^2}\\
&\Longleftrightarrow g \in \theta_0H^2, \ f \in \theta_1 H^2,
\end{aligned}
$$
which implies
$$
\hbox{ker}\,H_{\Phi_-^*}=\begin{bmatrix} z\theta_ 1&0\\0&z \theta_0
\end{bmatrix}H^2_{\mathbb C^2}.
$$
We thus have
$$
\Phi_-=\begin{bmatrix} z \theta_1&0\\0&z\theta_0 \end{bmatrix}
\begin{bmatrix} \theta_1&za\\zb&\theta_0
\end{bmatrix}^*  \quad(\hbox{right coprime factorization}).
$$
To get the left coprime factorization of $\Phi_-$, we take
$$
\widetilde{\Phi_-}=\begin{bmatrix} z
\widetilde{\theta_0}&0\\0&z\widetilde{\theta_1}
\end{bmatrix}
\begin{bmatrix} \widetilde{\theta_0}& z\widetilde{b}\\
z\widetilde{a}&\widetilde{\theta_1}
\end{bmatrix}^* \quad(\hbox{right coprime factorization}),
$$
which implies
$$
\Phi_-=\begin{bmatrix}\theta_0&za\\zb&\theta_1
\end{bmatrix}^*\begin{bmatrix} z\theta_0&0\\0&z\theta_1
\end{bmatrix} \quad(\hbox{left coprime factorization}).
$$
Since $T_\Phi$ is hyponormal and hence,
$$
\hbox{ker}\, H_{\Phi_+^*} \subseteq \hbox{ker}\,H_{\Phi_-^*}
=\begin{bmatrix} z\theta_1 H^2\\ z\theta_0 H^2\end{bmatrix},
$$
it follows that
$$
\psi_+=z \theta_0\theta_2\overline{c}\ \ \hbox{and}\ \ \varphi_+=z
\theta_1\theta_3\overline{d}\ \ \hbox{for some inner functions
$\theta_2,\theta_3$.}
$$
We can thus write
$$
\Phi_+=\begin{bmatrix}
0&z\theta_1\theta_3\overline{d}\\z\theta_0\theta_2\overline{c}&0\end{bmatrix}=\begin{bmatrix}
z\theta_1\theta_3&0\\0&z\theta_0\theta_2\end{bmatrix}\begin{bmatrix}0&c\\d&0\end{bmatrix}^*
\quad(\hbox{right coprime factorization}).
$$
Observe that $\hbox{LCM}\,(z \theta_1 \theta_3, z \theta_0\theta_2)$
is an inner divisor of $z\theta_0\theta_1\theta_2\theta_3$. \  Thus
we can write
$$
\Phi_+=\begin{bmatrix} 0&x\\ y&0\end{bmatrix}^*
I_{z\theta_0\theta_1\theta_2\theta_3}\quad\hbox{($x,y\in H^2$).}
$$
It follows from Lemma \ref{lem6.3} that
$$
\begin{bmatrix} z\theta_1\mathcal H (\theta_3)\\z\theta_0\mathcal H
(\theta_2) \end{bmatrix} \subseteq
\text{cl}\,\text{ran}\,[T_{\Phi}^*, T_{\Phi}]\subseteq
\begin{bmatrix} \mathcal H (\theta_1\theta_2\theta_3)\\ \mathcal H
(\theta_0\theta_2\theta_3)\end{bmatrix},
$$
which implies that
$$
z\mathcal H (\theta_3) \subseteq \mathcal H
(\theta_2\theta_3)\quad\hbox{and}\quad z \mathcal  H(\theta_2)
\subseteq \mathcal H (\theta_2\theta_3).
$$
Thus the same argument as in (\ref{6.9}) shows that
$$
\hbox{$\theta_2$ and $\theta_3$ are constant.}
$$
We now observe that $\hbox{LCM}(z \theta_1 , z \theta_0)$ is an
inner divisor of $z\theta_0\theta_1$. \  Thus we can write
$$
\Phi_+=\begin{bmatrix} 0&x\\ y&0\end{bmatrix}^*
I_{z\theta_0\theta_1}\quad\hbox{($x,y\in H^2$).}
$$
It follows from Lemma \ref{lem6.3} that
$\left[\begin{smallmatrix}\theta_1 H^2\\ \theta_0H^2
\end{smallmatrix}\right] \subseteq \hbox{ker}\,[T_{\Phi}^* ,
T_{\Phi}]$. \  In particular, $\left[\begin{smallmatrix}\theta_1\\0
\end{smallmatrix}\right] \in \hbox{ker}\,[T_{\Phi}^* , T_{\Phi}]$. \
Observe that
$$
\Phi_-^* \begin{bmatrix} \theta_1\\ 0\end{bmatrix}=
\begin{bmatrix} \overline{z}&\overline{\theta}_0 a\\
\overline{\theta}_1 b& \overline{z}
\end{bmatrix}\begin{bmatrix}\theta_1\\0
\end{bmatrix}=\begin{bmatrix} \overline{z} \theta_1\\b \end{bmatrix},
$$
so that
$$
H_{\Phi_-^*}\begin{bmatrix}\theta_1\\0
\end{bmatrix}=\begin{bmatrix}\theta_1(0)\\0
\end{bmatrix}.
$$
We thus have
$$
H_{\Phi_-^*}^*H_{\Phi_-^*}\begin{bmatrix} \theta_1\\0 \end{bmatrix}=\begin{bmatrix}\theta_1(0)\\
\theta_1(0)J(I-P)(\overline{\widetilde{\theta}}_0 \widetilde{a})
\end{bmatrix}.
$$
A similar calculation shows that
$$
H_{\Phi_+^*}^*H_{\Phi_+^*}\begin{bmatrix} \theta_1\\0
\end{bmatrix}=\begin{bmatrix}*\\0\end{bmatrix}.
$$
It thus follows that
$$
\theta_1(0)J(I-P)(\overline{\widetilde{\theta}}_0 \widetilde{a})=0
\Longrightarrow \overline{\widetilde{\theta}}_0 \widetilde{a} \in
H^2 \Longrightarrow \overline{\theta}_0 a\in H^2\Longrightarrow
\overline{\theta}_0 a \in H^2\cap \left(H^2\right)^{\perp},
$$
which implies that $a=0$ and hence $\phi$ is analytic. \  Similarly,
we can show that $\psi$ is also analytic. \  This gives
(\ref{6.11}). \

Now since by (\ref{6.11}), $\varphi, \psi\in H^\infty$ and
$|\varphi|=|\psi|$, we can write $\varphi=\theta_1 a$ and
$\psi=\theta_2 a$, where the $\theta_i$ are inner functions and $a$
is an outer function. \  Observe that
$$
\Phi_- \equiv B^*\Theta_2,
$$
where $B\equiv \left[\begin{smallmatrix} 1&0\\
0&1\end{smallmatrix}\right]$ and $\Theta_2\equiv
\left[\begin{smallmatrix} z&0\\ 0&z\end{smallmatrix}\right]$ are
coprime. \  Thus our symbol satisfies all the assumptions of Theorem
\ref{thm4.1}. \  Thus by Theorem \ref{thm4.1}, since $T_\Phi$ is
$2$-hyponormal then $T_\Phi$ must be normal. \  We thus have
\begin{equation}\label{6.12}
H_{\Phi_+^*}^*H_{\Phi_+^*}=H_{\Phi_-^*}^*H_{\Phi_-^*}.
\end{equation}
Now observe that
$$
\Phi_+=\begin{bmatrix}0&\varphi\\ \psi&0
\end{bmatrix}\quad\hbox{and}\quad
\Phi_-=\begin{bmatrix}{z}&0\\
0&{z}\end{bmatrix}.
$$
Since $T_\Phi$ is normal we have
$$
\begin{bmatrix}
H_{\overline{\varphi}}^*H_{\overline{\varphi}}&0\\0&H_{\overline{\psi}}^*H_{\overline{\psi}}
\end{bmatrix}=\begin{bmatrix}
H_{\overline{z}}&0\\0&H_{\overline{z}}\end{bmatrix},
$$
which implies
\begin{equation}\label{6.13}
H_{\overline{\varphi}}^*H_{\overline{\varphi}}=H_{\overline{z}}
=H_{\overline{\psi}}^*H_{\overline{\psi}},
\end{equation}
which says that $H_{\overline{\varphi}}$ and $H_{\overline{\psi}}$
are both rank-one operators. \  Now remember that if $T$ is a
rank-one Hankel operator then there exist $\omega \in \mathbb D$ and
a constant $c$ such that $T=c\,(k_{\overline{\omega}} \otimes
k_{\omega})$, where $k_{\omega}:=\frac{1}{1-\overline\omega z}$ is
the reproducing kernel for $\omega$. \  Note that
$k_{\overline\omega}\otimes k_\omega$ is represented by the matrix
$$
\begin{bmatrix}
1&\omega&\omega^2&\omega^3&\hdots\\
\omega&\omega^2&\omega^3&\hdots\\
\omega^2&\omega^3&\hdots\\
\omega^3\\
\vdots
\end{bmatrix}.
$$
By (\ref{6.13}) we have that $\omega=0$. \  We thus have
\begin{equation}\label{6.14}
\varphi=e^{i\theta_1} z+\beta_1\quad\hbox{and}\quad
\psi=e^{i\theta_2}z+ \beta_2\qquad (\beta_1,\beta_2\in\mathbb{C},\
\theta_1, \theta_2\in [0,2\pi)).
\end{equation}
But since $|\varphi|=|\psi|$, we have
\begin{equation}\label{6.15}
\varphi=e^{i\theta} z+\beta\quad\hbox{and}\quad
\psi=e^{i\omega}\varphi\qquad (\beta\in \mathbb C,  \ \theta, \omega
\in [0, 2\pi)).
\end{equation}

\vskip.5cm

{\bf Case B ($m=n=1$)}\quad In this case, $\theta_0=z\theta_0'$ and
$\theta_1=z \theta_1'$ ($\theta_0'(0), \theta_1'(0)\ne 0$). \ We
thus have
$$
\varphi_-=z \theta_0'\overline{a}\quad\hbox{and}\quad \psi_-=z
\theta_1'\overline{b},
$$
so that
$$
\Phi_-=\begin{bmatrix} z& z \theta_1'\overline{b}\\
z \theta_0'\overline{a}& z \end{bmatrix}.
$$
There are two subcases to consider. \

\bigskip

{\bf Case B-1} $\left((ab)(0)\neq
(\theta_0'\theta_1')(0)\right)$.\quad In this case, we have, by
Lemma \ref{lem6.6},
$$
\aligned \Phi_-=\begin{bmatrix} z&
z\theta_1'\overline{b}\\z\theta_0'\overline{a}&z \end{bmatrix}
&=\begin{bmatrix} z\theta_1'&0\\0&z\theta_0'\end{bmatrix}
\begin{bmatrix}
\theta_1'& a\\ b& \theta_0'\end{bmatrix}^* \qquad \hbox{(right coprime decompositon)}\\
&=\begin{bmatrix}
\theta_0'& a\\
b & \theta_1'\end{bmatrix}^*
\begin{bmatrix} z\theta_0'&0\\0&z\theta_1'\end{bmatrix}
\qquad \hbox{(left coprime factorization)}
\endaligned
$$
and
$$
\Phi_+=\begin{bmatrix}
z\theta_1'\theta_3&0\\0&z\theta_0'\theta_2\end{bmatrix}\begin{bmatrix}
0&c\\ d&0\end{bmatrix}^* \qquad \hbox{(right coprime
factorization)}.
$$
Suppose $\theta_2$ is not constant. \  Put $\theta_2=z^p\theta_2'$
and $\theta_3=z^q \theta_3'$ ($p,q \in \mathbb N \cup \{0\}$). \ Let
$N:=\hbox{max}(p,q)$. \  Then $\hbox{LCM}(\theta_1 \theta_3, \
\theta_0\theta_2)$ is a divisor of
$z^{N+1}\theta_0'\theta_1'\theta_2'\theta_3'$. \  Thus we can write
$$
\Phi_+=\begin{bmatrix} 0&x\\ y&0\end{bmatrix}^*
I_{z^{N+1}\theta_0'\theta_1'\theta_2'\theta_3'}\quad\hbox{($x,y\in
H^2$).}
$$
By Lemma \ref{lem6.3} we have
$$
\begin{bmatrix} z \theta_1' \mathcal H (\theta_3)\\ z\theta_0'
\mathcal H (\theta_2) \end{bmatrix} \subseteq
\begin{bmatrix}
\mathcal H (z^N\theta_1'\theta_2'\theta_3')\\
 \mathcal H (z^N\theta_0'\theta_2'\theta_3')
\end{bmatrix}.
$$
If $\theta_3$ is a constant then $p+1 \leq N=p$, giving a
contradiction. \  If instead $\theta_3$ is not constant then $q+1\le
N$ and $p+1\le N$, giving a contradiction. \  The same argument
gives $\theta_3$ is a constant. \  Therefore $\theta_2$ and
$\theta_3$ should be constant. \  Note that $\hbox{LCM}(z \theta_1'
, z \theta_0')$ is an inner divisor of $z\theta_0'\theta_1'$. \ Thus
we can write
$$
\Phi_+=\begin{bmatrix} 0&x\\ y&0\end{bmatrix}^*
I_{z\theta_0'\theta_1'}\quad\hbox{($x,y\in H^2$).}
$$
It follows from Lemma \ref{lem6.3} that
\begin{equation}\label{6.16}
\begin{bmatrix}\theta_1'H^2\\ \theta_0'H^2 \end{bmatrix}\ \subseteq\
\hbox{ker}\,[T_{\Phi}^* , T_{\Phi}].
\end{equation}
In particular, $\left[\begin{smallmatrix}\theta_1'\\0
\end{smallmatrix}\right]\in \hbox{ker}\,[T_{\Phi}^* , T_{\Phi}]$. \
Observe that
$$
\Phi_-^* \begin{bmatrix}\theta_1'\\0\end{bmatrix}
=\begin{bmatrix} \overline{z}&\overline{\theta}_0a\\
\overline{\theta}_1b& \overline{z}\end{bmatrix}
\begin{bmatrix}\theta_1'\\0\end{bmatrix}
=\begin{bmatrix} \overline{z} \theta_1'\\\overline{z}b
\end{bmatrix},
$$
so that
$$
H_{\Phi_-^*}\begin{bmatrix}\theta_1'\\0\end{bmatrix}
=\begin{bmatrix}\theta_1'(0)\\b(0)\end{bmatrix}.
$$
We thus have
$$
\widetilde{\Phi_-^*} \begin{bmatrix}\theta_1'(0)\\b(0)
\end{bmatrix}=
\begin{bmatrix} \overline{z}&\overline{\widetilde{\theta}}_1\widetilde{b}\\
\overline{\widetilde{\theta}}_0\widetilde{a}& \overline{z}
\end{bmatrix}\begin{bmatrix}\theta_1'(0)\\b(0)
\end{bmatrix}=\begin{bmatrix} *\\\theta_1'(0)\overline{\widetilde{\theta}}_0
\widetilde{a}+ b(0)\overline{z}\end{bmatrix},
$$
so that
$$
H_{\Phi_-^*}^*H_{\Phi_-^*}\begin{bmatrix} \theta_1'\\0 \end{bmatrix}=\begin{bmatrix}*\\
\theta_1'(0)J(I-P)(\overline{\widetilde{\theta}}_0
\widetilde{a})+b(0)
\end{bmatrix}.
$$
A similar calculation shows that
$$
H_{\Phi_+^*}^*H_{\Phi_+^*}\begin{bmatrix} \theta_1^{\prime}\\0
\end{bmatrix}=\begin{bmatrix}*\\0\end{bmatrix}.
$$
It thus follows that
$$
\theta_1'(0)J(I-P)(\overline{\widetilde{\theta}}_0\widetilde{a})=-b(0).
$$
Since $b(0)\ne 0$, we have that
$\overline{\widetilde{\theta}}_0\widetilde{a}\in \overline z H^2$,
which implies that $\overline{\theta_0} a=\alpha \overline{z}$ for a
nonzero constant $\alpha$. \  Therefore we must have that
$\theta_0^{\prime}$ is a constant. \  Similarly, we can show that
$\overline{\theta_1}b=\beta\overline z$ for a nonzero constant
$\beta$, and hence $\theta_1^{\prime}$ is also a constant. \
Therefore by (\ref{6.16}), $T_\Phi$ is normal. \ Now observe that
$$
\Phi_+=\begin{bmatrix}0&\varphi_+\\ \psi_+&0
\end{bmatrix}\quad\hbox{and}\quad
\Phi_-^*=\begin{bmatrix}\overline{z}&\alpha\overline{z}\\
\beta\overline{z}&\overline{z}
\end{bmatrix}\qquad (\alpha\ne 0\ne \beta).
$$
Since $T_\Phi$ is normal we have
$$
\begin{bmatrix}
H_{\overline{\varphi_+}}^*H_{\overline{\varphi_+}}&0\\0&H_{\overline{\psi_+}}^*H_{\overline{\psi_+}}
\end{bmatrix}=\begin{bmatrix}
(1+|\beta|^2)H_{\overline{z}}&(\alpha+\overline{\beta})H_{\overline{z}}\\
(\overline{\alpha}+\beta)H_{\overline{z}}&(1+|\alpha|^2)H_{\overline{z}}\end{bmatrix},
$$
which implies that
\begin{equation}\label{6.17}
\begin{cases}
\beta=-\overline{\alpha}\\
H_{\overline{\varphi_+}}^*H_{\overline{\varphi_+}}=(1+|\beta|^2)H_{\overline{z}}\\
H_{\overline{\psi_+}}^*H_{\overline{\psi_+}}=(1+|\alpha|^2)H_{\overline{z}}.
\end{cases}
\end{equation}
By the case assumption, $1\ne |ab|=|\alpha\beta|=|\alpha|^2$, i.e.,
$|\alpha|\ne 1$. \  By the same argument as in (\ref{6.13}) we have
$$
\varphi_+=e^{i\theta_1}\sqrt{1+|\alpha|^2}\,z+\beta_1\quad\hbox{and}\quad
\psi_+=e^{i\theta_2}\sqrt{1+|\alpha|^2}\,z+\beta_2,
$$
($\beta_1,\beta_2\in\mathbb C;\ \theta_1,\theta_2\in [0,2\pi)$)
which implies that
$$
\varphi=\alpha\, \overline{z}+e^{i
\theta_1}\sqrt{1+|\alpha|^2}\,z+\beta_1\quad\hbox{and}\quad
\psi=-\overline{\alpha}\, \overline{z}+ e^{i
\theta_2}\sqrt{1+|\alpha|^2}\,z+\beta_2.
$$
Since $|\varphi|=|\psi|$, it follows that
$$
\left| e^{i\theta_1}\sqrt{1+|\alpha|^2}\,z^2+ \beta_1
z+\alpha\right| = \left| e^{i
\theta_2}\sqrt{1+|\alpha|^2}\,z^2+\beta_2 z -
\overline\alpha\right|\quad \hbox{for all $z$ on $\mathbb T$.}
$$
We argue that if $p$ and $q$ are polynomials having the same degree
and the outer coefficients of the same modulus then
$$
|p(z)|=|q(z)|\ \ \hbox{on $|z|=1$}\ \Longrightarrow\
p(z)=e^{i\omega}q(z)\quad\hbox{for some $\omega\in [0,2\pi)$}.
$$
Indeed, if $|p(z)|=|q(z)|$ on $|z|=1$, then $p=\theta q$ for a
finite Blaschke product $\theta$, i.e., $p=\prod_{j=1}^n
\frac{z-\alpha_j}{1-\overline{\alpha}_j z} q$ ($|\alpha_j|\le 1$). \
But since the modulus of the outer coefficients are same, it follows
that $\prod_{j=1}^n |\alpha_j|=1$ and therefore, $p=e^{i\omega}q$
for some $\omega$. \  Using this fact we can see that
$\psi=e^{i\omega} \varphi$ for some $\omega\in [0,2\pi)$. \ But then
a straightforward calculation shows that
$\omega=\pi-2\,\hbox{arg}\,\alpha$, and hence
$$
\varphi=\alpha\, \overline z + e^{i
\theta}\sqrt{1+|\alpha|^2}\,z+\beta\quad\hbox{and}\quad
\psi=e^{i\,(\pi-2\,{\rm arg}\,\alpha)}\varphi,
$$
where $\alpha\ne 0,\ |\alpha|\ne 1,\ \beta\in\mathbb C$, and
$\theta\in [0,2\pi)$. \

\medskip

{\bf Case B-2}
$\left((ab)(0)=(\theta_0^{\prime}\theta_1^{\prime})(0)\right)$.
\quad A similar argument as in Case 1 shows that $\theta_2$ and
$\theta_3$ are constant and the same argument as in Case B-1 gives
that
$$
\varphi=\alpha\, \overline z + e^{i
\theta}\sqrt{1+|\alpha|^2}\,z+\beta\quad\hbox{and}\quad
\psi=e^{i\,(\pi-2\,{\rm arg}\,\alpha)}\varphi,
$$
where $\alpha\ne 0,\ |\alpha|=1,\ \beta\in\mathbb C,$ and $\theta\in
[0,2\pi)$. \  We here note that the condition $|\alpha|=1$ comes
from the case assumption
$1=|\theta_0^{\prime}\theta_1^{\prime}|=|ab|=|\alpha|^2$. \

\medskip

Therefore if we combine the two subcases of Case B-1 and B-2 then we
can conclude that
\begin{equation}\label{6.18}
\varphi=\alpha\, \overline z + e^{i
\theta}\sqrt{1+|\alpha|^2}\,z+\beta\quad\hbox{and}\quad
\psi=e^{i\,(\pi-2\,{\rm arg}\,\alpha)}\varphi,
\end{equation}
where $\alpha\ne 0,\ \beta\in\mathbb C$, and $\theta\in [0,2\pi)$. \
This completes the proof.
\end{proof}

\begin{remark}\label{remark6.7}
We would also ask whether there is a subnormal {\it non}-Toeplitz
completion of $\left[\begin{smallmatrix} T_{\overline z}& ?\\
?&T_{\overline z}\end{smallmatrix}\right]$. \  Unexpectedly, there
is a normal non-Toeplitz completion of $\left[\begin{smallmatrix}
T_{\overline z}& ?\\ ?&T_{\overline z}\end{smallmatrix}\right]$. \
To see this, let $B$ be a selfadjoint operator and put
$$
T=\begin{bmatrix} T_{\overline z} &T_z+B\\T_z+B&T_{\overline
z}\end{bmatrix}.
$$
Then
$$
[T^*, T]
=\begin{bmatrix}T_{\overline{z}}B+BT_z-(T_zB+BT_{\overline{z}})
&T_zB+BT_{\overline{z}}-(T_{\overline{z}}B+BT_z)\\
BT_{\overline z}+T_zB - (BT_z+T_{\overline z}B)
&T_{\overline{z}}B+BT_z-(T_zB+BT_{\overline{z}})
\end{bmatrix},
$$
so that $T$ is normal if and only if
\begin{equation}\label{6.19}
T_{\overline{z}}B+BT_z=T_zB+BT_{\overline{z}},\ \ \hbox{i.e.,}\
[T_z,B]=[T_{\overline z},B].
\end{equation}
We define
$$
\alpha_1:=0\quad\hbox{and}\quad
\alpha_n:=-\frac{2}{3}\left(1-\left(-\frac{1}{2}\right)^n\right)\quad
\hbox{for}\ n\ge 2.
$$
Let $D\equiv \hbox{diag}\,(\alpha_n)$, i.e., a diagonal operator
whose diagonal entries are $\alpha_n$ ($n=1,2,\hdots$) and for each
$n=1,2\hdots$, let $B_n$ be defined by
$$
B_n=-\frac{1}{2^{n-1}}\hbox{diag}(\alpha_{n-1})T_{z^{2n}}^*\,.
$$
Then
$$
||B_n||\leq
\frac{1}{2^{n-1}}\hbox{sup}\{\alpha_{n-1}\}<\frac{1}{2^{n-1}},
$$
which implies that
$$
\Biggl| \Biggl|\sum_{n=1}^{\infty}B_n\Biggr|\Biggr| \leq 2.
$$
We define $C$ by
$$
C:=\sum_{n=1}^{\infty}B_n.
$$
Then $C$ looks like:
$$
C=\begin{bmatrix}
0&0&1&0&\frac{1}{2}&0&\frac{1}{2^2}&0&\cdots\\
0&0&0&\frac{1}{2}&0&\frac{1}{2^2}&0&\frac{1}{2^3}&\cdots\\
0&0&0&0&\frac{3}{2^2}&0&\frac{3}{2^3}&0&\cdots\\
0&0&0&0&0&\frac{5}{2^3}&0&\frac{5}{2^4}&\cdots\\
0&0&0&0&0&0&\frac{11}{2^4}&0&\cdots\\
0&0&0&0&0&0&0&\frac{21}{2^5}&\cdots\\
0&0&0&0&0&0&0&0&\cdots\\
0&0&0&0&0&0&0&0&\cdots\\
0&0&0&0&0&0&0&0&\cdots\\
\vdots&\vdots&\vdots&\vdots&\vdots&\vdots&\vdots&\vdots&\ddots\\
\end{bmatrix}.
$$
Note that $C$ is bounded. \  If we define $B$ by
$$
B:=D+C+C^*,
$$
then a straightforward calculation shows that $B$ satisfies equation
(\ref{6.19}). \  Therefore the operator
$$
T=\begin{bmatrix} T_{\overline z}&T_z+B\\T_z+B&T_{\overline
z}\end{bmatrix}
$$
is normal. \  We note that $T_z+B$ is not a Toeplitz operator. \qed
\end{remark}

\bigskip

\begin{remark} \label{rem6.7}
In Theorem \ref{thm6.1} we have seen that a $2$-hyponormal Toeplitz
completion of $\left[\begin{smallmatrix} T_{\overline z}& ?\\
?&T_{\overline z}\end{smallmatrix}\right]$ is automatically normal.
\ Consequently, from the viewpoint of $k$-hyponormality as a bridge
between hyponormality and subnormality, there is no gap between the
$2$-hyponormality and the subnormality of $\left[\begin{smallmatrix}
T_{\overline z}& T_\varphi\\ T_\psi&T_{\overline
z}\end{smallmatrix}\right]$ ($\varphi,\psi\in H^2$). \ Of course
there does exist a gap between the hyponormality and the
2-hyponormality of $\left[\begin{smallmatrix} T_{\overline z}&
T_\varphi\\ T_\psi&T_{\overline z}\end{smallmatrix}\right]$. \ To
see this, let
$$
\Phi:=\left[\begin{matrix} \overline z& \overline{z}^2+2z^2\\
\overline{z}^2+2z^2&\overline z\end{matrix}\right].
$$
Then $\Phi$ is normal and if we put $K:=
\left[\begin{smallmatrix} \frac{1}{2}& \frac{z}{2}\\
\frac{z}{2}&\frac{1}{2}
\end{smallmatrix}\right]$, then $\Phi-K\,\Phi^*\in H^2_{M_2}$ and
$||K||_\infty=1$, so that $T_\Phi$ is hyponormal. \
But by Theorem
\ref{thm6.1}, $T_\Phi$ is not 2-hyponormal. \
However,
we have not been able to characterize all hyponormal completions of
$\left[\begin{smallmatrix} T_{\overline z}& ?\\ ?&T_{\overline
z}\end{smallmatrix}\right]$\,; this completion problem appears to be quite difficult. \qed
\end{remark}

\vskip 1cm

%
%
%
%

\section{Open Problems}

\noindent {\bf 1. Nakazi-Takahashi's Theorem for matrix-valued
symbols.}\quad T. Nakazi and K. Takahashi \cite{NT} have shown that
if $\varphi\in L^\infty$ is such that $T_\varphi$ is a hyponormal
operator whose self-commutator $[T_\varphi^*, T_\varphi]$ is of
finite rank then there exists a finite Blaschke product
$b\in\mathcal{E}(\varphi)$ such that
$$
\hbox{deg}\,(b)=\hbox{rank}\,[T_\varphi^*,T_\varphi].
$$
What is the matrix-valued version of Nakazi and Takahashi's Theorem\,? \ A
candidate is as follows: If $\Phi\in L^\infty_{M_n}$ is such that
$T_\Phi$ is a hyponormal operator whose self-commutator $[T_\Phi^*,
T_\Phi]$ is of finite rank then there exists a finite
Blaschke-Potapov product $B\in \mathcal{E}(\Phi)$ such that
$\hbox{deg}\,(B)=\hbox{rank}\,[T_\Phi^*,T_\Phi]$. \  We note that
the degree of the finite Blaschke-Potapov product $B$ is defined by
\begin{equation}\label{7.1}
\hbox{deg}\,(B):=\hbox{dim}\,\mathcal{H}(B)=\hbox{deg}\,(\hbox{det}\,B)\,,
\end{equation}
where the second equality follows from
the well-known Fredholm theory of block Toeplitz
operators \cite{Do2} that
$$
\begin{aligned}
\hbox{dim}\,\mathcal{H}(\Theta)
&=\hbox{dim ker}\,T_{\Theta^*}=-\hbox{index}\, T_\Theta\\
&=-\hbox{index}\, T_{\hbox{det}\,\Theta}=\hbox{dim ker}\,T_{\overline{\hbox{det}\,\Theta}}\\
&=\hbox{dim}\, \Bigl( \mathcal{H}(\hbox{det}\,\Theta)\Bigr) =
\hbox{deg}\,\Bigl(\hbox{det}\,\Theta\Bigr).
\end{aligned}
$$
Thus we conjecture the following:

\medskip

\noindent {\bf Conjecture 6.1.} If $\Phi\in L^\infty_{M_n}$ is such
that $T_\Phi$ is a hyponormal operator whose self-commutator
$[T_\Phi^*, T_\Phi]$ is of finite rank then there exists a finite
Blaschke-Potapov product $B\in \mathcal{E}(\Phi)$ such that
$\hbox{rank}\,[T_\Phi^*,T_\Phi]=\hbox{deg}\,\Bigl(\hbox{det}\,B\Bigr)$.
\

\bigskip

On the other hand, in \cite{NT}, it was shown that if $\varphi\in
L^\infty$ is such that $T_\varphi$ is subnormal and
$\varphi=q\overline\varphi$, where $q$ is a finite Blaschke product
then $T_\varphi$ is normal or analytic. \  We now we pose its block
version:

\medskip

\noindent {\bf Problem 6.2.} If $\Phi\in L^\infty_{M_n}$ is such
that $T_\Phi$ is subnormal and $\Phi=B\Phi^*$, where $B$ is a a
finite Blaschke-Potapov product, does it follow that $T_\Phi$ is
normal or analytic ?

\bigskip

\noindent {\bf 2. Subnormality of block Toeplitz operators.}\quad In
Remark \ref{remark4.7} we have shown that if the ``coprime"
condition of Theorem \ref{thm4.5} is dropped, then Theorem
\ref{thm4.5} may fail. \  However we note that the example given in
Remark \ref{remark4.7} is a direct sum of a normal Toeplitz operator
and an analytic Toeplitz operator. \  Based on this observation, we
have:

\medskip

\noindent {\bf Problem 6.3.} Let $\Phi\in L^\infty_{M_n}$ be a
matrix-valued rational function. \  If $T_\Phi$ and $T_\Phi^2$ are
hyponormal, but $T_\Phi$ is neither normal nor analytic, does it
follow that $T_\Phi$ is of the form
$$
T_\Phi=\begin{bmatrix} T_A&0\\ 0&T_B\end{bmatrix}\quad \hbox{(where
$T_A$ is normal and $T_B$ is analytic)}\,?
$$

\bigskip

It is well-known that if $T\in\mathcal{B(H)}$ is subnormal then
$\hbox{ker}\,[T^*, T]$ is invariant under $T$. \  Thus we might be
tempted to guess that if the condition ``$T_\Phi$ and $T_\Phi^2$ are
hyponormal"is replaced by ``$T_\Phi$ is hyponormal and
$\hbox{ker}\,[T_\Phi^*, T_\Phi]$ is invariant under $T_\Phi$," then
the answer to Problem 6.3 is affirmative. \  But this is not the
case. \  Indeed, consider
$$
T_\Phi=\begin{bmatrix} 2U+U^*&U^*\\ U^*&2U+U^*
\end{bmatrix}\,.
$$
Then a straightforward calculation shows that $T_\Phi$ is hyponormal
and $\hbox{ker}\,[T_\Phi^*, T_\Phi]$ is invariant under $T_\Phi$,
but $T_\Phi$ is never normal (cf. \cite[Remark 3.9]{CHL}). \
However, if the condition ``$T_\Phi$ and $T_\Phi^2$ are hyponormal"
is strengthened to ``$T_\Phi$ is subnormal", what conclusion do you
draw\,?

\bigskip

\noindent {\bf 3. Subnormal completion problem.}\quad Theorem
\ref{thm6.1} provides the subnormal Toeplitz completion of
\begin{equation}\label{7.2}
\begin{bmatrix} U^*&?\\ ?&U^*\end{bmatrix}\quad\hbox{($U$ is the shift on $H^2$)}.
\end{equation}
Moreover Remark \ref{remark6.7} shows that there is a normal
non-Toeplitz completion of (\ref{7.2}). \  However we were unable to
find all subnormal completions of (\ref{7.2}). \

\bigskip

\noindent {\bf Problem 6.4.} Let $U$ be the shift on $H^2$. \
Complete the unspecified entries of the partial block matrix
$\left[\begin{smallmatrix} U^*& ?\\ ?&U^*\end{smallmatrix}\right]$
to make it subnormal. \

\bigskip

On the other hand, Theorem \ref{thm6.1} shows that the solution of
the
subnormal Toeplitz completion of $\left[\begin{smallmatrix} U^*& ?\\
?&U^*\end{smallmatrix}\right]$ consists of Toeplitz operators with
symbols which are both analytic or trigonometric polynomials of
degree 1. \  Hence we might expect that if the symbols of the
specified Toeplitz operators of (\ref{7.2}) are co-analytic
polynomials of degree two then the non-analytic solution of the
unspecified entries consists of trigonometric polynomials of degree
$\le 2$. \

\bigskip

More generally, we have:

\bigskip

\noindent {\bf Problem 6.5.} If $\Phi$ and $\Psi$ are co-analytic
polynomials of degree $n$, does it follow that the non-analytic
solution of the subnormal Toeplitz completion of the partial
Toeplitz matrix $\left[\begin{smallmatrix} T_\Phi& ?\\
?&T_\Psi\end{smallmatrix}\right]$ consists of Toeplitz operators
whose symbols are trigonometric polynomials of degree $\le n$\,?


\vskip 1cm

%
%
%
%

\bibliographystyle{amsalpha}

\end{document}